\documentclass[11pt]{article}
\usepackage[left=1in, right=1in, top=1in]{geometry}
\usepackage{xargs}
\usepackage[numbers]{natbib}
\usepackage{fancyhdr}
\usepackage{setspace}
\usepackage{lastpage}
\usepackage{upgreek}
\usepackage{amsmath,mathtools,amsthm,amsfonts,amssymb}	
\usepackage{amssymb,dsfont,bbm}	
\usepackage{xcolor}  
\usepackage{bm}
\usepackage{enumitem}
\usepackage{mathrsfs}

\setcitestyle{number}

\usepackage[colorlinks=true,breaklinks=true,bookmarks=true,urlcolor=blue,citecolor=blue,linkcolor=blue,bookmarksopen=false,draft=false]{hyperref}

\usepackage{aliascnt}
\usepackage{cleveref}

\def\bA{\bar{\mathbf{A}}}
\def\barb{\bar{\mathbf{b}}}

\def\thetas{\theta^\star}
\newcommand{\funcA}[1]{\funcAw(#1)}
\def\funcAw{\mathbf{A}}
\def\funcbw{\mathbf{b}}
\newcommand{\funcb}[1]{\funcbw(#1)}
\def\msz{\mathsf{Z}}
\def\Am{{\bf A}}
\def\bm{{\bf b}}
\def\funnoisew{\varepsilon}
\newcommand{\funcnoise}[1]{\funnoisew(#1)}
\def\prtheta{\bar{\theta}}
\def\qcond{\kappa_{P}}
\def\bconst{\mathfrak{b}}
\def\State{Z}

\newcommandx{\zmfuncA}[2][1=]{\tilde{\funcAw}^{#1}(#2)}
\newcommandx{\zmfuncAw}[1][1=]{\tilde{\funcAw}_{#1}}
\newcommandx{\zmfuncb}[2][1=]{\tilde{\funcbw}^{#1}(#2)}
\def\funnoisew{\varepsilon}
\def\supconsteps{\supnorm{\funnoisew}}
\def\mcf{\mathcal{F}}
\newcommand{\supnorm}[1]{\norm{ #1 }[\infty]}
\newcommandx{\norm}[2][2=]{\Vert#1 \Vert_{{#2}}}
\def\smallAconst{\smallConst{\funcAw}}
\newcommand{\smallConst}[1]{\operatorname{c}_{{#1}}}

\def\ProdBa{\ProdB^{(\gamma)}}
\def\ProdB{\Gamma}

\newcommandx{\vartconstwas}[1][1=V]{c_{#1}}
\newcommandx{\ratewas}[1][1=]{\varrho_{#1}}
\def\trace{\operatorname{Tr}}

\newcommandx{\normop}[2][2=]{\Vert{#1}\Vert_{{#2}}}

\newcommand{\ConstD}{\mathsf{D}}
\newcommand{\bConst}[1]{\operatorname{C}_{{\bf #1}}}
\newcommand{\barf}{\bar{f}}
\newcommand{\MKQ}{\mathrm{Q}}
\newcommand{\PP}{\mathbb{P}}

\def\Id{\mathrm{I}}
\def\rmd{\mathrm{d}}
\def\rme{\mathrm{e}}
\def\rset{\mathbb{R}}

\def\Bound{\operatorname{T}}
\def\WBound{\operatorname{G}}

\def\ltwo{\mathrm{L}^2}

\newcommand{\continuation}{??}

\def\eqsp{\,}

\newcommand{\indi}[1]{\mathbbm{1}_{#1}}
\newcommand{\indiacc}[1]{\mathbbm{1}_{\{#1\}}}

\def\PE{\mathsf{E}}
\def\PVar{\mathsf{Var}}

\newcommandx{\CPE}[3][1=]{\mathsf{E}_{#1}^{#3}\bigl[  #2 \bigr]}

\def\B{\mathsf{B}}

\def\nset{\mathbb{N}}
\def\nsets{\mathbb{N}^{\star}}
\def\rset{\mathbb{R}}

\def\tmix{{\boldsymbol{\tau}}}

\newcommand{\tvnorm}[1]{\| #1 \|_{\operatorname{TV}}}
\newcommand{\infnorm}[1]{\| #1 \|_{\operatorname{\infty}}}

\newcommand{\Vnorm}[2][1=V]{\| #2 \|_{#1}}

\DeclareMathAlphabet\mathbfcal{OMS}{cmsy}{b}{n} 

\newtheorem{theorem}{Theorem}
\crefname{theorem}{theorem}{Theorems}
\Crefname{Theorem}{Theorem}{Theorems}

\newtheorem{lemma}{Lemma}
\crefname{lemma}{lemma}{lemmas}
\Crefname{Lemma}{Lemma}{Lemmas}

\newtheorem{corollary}{Corollary}
\crefname{corollary}{corollary}{corollaries}
\Crefname{Corollary}{Corollary}{Corollaries}

\newaliascnt{proposition}{theorem}
\newtheorem{proposition}[proposition]{Proposition}
\aliascntresetthe{proposition}
\crefname{proposition}{proposition}{propositions}
\Crefname{Proposition}{Proposition}{Propositions}

\newaliascnt{definition}{theorem}

\aliascntresetthe{definition}
\crefname{definition}{definition}{definitions}
\Crefname{Definition}{Definition}{Definitions}

\newaliascnt{definitionProposition}{theorem}

\aliascntresetthe{definitionProposition}
\crefname{Proposition and Definition}{Proposition and Definition}{Proposition and Definition}
\Crefname{Proposition and Definition}{Proposition and Definition}{Proposition and Definition}

\crefname{remark}{remark}{remarks}
\Crefname{Remark}{Remark}{Remarks}

\crefname{example}{example}{examples}
\Crefname{Example}{Example}{Examples}

\crefname{figure}{figure}{figures}
\Crefname{Figure}{Figure}{Figures}

\newtheorem{assumption}{\textbf{A}\hspace{-3pt}}
\Crefname{assumption}{\textbf{A}\hspace{-3pt}}{\textbf{A}\hspace{-3pt}}
\crefname{assumption}{\textbf{A}}{\textbf{A}}

\newtheorem{assumptionLSA}{\textbf{LSA}\hspace{-3pt}}
\Crefname{assumptionLSA}{\textbf{LSA}\hspace{-3pt}}{\textbf{L}\hspace{-3pt}}
\crefname{assumptionL}{\textbf{L}}{\textbf{L}}

\newtheorem{assumptionMC}{\textbf{MC}\hspace{-1pt}}
\Crefname{assumptionMC}{\textbf{MC}\hspace{-1pt}}{\textbf{C}\hspace{-1pt}}
\crefname{assumptionMC}{\textbf{MC}}{\textbf{MC}}

\newtheorem{assumptionC}{\textbf{C}\hspace{-1pt}}
\Crefname{assumptionC}{\textbf{C}\hspace{-1pt}}{\textbf{C}\hspace{-1pt}}
\crefname{assumptionC}{\textbf{C}}{\textbf{C}}

\Crefname{assumptionG}{\textbf{G}\hspace{-3pt}}{\textbf{G}\hspace{-3pt}}
\crefname{assumptionG}{\textbf{G}}{\textbf{G}}

\def\Zset{\mathsf{Z}}
\def\Zsigma{\mathcal{Z}}

\newcommand{\PEcoupling}[2]{\tilde{\mathbb{E}}_{#1,#2}}
\newcommand{\PPcoupling}[2]{\tilde{\mathbb{P}}_{#1,#2}}
\def\metricz{\mathsf{d}_{\Zset}}

\newcommand{\Wass}[1]{\mathbf{W}_{{#1}}}

\def\Lclass{\mathcal L}
\def\lyapW{W}

\def\deltawas{\delta_*}
\def\vartconstwas{c_{\MKK}}

\def\boundmetric{\kappa_{\MKK}}

\def\ratewas{\varrho}

\def\mrl\mathrm{L}

\def\MK{{\rm Q}}
\def\MKK{{\rm K}}

\def\PE{\mathrm{E}}
\def\msa{\mathsf{A}}

\def\metricc{c}
\def\minorwas{\epsilon}

\newcommand{\CKset}{\bar{\mathsf{C}}}
\newcommand{\Nnorm}[2][1=V]{[ #2]_{#1}}

\newcommand{\lr}[1]{\left( #1 \right)}
\newcommand{\lrb}[1]{\left[ #1 \right]}

\def\rootwas{\zeta}
\def\minorwas{\varepsilon}

\def\Sstat{S}
\newcommand{\ConstA}{\mathsf{A}}
\def\covgammatheta{\bar{\Sigma}_{\gamma}}
\def\couplingmeasure{\mathcal{C}}
\def\cost{\mathsf{c}}
\def\mrl{\mathrm{L}}
\newcommand{\abs}[1]{\left\vert #1 \right\vert}
\def\qbf{\mathbf{q}}

\begin{document}

\title{Rosenthal-type inequalities for linear statistics of Markov chains}

\author{A. Durmus~\footnote{Ecole Polytechnique, Paris, France,  \texttt{alain.durmus@polytechnique.edu}. }, \, E. Moulines~\footnote{Ecole Polytechnique, Paris, France, and Mohamed bin Zayed University of Artificial Intelligence (MBZUAI), UAE,  \texttt{eric.moulines@polytechnique.edu}.}, \, A. Naumov~\footnote{HSE University, Moscow, Russia,  \texttt{anaumov@hse.ru}.}, \, S. Samsonov~\footnote{HSE University, Moscow, Russia,  \texttt{svsamsonov@hse.ru}.}, \, M.Sheshukova~\footnote{HSE University, Moscow, Russia,  \texttt{msheshukova@hse.ru}.}}

\maketitle

\begin{abstract}
In this paper, we establish novel concentration inequalities for additive functionals of geometrically ergodic Markov chains similar to  Rosenthal inequalities for sums of independent random variables. We pay special attention to the dependence of our bounds on the mixing time of the corresponding chain. Precisely, we establish explicit bounds that are linked to the constants from the martingale version of the Rosenthal inequality, as well as the constants that characterize the mixing properties of the underlying Markov kernel. Finally, our proof technique is, up to our knowledge, new and is based on a recurrent application of the Poisson decomposition.
\end{abstract}

\section{Introduction}
\label{sec:intro}
Let $\{Y_\ell\}_{\ell=0}^{n-1}$ be a sequence of independent centered random variables. The Rosenthal inequality provides the following moment bound  for $\sum\nolimits_{\ell=0}^{n-1} Y_{\ell}$ (see \cite{Rosenthal1970,pinelis_1994}): for any $p \geq 2$,
\begin{equation}
\label{eq:ros_independent_pinelis}
\textstyle \PE^{1/p}\left[\left|\sum_{\ell=0}^{n-1} Y_{\ell}\right|^{p}\right] 
\le C \left\{ p^{1/2} \PVar^{1/2}\left(\sum_{\ell=0}^{n-1} Y_{\ell}\right) 
+  p \PE^{1/p}\left[\max_{\ell \in \{0,\dots,n-1\}}\left|Y_{\ell}\right|^p\right] \right\} \eqsp,
\end{equation}
where $C$ is a universal constant. 
Following prior works such as \cite{clemenccon2001moment,doukhan:louhichi:1999,dedecker2015deviation,dedecker2019deviation,doukhan2007probability,durmus2024probability}, the main goal of this paper is to derive bounds analogous to \eqref{eq:ros_independent_pinelis} for linear statistics of a Markov chain $(X_{\ell})_{\ell\in\nset}$ on a measurable space $(\Zset,\Zsigma)$. Namely, we consider  sums of the form
\begin{equation}
\label{eq:def_S_n_g_main}
S_n = \sum_{\ell=0}^{n-1}\{f(X_{\ell}) - \pi(f)\} \eqsp, \qquad n \geq 1 \eqsp, 
\end{equation}
where $(X_\ell)_{\ell \in \nset}$ is a Markov chain with Markov kernel $\MK$, which has a unique invariant distribution $\pi$, and $f: \Zset \to \rset$ is a $\pi$-integrable measurable function. We assume that $X_0$ follows some distribution $\xi$, which might be different from $\pi$. We aim to recover a counterpart of the Rosenthal inequality for $S_n$, with particular attention to the dependence of the resulting bounds on the mixing time associated with the Markov kernel $\MK$ and the asymptotic variance of $(S_n)_{n\in\nset}$. We establish our results under three different geometric ergodicity conditions on the Markov kernel $\MK$. Specifically, we consider Markov kernels $\MK$ such that their iterates $\xi \MK^{n}$ converge to the stationary distribution $\pi$ at an exponential rate, either in the total variation distance, the $V$-total variation norm (weighted total variation distance), or in the Wasserstein semi-metric associated with a general cost function. The latter setting is of particular interest, as it covers the case of non-irreducible Markov kernels, which naturally arise in applications such as stochastic approximation algorithms.
\par 
Under each of the aforementioned conditions, we establish bounds with explicit constants and explicit dependence on the parameters of the underlying Markov chain. These results can be summarized as follows: for any initial distribution $\xi$, $n \in \nset$, and $p \geq 2$,  
\begin{equation}
\label{eq:rosenthla_mc_intro}
 \PE_{\xi}^{1/p}\bigl[\bigl| \Sstat_n \bigr|^{p}\bigr]  \leq  C\left(p^{1/2}n^{1/2}\sigma_{\pi}(f) +
     \ConstA_{1,\pi,f} \, p^{\beta} n^{1/4}\tmix^{3/4} + \ConstA_{2,\pi,f,\xi}\, p^{\beta} \tmix   u_n \right) \eqsp,
\end{equation}
where $\sigma_{\pi}(f)$ is the asymptotic variance associated with $(S_n)_{n \in \nset}$ (see its precise definition below in \eqref{eq:asympt_var_def}), $C$ is a universal constant, $\ConstA_{1,\pi,f}$ (resp. $\ConstA_{2,\pi,f,\xi}$) are constants depending only on $\pi,f$ (resp. $\pi,f,\xi$), $\beta \geq 0$, and $(u_n)_{n \in \nset}$ is a sequence depending only on $n$ and satisfying $u_n = O(n^{1/(2p)})$. Under additional conditions on $f$, we improve our results, showing that $u_n = \log^{\gamma+1}(n)$ for $\gamma \geq 0$. As a consequence, optimizing over the exponent $p$, we obtain Bernstein-type inequalities for $(S_n)_{n\in\nset}$, extending and completing the literature on this subject \cite{merlevede2011bernstein,paulin2015concentration,dedecker2015deviation,jiang2018bernstein,doukhan2007probability,adamczak2008tail,adamczak2015exponential,lemanczyk2020general}.
\par 
The main challenge we address in our bounds is to provide bounds of the form \eqref{eq:rosenthla_mc_intro} with a sharp leading term, that is, the one driven by the asymptotic variance $\sigma_{\pi}(f)$. In particular, the coefficient in front of the leading (in $n$) term $p^{1/2} \sigma_{\pi}(f)$ matches the $p$-th moment of a limiting Gaussian distribution of the normalized sequence $(n^{-1/2}S_n)_{n \in \mathbb{N}}$. This result contrasts with prior approaches, which either express the bounds in terms of $\PVar_{\pi}(S_n)$ instead of $\sigma_{\pi}(f)$ \cite{doukhan2007probability,durmus2024probability}, or rely on various variance proxies \cite{dedecker2015deviation,jiang2018bernstein,bertail2018new,dedecker2019deviation}, or impose additional assumptions such as strong aperiodicity of the underlying Markov kernel or bounded functions \cite{adamczak2015exponential,lemanczyk2020general}. 
\par 
Furthermore, with a few exceptions such as \cite{paulin2015concentration,jiang2018bernstein}, most existing works \cite{dedecker2015deviation,adamczak2015exponential,bertail2018new,merlevede2011bernstein,lemanczyk2020general} do not make the dependence on the mixing time $\tmix$ explicit for the higher-order terms. When this dependence is tracked, it is typically inhomogeneous in $\tmix$ and $n$. Formally, this means that achieving a precision of $\varepsilon > 0$ for $\PE^{1/p}[|\Sstat_n/n|^{p}] \leq \varepsilon$ requires that $n$ scales as $\varepsilon^{-1} \tmix^{\alpha}$ with $\alpha > 1$, or even exhibits an exponential dependence of $n$ upon $\tmix$. In contrast, our bounds are homogeneous in $\tmix$ and $n$ (up to logarithmic factors). This refined dependence is important for applications in Markov chain Monte Carlo (MCMC) methods and stochastic approximation.
\par 
Our proof method is, to the best of our knowledge, new and is based on the Poisson equation decomposition associated with $\MKQ$; see \cite[Chapter 21]{douc:moulines:priouret:soulier:2018}. This technique allows us to relate the Rosenthal inequality for Markov chains to its counterpart for martingale-difference sequences, provided e.g. in \cite{pinelis_1994}. The main novelty of our approach lies in the recursive application of the Rosenthal inequality with different moment values $p$, a crucial factor in obtaining an explicit and homogeneous dependence on the mixing time of the Markov chain. Importantly, our method relies only on certain properties of the solution to the Poisson equation and can be generalized beyond geometrically ergodic Markov kernels.
\par
The structure of the paper is as follows. We present our main results in \Cref{sec:main-results}. First, we consider uniformly geometrically ergodic (UGE) Markov kernels in \Cref{sec:uge-chains}, and then we generalize our results in \Cref{sec:geom-v-ergod} and \Cref{sec:WGE} with geometric rates of convergence in the $V$-total variation norm and the Wasserstein semi-metric, respectively. In \Cref{sec:discussion_litt}, we provide a detailed literature review and compare our results with existing counterparts. In \Cref{sec:apps}, we discuss applications of the developed results to MCMC methods and linear stochastic approximation. The main proofs are collected in \Cref{sec:proofs}. At the beginning of \Cref{sec:proofs}, we provide a brief outline of our proof technique for Markov kernels under UGE. Additional technical lemmas are given in \Cref{appB}. For the reader's convenience, we have compiled all constants from the technical lemmas and main results in \Cref{sec:constants}. Some technical proofs are provided in the supplementary material.


\par 
\noindent\textbf{Notation and definitions.}
\label{sec:notations}
 Let $(\Zset,\Zsigma)$ be a measurable space. If $(\Zset,\metricz)$ is a complete separable metric space, then $\Zsigma$ is taken to be its Borel $\sigma$-field. Denote by $(\Zset^{\nset},\Zsigma^{\otimes \nset})$ be the corresponding canonical space. Consider the Markov kernel $\MKQ$ defined on $\Zset \times \Zsigma$ and denote by $\PP_{\xi}$ and $\PE_{\xi}$ the corresponding probability distribution and expectation with initial distribution $\xi$. Without loss of generality, we assume that $(Z_k)_{k \in \nset}$ is the corresponding canonical process. By construction, for any $\msa \in \Zsigma$, it holds that $\PP_{\xi}(Z_k \in \msa | Z_{k-1}) = \MKQ(Z_{k-1},\msa)$, $\PP_{\xi}$-a.s. If $\xi = \delta_{z}$, $z \in \Zset$, we write $\PP_{z}$ and $\PE_{z}$ instead of $\PP_{\delta_{z}}$ and $\PE_{\delta_{z}}$, respectively. We denote $\mcf_{i} = \sigma(Z_j, j \leq i)$.
 \par 

 For a (signed) measure $\mu$ on $(\Zset,\Zsigma)$ and a measurable function $f: \Zset \to \rset$, we denote by $\mu(f) = \int\nolimits_{\Zset}f(z)\mu(\rmd z)$. For a measurable function $V: \Zset \to [1,\infty)$, we define $\mrl_V$ as a set of all measurable functions $f: \Zset \to \rset$, such that $\Vnorm[V]{f}= \sup_{z \in \Zset} \{|f(z)|/V(z)\} < \infty$.
The $V$-norm (also referred to as $V$-total variation norm) of a signed measure $\mu$ is defined as $\Vnorm[V]{\mu} = \int\nolimits_{\Zset} V(z) \abs{\mu}(\rmd z)$,
where $\abs{\mu}$ is the total variation of $\mu$. When $V \equiv 1$, the $V$-norm is the total variation norm and is denoted by $\tvnorm{\cdot}$. Equivalently, we can define $\Vnorm[V]{\mu} = \sup\{ \mu(f) \, :\, \Vnorm[V]{f}\leq 1\}$ (see \cite[Theorem D.3.2]{douc:moulines:priouret:soulier:2018} for details).
\par 
For a vector $x \in \rset^{d}$ and matrix $\mathbf{A} \in \rset^{d \times d}$ we write $\norm{x}$ and $\norm{\mathbf{A}}$ for their Euclidean and spectral norm, respectively. We write $\sigma_{\min}(\mathbf{A})$ for its minimal singular value. One of the goals of the paper is to provide results with explicit constants. To avoid burdening the theorem statements, we have compiled the constants in a table. The notation $\ConstD_{xx,y}$ means that the respective constant is defined in the Theorem, Proposition, Lemma, etc numbered $xx$.

\section{Main results}
\label{sec:main-results}
\subsection{Uniformly geometrically ergodic Markov kernels}
\label{sec:uge-chains}
In what follows, we first consider the setting of a uniformly geometrically ergodic Markov kernel $\MKQ$, before turning to more general ergodicity conditions. Specifically, we assume that:
\begin{assumption}
\label{ass:UGE}
The Markov kernel $\MKQ$ admits a unique invariant distribution $\pi$. Moreover, $\MKQ$ is uniformly geometrically ergodic, i.e., there exists $\tmix \in \nset$ such that for all $n \in \nset$ and $z \in \Zset$,
\begin{equation*}
(1/2)\tvnorm{\MKQ^{n}(z, \cdot) - \pi} \leq (1/4)^{\left\lfloor n/\tmix \right\rfloor}\eqsp.
\end{equation*}
\end{assumption}
This setting is well studied, and counterparts of Rosenthal's inequality \eqref{eq:ros_independent_pinelis} can be established under \Cref{ass:UGE} using various techniques; see  \Cref{sec:discussion_litt}. The parameter $\tmix$ in \Cref{ass:UGE} is commonly referred to as the \emph{mixing time} of $\MKQ$, following \cite[Definition~1.3]{paulin2015concentration}. In this section, we consider the statistics $\Sstat_n$ defined in \eqref{eq:def_S_n_g_main}, with $f : \Zset \mapsto \rset$ being a bounded measurable function. Throughout the paper, we set $\bar{f}(z) = f(z) - \pi(f)$, $z \in \Zset$. It is well known (see \cite[Corollary 5]{jones:2009} or \cite[Chapter~19]{ibragimov1971}) that under \Cref{ass:UGE} the central limit theorem (CLT) holds:  
\[
\sqrt{n} \Sstat_n \overset{\text{dist}}{\to} \mathcal{N}(0,\sigma^{2}_{\pi}(f))\eqsp, \text{ as $n\to \infty$}\eqsp, 
\]
where the asymptotic variance $\sigma^{2}_{\pi}(f)$ has the form  
\begin{equation}
\label{eq:asympt_var_def}
\sigma^{2}_{\pi}(f) = \pi(\bar{f}^2)  + 2\sum_{\ell = 1}^{\infty} \pi(\barf \MKQ^{\ell} \barf) \eqsp.
\end{equation}

\noindent Now we state the main result of this section, which is a version of \eqref{eq:rosenthla_mc_intro} with explicit constants.
\begin{theorem}
\label{theo:rosenthal}
Under \Cref{ass:UGE}, for any bounded measurable function $f: \Zset \rightarrow \rset$, and $p \geq 2$,
\begin{multline}
\label{eq:main_text_rosenthal_pi}
\PE_{\pi}^{1/p}[| \Sstat_n|^p] \leq \bConst{\sf{Rm}, 1} \sqrt{2}p^{1/2}n^{1/2}\sigma_{\pi}(f) \\ +2\ConstD_{\ref{lem:UGE_dyadic},1} n^{1/4}\tmix^{3/4}p\log_2(2p) \infnorm{f} + 2\ConstD_{\ref{lem:UGE_dyadic},2} \tmix p \log_2(2p) \infnorm{f}\eqsp,
\end{multline}
where $\bConst{\sf{Rm}, 1},\ConstD_{\ref{lem:UGE_dyadic},1}$ and $\ConstD_{\ref{lem:UGE_dyadic},2}$ are universal constants gathered in \Cref{tab:univ_constants} in \Cref{sec:constants}. Moreover, for any initial probability $\xi$ on $(\Zset,\Zsigma)$,
\begin{equation}
\label{eq:main_text_rosenthal_ksi}
\PE_{\xi}^{1/p}[| \Sstat_n|^p] \leq \PE_{\pi}^{1/p}[| \Sstat_n|^p] + 2 \ConstD_{\ref{lem:UGE_dyadic},2} \tmix p \log_2(2p) \infnorm{f}\eqsp.
\end{equation}
\end{theorem}
The proof is given in \Cref{sec:proof:theo:rosenthal}. Note that the constants in the bound \eqref{eq:main_text_rosenthal_pi}, namely $\bConst{\sf{Rm}, 1}$, $\ConstD_{\ref{lem:UGE_dyadic},1}$, and $\ConstD_{\ref{lem:UGE_dyadic},2}$, are universal (numeric) and depend only on the constants from Rosenthal's inequality for martingales given in \cite{pinelis_1994}. In the case of a bounded vector-valued function $f: \Zset \rightarrow \rset^{d}$, the counterpart of \Cref{theo:rosenthal} follows with $\sigma_{\pi}(f)$ replaced by $\sqrt{\trace{\Sigma_{\pi}(f)}}$, where
\begin{equation}
\label{eq:asympt_var_def_vector}
\Sigma_{\pi}(f) = \pi(\bar{f} \bar{f}^{\top}) + 2\sum_{\ell = 1}^{\infty} \pi(\barf (\MKQ^{\ell} \barf)^{\top}) \eqsp, 
\end{equation}
and $\infnorm{f} = \sup_{z \in \Zset}\norm{f(z)}$, with $\norm{\cdot}$ denoting the standard (Euclidean) norm on $\rset^{d}$.
\begin{theorem}
\label{theo:rosenthal_large_dimention}
Under \Cref{ass:UGE}, for any bounded measurable function $f: \Zset \rightarrow \rset^{d}$, and $p \geq 2$,
\begin{multline*}
\PE_{\pi}^{1/p}[\| \Sstat_n\|^p] \leq \bConst{\sf{Rm}, 1} \sqrt{2}p^{1/2}n^{1/2}\sqrt{\trace{\Sigma_{\pi}(f)}} \\ 
+ 2\ConstD_{\ref{lem:UGE_dyadic},1} n^{1/4}\tmix^{3/4}p\log_2(2p) \infnorm{f} + 2\ConstD_{\ref{lem:UGE_dyadic},2} \tmix p \log_2(2p) \infnorm{f}\eqsp.
\end{multline*}
Moreover, for any initial probability $\xi$ on $(\Zset,\Zsigma)$,
\begin{equation*}
\PE_{\xi}^{1/p}[\| \Sstat_n\|^p] \leq \PE_{\pi}^{1/p}[\| \Sstat_n\|^p] + 2 \ConstD_{\ref{lem:UGE_dyadic},2} \tmix p \log_2(2p) \infnorm{f}\eqsp.
\end{equation*}
\end{theorem}

The same remark applies to the main results of the subsequent sections (namely, \Cref{theo:rosenthal_VGE} and \Cref{theo:rosenthal_WGE_polynom}). Using Markov inequality (see details in \Cref{lem:deviation_bound}), \Cref{theo:rosenthal} straightforwardly implies a Bernstein-type inequality.

\begin{corollary}
\label{cor:uge_bound}
Assume \Cref{ass:UGE}.
For any $\delta \in (0,1/\rme^2)$, with $\PP_{\pi}$-probability at least $1-\delta$,
\begin{equation*}
| \Sstat_n| \leq \bConst{\sf{Rm}, 1} \sqrt{2 n \ln(1/\delta)}\sigma_{\pi}(f) + 
2\ConstD_{\ref{lem:UGE_dyadic},1} n^{1/4}\tmix^{3/4} \infnorm{f} \widetilde{\ln}(1/\delta)  + 2\ConstD_{\ref{lem:UGE_dyadic},2} \tmix \infnorm{f} \widetilde{\ln}(1/\delta) \eqsp,
\end{equation*}
where we set $\widetilde{\ln}(s) = \ln(s) \log_2(2 \ln(s))$.
\end{corollary}

The attractive feature of bounds provided in \Cref{theo:rosenthal} and \Cref{cor:uge_bound} is that, once divided by $n$, their right-hand side takes the form 
\[
\textstyle 
C_1\sigma_{\pi}(f)/n^{1/2} + C_2(\tmix/n)^{3/4} + C_3\tmix/n\eqsp,
\]
with constants $C_1,C_2,C_3 \geq 0$ independent of $\tmix$ and $n$. Since $\sigma_{\pi}(f) \lesssim \tmix^{1/2}$, this bound shows that the required number of summands $n$ grows only linearly with the mixing time $\tmix$ in order to guarantee $\PE_{\pi}^{1/p}[|\Sstat_n/n|^p] \leq \epsilon$ (or $|\Sstat_n/n| \leq \epsilon$ with high probability). 

\subsection{$V$-uniformly geometrically ergodic Markov kernels}
\label{sec:geom-v-ergod}
We now consider the case where the Markov kernel $\MKQ$ is $V$-uniformly geometrically ergodic. This setup generalizes the one considered in \Cref{sec:uge-chains} and covers many scenarios of interest; see, e.g., \cite{meyn:tweedie,douc:moulines:priouret:soulier:2018}. Let $V: \Zset \rightarrow [\rme; +\infty)$. Note that we require the function $V$ to take values in $[\rme;+\infty)$ rather than $[1,+\infty)$, as is common in the literature. This choice simplifies the presentation and the proof of our subsequent results, since we consider the class of functions $\mrl_{\log{V}}$. We make the following assumption:
\begin{assumption}
\label{ass:VGE}
The Markov kernel $\MKQ$ has a unique invariant distribution $\pi$, $\pi(V) < \infty$, and there exist $\tmix \in \nset$ and $\varkappa \geq 1$ such that for any $n \in \nset$,
\begin{equation}
\label{eq:VGE}
\sup_{z,z' \in \Zset} \frac{\Vnorm[V]{\delta_z \MKQ^n - \delta_{z'} \MKQ^n}}{V(z) + V(z')} \leq \varkappa \bigl(1/4\bigr)^{\lfloor n/\tmix \rfloor}\eqsp.
\end{equation}
\end{assumption}

There are numerous tools to verify \Cref{ass:VGE} and to compute mixing time $\tmix$. The classical approach is based on geometric drift assumptions \emph{à la} Foster--Lyapunov and minorization conditions for kernel iterates on "small" sets; see, e.g., \cite[Chapter~15--16]{meyn:tweedie}, \cite{roberts:rosenthal:2004}, and \cite[Chapter~19]{douc:moulines:priouret:soulier:2018}. \Cref{ass:VGE} implies that for any probability measure $\xi$ with $\xi(V) < \infty$ and $n \in \nset$,  
\[
\Vnorm[V]{\xi \MKQ^n - \pi} \leq \varkappa (\xi(V) + \pi(V)) \bigl(1/4\bigr)^{\lfloor n/\tmix \rfloor}\eqsp.
\]
Therefore, \Cref{ass:VGE} is a generalization of \Cref{ass:UGE}. In the following, we present an analogue of \Cref{theo:rosenthal} under this weaker assumption, at the expense of slightly degraded (in terms of $p$) and more complex bounds.

\begin{theorem}
\label{theo:rosenthal_VGE}
Assume \Cref{ass:VGE} and let $\beta \geq 0$. Then for any measurable function $f: \Zset \rightarrow \rset$ with $\Vnorm[\log^{\beta}{V}]{f} < \infty$, $p \geq 2$, and $n \geq 3$, it holds that
\begin{multline*}
 \PE_{\pi}^{1/p}\bigl[\bigl| \Sstat_n \bigr|^{p}\bigr]  \leq  \bConst{\sf{Rm}, 1}\sqrt{2}p^{1/2}n^{1/2}\sigma_{\pi}(f) + \ConstD^{(\beta)}_{\ref{theo:rosenthal_VGE},1}  \{\varkappa \pi(V)\}^{2/p} p^{2+\beta} \log_2(2p) n^{1/4}\tmix^{3/4}  \Vnorm[\log^{\beta}{V}]{f} \\
+ \ConstD^{(\beta)}_{\ref{theo:rosenthal_VGE},2} \{\varkappa \pi(V)\}^{1/p} p^{2+\beta} \log_2(2p) (\log{n})^{\beta+1} \tmix \Vnorm[\log^{\beta}{V}]{f}\eqsp,
\end{multline*}
where 
\begin{equation*}
\ConstD^{(\beta)}_{\ref{theo:rosenthal_VGE},1} = 4 \ConstD_{\ref{lem: rosenthal_VGE_for_2^s}, 1} (2\beta / \rme)^{\beta}\eqsp, \quad \ConstD^{(\beta)}_{\ref{theo:rosenthal_VGE},2} = 4 \ConstD_{\ref{lem: rosenthal_VGE_for_2^s}, 2} (2\beta / \rme)^{\beta} \eqsp,
\end{equation*}
and 
$\bConst{\sf{Rm}, 1}$, $\ConstD_{\ref{lem: rosenthal_VGE_for_2^s}, 1}$, $\ConstD_{\ref{lem: rosenthal_VGE_for_2^s}, 2}$ are universal constants gathered in \Cref{tab:univ_constants} in \Cref{sec:constants}. Moreover, for any
initial probability $\xi$ on $(\Zset, \Zsigma)$,
\begin{equation}
\label{eq:main_rosenthal_VGE_xi}
\PE_{\xi}^{1/p}\bigl[\bigl| \Sstat_n \bigr|^{p}\bigr]  \leq \PE_{\pi}^{1/p}\bigl[\bigl| \Sstat_n \bigr|^{p}\bigr]  + (32/15) (\beta / \rme)^{\beta} \varkappa^{1/p} \bigl(\xi(V) + \pi(V)\bigr)^{1/p} p^{\beta+1}\tmix \Vnorm[\log^{\beta}{V}]{f}\eqsp.
\end{equation}
\end{theorem}
The proof is given in \Cref{sec:proof_rosenthal_VGE}. Similarly to \Cref{theo:rosenthal}, the right-hand side of \Cref{theo:rosenthal_VGE}, after normalization by $n$, is a function of the ratio $\tmix/n$, up to logarithmic factors. Below we provide a version of \Cref{theo:rosenthal_VGE} to control $\PE_{\pi}^{1/p}\bigl[\bigl| \Sstat_n \bigr|^{p}\bigr]$, relaxing the growth condition on $f$ and assuming only that $\Vnorm[V^{1/2p}]{f} < \infty$ for some $p \geq 2$.
\begin{theorem}
\label{theo:rosenthal_VGE_polynom}
Assume \Cref{ass:VGE} and let $p \geq 2$. Then for any measurable function $f: \Zset \rightarrow \rset$ with $\Vnorm[V^{1/2p}]{f} < \infty$, and $n \geq 3$, it holds that
\begin{multline*}
 \PE_{\pi}^{1/p}\bigl[\bigl| \Sstat_n \bigr|^{p}\bigr]  \leq  \bConst{\sf{Rm}, 1}\sqrt{2}p^{1/2}n^{1/2}\sigma_{\pi}(f) + 4 \ConstD_{\ref{lem: rosenthal_VGE_for_2^s}, 1} \{\varkappa \pi(V)\}^{2/p} p^2 \log_2(2p) n^{1/4}\tmix^{3/4}\Vnorm[V^{1/2p}]{f}  \\
    + 4 \ConstD_{\ref{lem: rosenthal_VGE_for_2^s}, 2} \{\varkappa\pi(V)\}^{1/p} p^2 \log_2(2p) n^{1/p}\tmix \Vnorm[V^{1/2p}]{f}\eqsp.
\end{multline*}
Moreover, for any
initial probability $\xi$ on $(\Zset, \Zsigma)$,
\begin{equation*}
\PE_{\xi}^{1/p}\bigl[\bigl| \Sstat_n \bigr|^{p}\bigr]  \leq \PE_{\pi}^{1/p}\bigl[\bigl| \Sstat_n \bigr|^{p}\bigr]  + (32/15) \varkappa^{1/p}\bigl(\xi(V) + \pi(V)\bigr)^{1/p}p\tmix\Vnorm[V^{1/2p}]{f} \eqsp.
\end{equation*}
\end{theorem}
The proof follows the same lines as \Cref{theo:rosenthal_VGE} and is omitted. Note that in \Cref{theo:rosenthal_VGE_polynom}, we assume that $\Vnorm[V^{1/2p}]{f} < \infty$ in order to bound the $p$-th moment of $\Sstat_n$. However, it is expected that the weaker condition $\Vnorm[V^{1/p}]{f} < \infty$ should be sufficient. The stronger condition we impose on $f$ is most likely an artifact of the proof based on a dyadic decomposition. Similarly to the UGE setting \Cref{ass:UGE}, we can derive Bernstein-type bounds for $\Sstat_n$ under \Cref{ass:VGE}, based on \Cref{theo:rosenthal_VGE} and \Cref{lem:deviation_bound}. The precise statement is given in Subsection 3 of \Cref{sec:supplement}.

\subsection{Geometrically ergodic Markov kernels with respect to Wasserstein semi-metric}
\label{sec:WGE}
In this section, we extend the results obtained in \Cref{sec:uge-chains} and \Cref{sec:geom-v-ergod} to the case of Markov kernels that are geometrically contracting with respect to the Wasserstein semi-metric; see \cite{hairer2008ergodic,hairer2011asymptotic} and \cite[Chapter~20]{douc:moulines:priouret:soulier:2018}. This setting covers cases where the Markov kernel is not irreducible. In this context,  most techniques based on regeneration methods (see, e.g., \cite{adamczak2015exponential}) are not applicable, since they rely on the splitting method \cite{athreya1978new,Nummelin1978AST}, which in turn requires the existence of a small set to split the sum $\Sstat_n$ into a random number of independent blocks.
\par 
Important examples of non-irreducible Markov chains are given by stochastic algorithms and Markov chains in infinite dimensions; see \cite{hairer:stuart:vollmer:2012,eberle:2014} and \cite[Chapter~20]{douc:moulines:priouret:soulier:2018}. 

We now introduce central concepts and definitions for this section. First, we define the Wasserstein semimetric $\Wass{\cost}$ with respect to a cost function $\cost: \Zset \times \Zset \to \rset_{+}$, 
satisfying:
\begin{assumption}
\label{ass:cost_fun}
The function $\cost: \Zset \times \Zset \rightarrow \rset_{+}$ is symmetric, lower semi-continuous and
$\cost(x, y) = 0$ if and only if $x = y$. Moreover, there exists $q \in \nset$ such that
$(d(x, y) \wedge 1)^q \leq \cost(x, y) \leq 1$.
\end{assumption}
For two probability measures $\xi$ and $\xi'$ on $(\Zset,\Zsigma)$, we say that a probability measure $\nu$ on $(\Zset^2,\Zsigma^{\otimes 2})$ is a coupling of $\xi$ and $\xi'$, if for each $\msa \in \Zsigma$, $\nu(\msa \times \Zset) = \xi(\msa)$ and $\nu( \Zset \times \msa) = \xi'(\msa)$. Denote by $\mathcal{C}(\xi,\xi')$  the set of couplings of $\xi$ and $\xi'$. The Wasserstein semimetric $\Wass{\cost}$, associated with the cost function $\cost$, is defined for  two probability measures $\xi$ and $\xi'$ on $(\Zset,\Zsigma)$ as
\begin{equation}
\label{eq:def_wasserstein_distanse}
\Wass{\cost}(\xi,\xi') = \underset{\nu \in \mathcal{C}(\xi, \xi')}{\inf}\int_{\Zset \times \Zset} \cost(z, z')\nu (\rmd z, \rmd z')\eqsp.
\end{equation}

As in \Cref{sec:geom-v-ergod}, we fix a function $V: \Zset \to [\rme; +\infty)$. Equipped with $\Wass{c}$ and \Cref{ass:cost_fun}, we can state the main assumption of this section.
\begin{assumption}
\label{assum:wasserstein-convergence}
The Markov kernel $\MK$ has a unique invariant probability measure $\pi$ satisfying $\pi(V) < \infty$, and there exist $\tmix \in \nset$, $\varkappa_{\cost} \geq 1$, $\varsigma \geq 1$, such that, for any $\alpha \in (0,1/2]$ and $n \in \nset$
\begin{equation}
\label{eq:wasser:geo:bound:pi}
\sup_{z,z' \in \Zset} \frac{\Wass{\cost^{1/2} V^\alpha}( \delta_z \MKQ^n, \delta_{z'} \MKQ^n)}{\cost^{1/2}(z, z')\bigl(V^{\alpha}(z)+V^{\alpha}(z')\bigr)} \leq \varsigma
 \varkappa_{\cost}^{\alpha} (1/4)^{\lfloor n/\tmix\rfloor \alpha} \eqsp.
\end{equation}
\end{assumption}

 It is possible to establish \eqref{eq:wasser:geo:bound:pi}, with $V$ constant, using the $\cost$-Dobrushin coefficient associated with $\MKQ$. However, this approach and its corresponding results are very limited in scope; see \cite[Chapter 20]{douc:moulines:priouret:soulier:2018}. In contrast, \cite{hairer2011asymptotic,durmus2015quantitative} have shown that \Cref{assum:wasserstein-convergence} can be obtained by combining a geometric Foster--Lyapunov drift condition \cite[Chapters~15--16]{meyn:tweedie} with a $d$-small set condition. For completeness, we provide in Subsection~1 of \Cref{sec:supplement} a detailed account of these conditions and the derivations establishing \Cref{assum:wasserstein-convergence} from them.
 Finally, the form of \Cref{assum:wasserstein-convergence} follows directly from these results; see \Cref{cor:wasserstein}.

\par 
Define the function $\bar V : \Zset \times \Zset \to [\rme,+\infty)$ by $\bar V(z,z') = \{V(z) + V(z')\}/2$ for any $z,z'\in\Zset$. In addition, for $\beta \geq 0$, we define the weighted seminorm of $f :\msz \to \rset$ by the formula 
\[
\Nnorm[\beta, V]{f} = \max \left\{ \underset{z, z' \in \Zset, z \neq z'}{\sup} \frac{|f(z) - f(z')|}{\cost^{1/2}(z,z')\bar{V}^{\beta}(z, z')}, \underset{z \in \Zset}{\sup}\eqsp \frac{|f(z)|}{V^{\beta}(z)}\right\}\eqsp,
\]
and denote by $\mathcal{L}_{\beta, V}$ the class of functions ${\mathcal{L}_{\beta, V} = \{f: \Zset \rightarrow \mathbb{R}: \Nnorm[\beta, V]{f} < \infty \}}$.
Under \Cref{ass:cost_fun} and \Cref{assum:wasserstein-convergence} from \cite[Corollary 20.4.7]{douc:moulines:priouret:soulier:2018} for $\alpha  \in (0,1/2]$ and $f \in \mathcal{L}_{\alpha, V} $ it holds
\begin{equation*}
|\xi Q^n(f) - \pi(f)| \leq \varsigma \varkappa_{\cost}^{\alpha} (1/4)^{\lfloor n/\tmix\rfloor \alpha}\{\xi(V^{\alpha})+\pi(V^{\alpha})\}\Nnorm[\alpha, V]{f} \eqsp.
\end{equation*}
Therefore, $\sigma_{\pi}(f)$ in \eqref{eq:asympt_var_def} is well defined for any $f \in \mathcal{L}_{1/2,V}$.

\begin{theorem}
\label{theo:rosenthal_WGE}
Assume  \Cref{ass:cost_fun} and \Cref{assum:wasserstein-convergence}, and let $\beta \geq 0$. Then for any measurable function $f: \Zset \rightarrow \rset$ with $\Nnorm[\beta, \log V]{f} < \infty$, $p \geq 2$, and $n \geq 3$, it holds that
\begin{multline*}
\PE_{\pi}^{1/p}\bigl[\bigl| \Sstat_n \bigr|^{p} \bigr] \leq \bConst{\sf{Rm}, 1}\sqrt{2}p^{1/2}n^{1/2}\sigma_{\pi}(f) +  \ConstD_{\ref{theo:rosenthal_WGE},1}^{(\beta)} p^{\beta+2}n^{1/4}\tmix^{3/4} \varsigma \{\varkappa_{\cost} \pi(V)\}^{2/p}\log_2(2p) \Nnorm[\beta, \log V]{f} \\ + \ConstD_{\ref{theo:rosenthal_WGE},2}^{(\beta)} p^{\beta +2} \tmix \varsigma (\log{n})^{\beta+1} \{\varkappa_{\cost} \pi(V)\}^{1/p}\log_2(2p) \Nnorm[\beta, \log V]{f} \eqsp,
\end{multline*}
where 
\begin{equation*}
\ConstD_{\ref{theo:rosenthal_WGE},1}^{(\beta)} = 4\ConstD_{\ref{lem: rosenthal_WGE_for_2^s}, 1}(2\beta / \rme)^{\beta}\eqsp, \quad \ConstD_{\ref{theo:rosenthal_WGE},2}^{(\beta)} = 4\ConstD_{\ref{lem: rosenthal_WGE_for_2^s}, 2}(2\beta / \rme)^{\beta} \eqsp, 
\end{equation*}
and $\ConstD_{\ref{lem: rosenthal_WGE_for_2^s}, 1}$, $\ConstD_{\ref{lem: rosenthal_WGE_for_2^s}, 2}$ are universal constants gathered in \Cref{tab:univ_constants} in \Cref{sec:constants}. Moreover, for any
initial probability $\xi$ on $(\Zset, \Zsigma)$,
\begin{equation*}
\PE_{\xi}^{1/p}\bigl[\bigl| \Sstat_n \bigr|^{p}\bigr]  \leq \PE_{\pi}^{1/p}\bigl[\bigl| \Sstat_n \bigr|^{p}\bigr]  + \tfrac{32}{15} (2\beta/\rme)^{\beta} p^{\beta + 1} \varsigma \tmix \varkappa_{\cost}^{1/(2p)}(\{\pi(V)\}^{1/(2p)} + \{\xi(V)\}^{1/(2p)})\Nnorm[\beta, \log V]{f}\eqsp.
\end{equation*}
\end{theorem}
The proof is given in \Cref{sec:proof_rosenthal_WGE}. Similarly to \Cref{sec:geom-v-ergod}, we provide the counterpart of \Cref{theo:rosenthal_WGE} under the condition $\Nnorm[1/(2p), V]{f} < \infty$. Since its proof follows the same line as \Cref{theo:rosenthal_WGE}, it is omitted.

\begin{theorem}
\label{theo:rosenthal_WGE_polynom}
Assume  \Cref{ass:cost_fun} and \Cref{assum:wasserstein-convergence}, and let $p \geq 2$. Then for any measurable function $f: \Zset \rightarrow \rset$ with $\Nnorm[1/(2p), V]{f} < \infty$, and $n \geq 3$, it holds that
\begin{multline*}
\PE_{\pi}^{1/p}\bigl[\bigl| \Sstat_n \bigr|^{p} \bigr] \leq \bConst{\sf{Rm}, 1}\sqrt{2}p^{1/2}n^{1/2}\sigma_{\pi}(f) +
4\ConstD_{\ref{lem: rosenthal_WGE_for_2^s}, 1} p^{2}n^{1/4}\tmix^{3/4} \varsigma \{\varkappa_{\cost} \pi(V)\}^{2/p}\log_2(2p)\Nnorm[1/(2p), V]{f} \\ + 4\ConstD_{\ref{lem: rosenthal_WGE_for_2^s}, 2} p^2 \tmix \varsigma n^{1/p} \{\varkappa_{\cost} \pi(V)\}^{1/p}\log_2(2p)  \Nnorm[1/(2p), V]{f}\eqsp.
\end{multline*}
Moreover, for any initial probability $\xi$ on $(\Zset, \Zsigma)$,
\begin{equation*}
\PE_{\xi}^{1/p}\bigl[\bigl| \Sstat_n \bigr|^{p}\bigr]  \leq \PE_{\pi}^{1/p}\bigl[\bigl| \Sstat_n \bigr|^{p}\bigr]  + (32/15)p \varsigma \tmix \varkappa_{\cost}^{1/(2p)}(\{\pi(V)\}^{1/(2p)} + \{\xi(V)\}^{1/(2p)})\Nnorm[1/(2p), V]{f}\eqsp.
\end{equation*}
\end{theorem}


\section{Literature review}
\label{sec:discussion_litt}
Concentration inequalities for additive functionals of Markov chains have been studied using a wealth of different techniques. Making a complete survey of existing results is out of the scope of this paper, and we focus here only on the main research directions and techniques devoted to them.
\par 
\textbf{Concentration inequalities under geometric ergodicity conditions.} A popular research direction, developed in particular in \cite{clemenccon2001moment,bertail2010sharp,adamczak2008tail,adamczak2015exponential,bertail2018new,lemanczyk2020general}, is to use regenerative decompositions to obtain moment bounds and Bernstein inequalities under geometric ergodicity assumptions. This setting covers \Cref{ass:UGE} and \Cref{ass:VGE} but not \Cref{sec:WGE}. These techniques are based on the Nummelin splitting construction (see \cite{athreya1978new} and \cite{Nummelin1978AST}). The closest counterparts of our results are those obtained in \cite{adamczak2015exponential} and \cite{lemanczyk2020general}. \cite[Theorem~1]{adamczak2015exponential} provides a Bernstein-type inequality for $V$-geometrically ergodic, strongly aperiodic Markov kernels (see, e.g., \cite[Definition 9.3.5]{douc:moulines:priouret:soulier:2018} for the definition) and unbounded functions, making it a direct counterpart of \Cref{theo:rosenthal_VGE}. Similar to the setting of \Cref{theo:rosenthal_VGE}, the result of \cite[Theorem~1]{adamczak2015exponential} applies to functions $f$ satisfying $\Vnorm[\log^{\beta}{V}]{f} < \infty$ for some $\beta \geq 0$. Compared to \cite[Theorem~1]{adamczak2015exponential}, our results do not require strong aperiodicity of the chain, while still retaining the exact Gaussian leading term. Moreover, our bounds are explicit and homogeneous with respect to $\tmix$.
The strong aperiodicity requirement was previously relaxed in \cite[Theorem~1]{lemanczyk2020general}, but that result applies only to bounded functions and does not provide an explicit expression for the constants, in particular with respect to $\tmix$.
 The case of regenerative positive recurrent Markov chains and unbounded functions was considered in \cite{bertail2018new}, but with a variance proxy (an upper bound for the exact asymptotic variance $\sigma_{\pi}^2(f)$) as the leading term, and without explicit dependence on the mixing time.
\par 
Alternative set of techniques, initially developed for weakly dependent sequences, is based on the  cumulant method and the  Leonov-Shiryaev formula \cite{leonov:sirjaev:1959}. Corresponding results for weakly dependent sequences can be found in \cite{saulis:statulevicius:1991,doukhan:louhichi:1999, doukhan2007probability}. This technique has been applied to the Markov kernels under both \Cref{ass:VGE} or \Cref{assum:wasserstein-convergence} in \cite{durmus2024probability}. In contrast to \cite{durmus2024probability}, our bounds rely on the asymptotic variance $\sigma_{\pi}^2(f)$, not $\PVar_{\pi}(\Sstat_n)$, as a leading term, and imply tighter dependence in $p$. Precisely, the right-hand side in \Cref{theo:rosenthal_VGE_polynom} is of order $p^2\log{p}$, whereas it is at least of order $p^3$ in the right-hand side of the Rosenthal-type bounds established in \cite{durmus2024probability}. Furthermore, the remainder terms of the Rosenthal inequality from \cite{durmus2024probability} scale with $(\PVar_{\pi}[\Sstat_n])^{-1}$. Thus, to keep precise scaling with $\tmix$, one has to establish lower bounds of the form $\PVar_{\pi}(\Sstat_n) \gtrsim n \tmix$, which, to our knowledge, have never been established under general conditions such as \Cref{ass:VGE} or \Cref{assum:wasserstein-convergence} without additional assumptions. 
\par 
\textbf{Concentration inequalities under mixing conditions.} Rosenthal-type and Bernstein-type inequalities can be obtained for general dependent sequences under various mixing conditions; see \cite{rio2017asymptotic} for an overview. Apart from the results using the cumulant methods discussed above, \cite{merlevede2011bernstein} established Bernstein-type inequalities for $\tau$-mixing processes, implying such inequalities for additive functionals of atomic chains or for Markov kernels satisfying \Cref{assum:wasserstein-convergence} with $c^{1/2} = d$ and $V \equiv 1$. The variance term identified in \cite{merlevede2011bernstein} is an upper bound on the asymptotic variance $\sigma^2_{\pi}(f)$. The same type of bound, with a variance proxy as a leading term, can also be deduced under \Cref{ass:VGE} using the result of \cite[Theorem 6.2]{rio2017asymptotic}. Moreover, the constant reflecting dependence on the mixing time of the chain in \cite{merlevede2011bernstein} is not explicit. Tracking this dependence naively, we obtain Rosenthal-type bounds that are not homogeneous with respect to $n$ and $\tmix$. In \cite[Theorem 2.1]{dedecker2024deviation}, the authors obtain deviation bounds for $\theta$-mixing sequences (see \cite[Definition 1.1]{dedecker2024deviation}) with two regimes: an exponential one with the coefficient scaling with $\sigma_{\pi}^2(f)$, and a polynomial one with implicit dependence on the mixing time. While the $\theta$-mixing setting covers both \Cref{ass:VGE} and \Cref{assum:wasserstein-convergence}, the latter result is weaker than the exponential inequalities provided by \Cref{theo:rosenthal_VGE} and \Cref{theo:rosenthal_WGE} in this paper.
\par 
\textbf{Concentration inequalities under Wasserstein contraction.} Some works considered Markov kernels satisfying \Cref{assum:wasserstein-convergence} with $c^{1/2}=d$ and $V \equiv 1$ and derived concentration inequalities for $(S_n)_{n\in\nset}$. This particular form of \Cref{assum:wasserstein-convergence} is referred to as the Positive Ricci Curvature (PCR) condition in  \cite{joulin:ollivier:2010}. Under PCR and additional technical conditions, the authors of \cite{joulin:ollivier:2010} show that sub-Gaussian concentration holds for Lipschitz functions $f$, with a variance proxy (sub-Gaussian parameter) as the leading term. One-step contraction conditions in $\mathrm{L}^{p}(\pi)$, essentially similar to the PCR condition, were also considered in \cite{dedecker2019deviation,dedecker2015deviation,doukhan2024deviation}. In these papers, the authors derive McDiarmid- and Bernstein-type bounds for separately Lipschitz functions, a more general setting compared to that of additive functionals. To our knowledge, they only obtain variance proxies in their bounds instead of the asymptotic variance $\sigma_{\pi}^2(f)$. Moreover, the higher-order terms are not explicit with respect to $\tmix$. Tracking this dependence naively, we obtain Rosenthal-type bounds that are not homogeneous with respect to $n$ and $\tmix$.
\par 
\textbf{Bernstein-type inequalities under a spectral gap condition.} Building upon \cite{lezaud:1998}, the analyses in \cite{paulin2015concentration,jiang2018bernstein} lead to Bernstein-type inequalities for bounded functions under the assumption that $\MKQ$ admits a spectral gap in $\ltwo(\pi)$. However, as shown in \cite[Theorem 1.4]{kontoyiannis2012geometric}, the existence of a spectral gap in $\ltwo(\pi)$ does not necessarily follow from geometric ergodicity for non-reversible Markov kernels; thus, this assumption is not directly comparable with \Cref{ass:VGE}. Furthermore, \cite{paulin2015concentration} establishes a Bernstein-type inequality with the exact asymptotic variance $\sigma_{\pi}^2(f)$ of $(S_n)_{n\in\nset}$ only if $\MKQ$ is reversible, while \cite{jiang2018bernstein} provides a version of the Bernstein inequality using a variance proxy.
\par 
\textbf{Hoeffding-type and McDiarmid-type inequalities.} Concentration inequalities other than Rosenthal- and Bernstein-type ones have also been actively developed. Hoeffding-type bounds were established in \cite{glynn2002hoeffding} under a uniform minorization condition equivalent to \Cref{ass:UGE}. The proof technique of \cite{glynn2002hoeffding}, as well as our approach, is based on properties of solutions to the Poisson equation (see \eqref{eq:Poisson} below) and on a martingale decomposition. However, the resulting bounds involve the infinity norm of the solution to the Poisson equation, which implies that fluctuations of $\Sstat_n$ are of order $n^{1/2}\tmix$ instead of $n^{1/2}\tmix^{1/2}$, as predicted by the CLT. The authors of \cite{kontoyiannis2005relative} also develop Hoeffding-type bounds by combining information-theoretic ideas with techniques from optimization. In particular, \cite[Theorem~4]{kontoyiannis2005relative} implies a high-probability bound for $\Sstat_n$ of the same order $n^{1/2}\tmix$. \cite[Section~2]{paulin2015concentration} derives optimal (in terms of $\tmix$ and $n$) Hoeffding-type bounds using a coupling argument based on the Marton coupling \cite{marton1996measure}. Let us also mention \cite{miasojedow2014hoeffding,fan2021hoeffding}, which establish Hoeffding-type bounds using spectral theory in $\ltwo(\pi)$, involving the spectral gap of $\MKQ$ in $\ltwo(\pi)$. These bounds provide the optimal scaling with respect to $\tmix$ and $n$. Finally, McDiarmid-type inequalities under assumptions equivalent to \Cref{ass:VGE} were established in \cite{dedecker2015subgaussian,havet:mcdiarmid:2020}.

\section{Applications}
\label{sec:apps}
\subsection{Linear Stochastic Approximation} In this section we apply our results to the linear stochastic approximation (LSA) iterates for solving the linear system $\bA \thetas = \barb$. In the context of LSA, neither $\bA$ nor $\barb$ are known, instead the learner relies on the sequence of observations $\{( \funcA{Z_k}, \funcb{Z_k})\}_{k \in \nset}$, where $\Am: \msz \to \rset^{d \times d}$, $\bm: \msz \to \rset^d$ are measurable functions, and $(Z_k)_{k \in \nset}$ is an i.i.d. sequence of random variables. With a fixed step size $\gamma > 0$, the LSA algorithm constructs the sequence $(\theta_k)_{k \in \nset}$ defined recurrently by
\begin{equation}
\label{eq:lsa}
\textstyle \theta_{k} = \theta_{k-1} - \gamma \{ \funcA{Z_k} \theta_{k-1} - \funcb{Z_k} \} \eqsp,~~ k \geq 1 \eqsp,
\end{equation}
with the deterministic initialization $\theta_0 \in \rset^{d}$. It is well-known (see e.g. \cite{durmus2021tight}), that under appropriate assumptions the Markov chain $( \theta_k )_{k \in \nset}$ is geometrically ergodic in weighted Wasserstein semi-metric. A popular estimate of $\thetas$ based on \eqref{eq:lsa} is the Polyak-Ruppert averaged estimator \cite{ruppert1988efficient,polyak1992acceleration}, which writes as
\begin{equation}
\label{eq:lsa-pr}
\textstyle{
\prtheta_{n} = n^{-1} \sum_{k=0}^{n-1} \theta_k 
}\eqsp.
\end{equation}
The finite-time performance of Polyak-Ruppert averaged LSA (LSA-PR) has attracted much attention recently; see \cite{durmus2022finite,mou2020linear,mou2021optimal,huo2023bias}. 

Our goal is to study the properties of the Markov chain $(\theta_k)_{k \in \nset}$ and obtain moment bounds on $\prtheta_{n} - \thetas$. To this end, we consider assumptions based on the equivalent formulation of \eqref{eq:lsa}:
\begin{equation}
\label{eq:clever_lsa_step}
\theta_{k} - \thetas = (\Id - \gamma \funcA{Z_k})(\theta_{k-1} - \thetas) - \gamma \funcnoise{\State_{k}}\eqsp, \text{ where }
\end{equation}
\begin{equation}
\label{eq:def_center_version_and_noise}
\textstyle
\funcnoise{z} =  \zmfuncA{z} \thetas - \zmfuncb{z}\eqsp, \quad \zmfuncA{z}  = \funcA{z} - \bA \eqsp, \quad   \zmfuncb{z} = \funcb{z} - \barb \eqsp \eqsp.
\end{equation}
With the introduced notations, we state our first assumption.
\begin{assumptionLSA}
\label{assum:noise-level}
$(Z_k)_{k \in \nset}$ is a sequence of i.i.d. random variables defined on a probability space $(\Omega,\mcf,\PP)$ with distribution $\nu$, such that 
$\int_{\Zset}\funcA{z}\nu(\rmd z) = \bA$ and $\int_{\Zset}\funcb{z}\nu(\rmd z) = \barb$. Moreover, $\supconsteps = \sup_{z \in \msz}\normop{\funcnoise{z}} < \infty$.
\end{assumptionLSA} 
Our next assumption concerns the mappings $z \mapsto \zmfuncAw(z)$ and the system matrix $\bA$:
\begin{assumptionLSA}
\label{assum:A-b}
$\bConst{A} = \sup_{z \in \msz} \normop{\funcA{z}} \vee \sup_{z \in \msz} \normop{\zmfuncA{z}} < \infty$ and  the matrix $-\bA$ is Hurwitz.
\end{assumptionLSA}
Note that \Cref{assum:noise-level} and \Cref{assum:A-b} are common assumptions in the stochastic approximation literature, see e.g. \cite{mou2020linear,durmus2022finite}. Part of \Cref{assum:noise-level} concerning the almost sure boundedness of $\funcnoise{Z_1}$ can be relaxed to an appropriate moment-type bound, yet for clarity we prefer to keep this assumption. The condition that $-\bA$ is Hurwitz in \Cref{assum:A-b} implies that the linear system $\bA \theta = \barb$ has a unique solution $\thetas$.  \Cref{{assum:A-b}} also allows to ensure that $\normop{\Id - \gamma \bA}[\mathbf{P}] < 1$ for an appropriate operator norm. We state the corresponding proposition below. We denote here by $\norm{x}[\mathbf{S}] = x^{\top} \mathbf{S} x$ and for any matrix $\mathbf{B}$, we denote by $\normop{\mathbf{B}}[\mathbf{S}]= \sup_{x\,: \, \norm{x}[\mathbf{S}] \leq 1} \norm{\mathbf{B}x}[\mathbf{S}]$.
\begin{proposition}[\protect{\cite[Proposition 1]{durmus2021tight}}]
\label{fact:Hurwitzstability}
Let $-\bA$ be a Hurwitz matrix. Then there exists a unique symmetric positive definite matrix $\mathbf{P}$ satisfying the equation
$\bA^\top \mathbf{P} + \mathbf{P} \bA =  \Id$. In addition, setting
\begin{equation}
\label{eq: kappa_def}
a = \normop{\mathbf{P}}^{-1}/2\eqsp, \quad
\text{and} \quad \gamma_\infty = (1/2) \normop{\bA}[\mathbf{P}]^{-2} \normop{\mathbf{P}}^{-1} \wedge \normop{\mathbf{P}} \eqsp,
\end{equation}
it holds for any $\gamma \in [0, \gamma_{\infty}]$ that $\normop{\Id - \gamma \bA}[\mathbf{P}]^2 \leq 1 - a \gamma$, and $\gamma a \leq 1/2$.
\end{proposition}
Define for $q \geq 2$ the constants
\begin{align}
\label{eq:def_qcond_b_Q}
&\textstyle{\qcond = \lambda_{\sf max}( \mathbf{P} )/\lambda_{\sf min}( \mathbf{P} )  \eqsp, \quad  b_{\mathbf{P}} = 2 \sqrt{\qcond} \bConst{A} \eqsp,} \quad \\
\label{eq:def_alpha_p_infty}
&\textstyle{\gamma_{q,\infty} = \gamma_{\infty} \wedge \smallAconst/q \eqsp, \quad \smallAconst = a/\{2b_\mathbf{P}^2\}}\eqsp, \quad \ConstD = (2 \qcond)^{1/2} a^{-1}(1 + 4 \qcond^{1/2} \bConst{A}  a^{-1})\eqsp.
\end{align}
The sequence $(\theta_{k})_{k \in \nset}$ defined by \eqref{eq:lsa} is a Markov chain associated with the Markov kernel $\MK_{\gamma}$, defined by
\begin{equation}
\label{eq:MK_alpha_def}
\textstyle{
\MK_{\gamma} f(\theta) = \int_{\msz}f \left( \{ \Id - \gamma \funcA{z} \} \theta + \gamma \funcb{z} \right) \nu(\rmd z)
}\eqsp,
\end{equation}
where $f: \rset^{d} \to \rset$ is a bounded measurable function, and $\theta \in \rset^{d}$. Assumptions \Cref{assum:noise-level} and \Cref{assum:A-b} are sufficient to check that $(\theta_{k})_{k \in \nset}$ is geometrically ergodic in weighted Wasserstein semi-metric with a particular cost function $\cost$. For $p \geq 2$, define the drift function
\begin{equation}
\label{eq:drift_function}
V_p(\theta) = \rme + \norm{\theta - \thetas}^{p}\eqsp.
\end{equation} 
We can formulate the following result. 
\begin{proposition}
\label{th:explicit_constants_under_A1_more_specific}
Assume \Cref{assum:noise-level} and \Cref{assum:A-b}. Let $p \geq 2$. Then, for $\gamma \in (0,\gamma_{p(1+\log d),\infty}]$, the Markov kernel $\MK_{\gamma}$ defined in \eqref{eq:MK_alpha_def} has a unique invariant distribution $\pi_{\gamma}$ and satisfies \Cref{ass:cost_fun} and \Cref{assum:wasserstein-convergence} with the drift function $V_p$ defined in \eqref{eq:drift_function}, and
\begin{equation*}
\begin{split}
\cost(\theta,\theta') &= 1 \wedge \norm{\theta - \theta'}^2\eqsp, \quad  \varsigma \lesssim 1 \eqsp, \quad \varkappa_{c}^{1/(2p)} \lesssim 1 \eqsp, \quad 
\varrho  \leq 7/8 \eqsp, \quad \tmix  \lesssim p/\gamma\eqsp,
\end{split}
\end{equation*}
where $\lesssim$ stands for inequality up to a constant, not depending upon $p$, $\gamma$, and $n$
\end{proposition}

We provide the proof of \Cref{th:explicit_constants_under_A1_more_specific} along with a statement with explicit constants  in Subsection~2 of \Cref{sec:supplement}. \Cref{th:explicit_constants_under_A1_more_specific} guarantees, that the following asymptotic covariance matrix is well-defined:
\begin{equation}
\label{eq:ass_cov_matrix}
\Sigma_{\pi_{\gamma}} = \PE_{\pi_{\gamma}}[(\theta_0 - \theta^*)(\theta_0 - \theta^*)^{T}] + 2\sum_{k=1}^{\infty}\PE_{\pi_{\gamma}}[(\theta_k - \theta^*)(\theta_0 - \theta^*)^{T}]\eqsp.
\end{equation}
We further show, that the matrix $\Sigma_{\pi_{\gamma}}$ is close to the matrix 
\begin{equation}
\label{eq:sigma_infty}
\Sigma_{\infty} =  \bA^{-1}\Sigma_{\varepsilon}(\bA^{-1})^{\top}\eqsp, \quad \text{ where } \Sigma_{\varepsilon} := \mathbb{E}[\varepsilon(z)\varepsilon(z)^\top]\eqsp.
\end{equation}
The importance of $\Sigma_{\infty}$ is due to the fact that it is known to be optimal from information-theoretical point of view for SA algorithms with decreasing step size, see \cite{polyak1992acceleration,fort2015central}. Closeness between $\Sigma_{\pi_{\gamma}}$ and $\Sigma_{\infty}$ can be quantified as follows: 
\begin{lemma}
\label{lem:lsa_sigmas_eq}
Assume \Cref{assum:noise-level} and \Cref{assum:A-b}. Then for any step size 
\[
\gamma \in (0,\gamma_{2,\infty} \wedge \sigma_{\min}(I\otimes\bA + \bA\otimes I)/\bConst{A}^2]\eqsp,
\]
where $\gamma_{2,\infty}$ is defined in \eqref{eq:def_alpha_p_infty}, it holds that
\begin{equation}
\Sigma_{\pi_{\gamma}} = \Sigma_{\infty} + \gamma \mathbf{B} \eqsp, 
\end{equation}
where $\mathbf{B} \in \rset^{d \times d}$ is a matrix, such that $\norm{\mathbf{B}} \leq C$ for some constant $C$ not depending upon $\gamma$.
\end{lemma}
The proof of \Cref{lem:lsa_sigmas_eq} is given  in Subsection~2 of \Cref{sec:supplement}. Combining \Cref{th:explicit_constants_under_A1_more_specific} and \Cref{lem:lsa_sigmas_eq}, we can conclude:
\begin{theorem}
\label{prop:W_ergodic_LSA}
Assume \Cref{assum:noise-level}, \Cref{assum:A-b}, and let $p \geq 2$. Then for any step size 
\[
\gamma \in (0, \gamma_{p(1+\log{d}),\infty} \wedge \sigma_{\min}(I\otimes\bA + \bA\otimes I)/\bConst{A}^2]\eqsp,
\]
where $\gamma_{p(1+\log{d}),\infty}$ is defined in \eqref{eq:def_alpha_p_infty},  for any $\theta_0 \in \rset^{d}$ and $n \in\nset$, it holds that 
\begin{multline*}
\PE^{2/p}\bigl[\|\prtheta_{n} - \thetas\|^{p/2}\bigr]  \\ \lesssim \frac{p^{1/2} \sqrt{\trace{\Sigma_{\infty}}}}{n^{1/2}} + \sqrt{\frac{p\gamma}{n}}+ \frac{p^{11/4} \log_2(p)}{(\gamma n)^{3/4}} + \frac{p^3\log_2(2p)}{\gamma n^{1-1/p}} + \frac{(\norm{\theta_0 - \thetas} + 1) p^2}{\gamma n} \eqsp,
\end{multline*}
where $\lesssim$ stands for inequality up to a constant, not depending upon $p$, $\gamma$, and $n$.
\end{theorem}
\begin{proof}
To apply \Cref{theo:rosenthal_WGE_polynom}, we need to check that $f(\theta) = \theta - \thetas \in \mathcal{L}_{1/p, V_p}$, where $V_p$ is defined in \eqref{eq:drift_function}. Simple algebraic manipulations yield
\begin{equation}
    \label{eq:bound_v_norm_lsa}
    \sup_{\theta \in \rset^{d}}\frac{f(\theta)}{(V_p(\theta))^{1/p}} \leq 1 \eqsp,
\end{equation}
\begin{equation}
    \label{eq:bound_lip_norm_lsa}
    \underset{\theta}{\sup}\frac{f(\theta) - f(\theta')}{c^{1/2}(\theta, \theta')(\bar{V_p}(\theta, \theta')^{1/p}} \leq
    \begin{cases}
    \underset{\theta}{\sup}\frac{\norm{\theta-\theta^*}+\norm{\theta'-\theta^*}}{((\norm{\theta-\theta^*}^p+\norm{\theta'-\theta^*}^p)/2)^{1/p}} \leq 2, &\text{if $\norm{\theta - \theta'} \geq 1$} \\
    \underset{\theta}{\sup}\frac{\norm{\theta-\theta'}}{\norm{\theta-\theta'}\rme^{1/p}} \leq 1\eqsp, & \text{if $\norm{\theta - \theta'} \leq 1$}
    \end{cases}
\end{equation}
Hence, we get that $\Nnorm[1/p, V_p]{f} \leq 2$.
From \Cref{th:explicit_constants_under_A1_more_specific} it follows that, for any $\theta_0 \in \rset^{d}$, 
\[
\textstyle{
\pi_{\gamma}(V_p) = \lim_{n \rightarrow \infty} \PE_{\theta_0}[e + \norm{\theta_n - \theta^*}]  \lesssim 
(\gamma  p)^{p/2}
}\eqsp.
\]
The rest of the proof follows from \Cref{th:explicit_constants_under_A1_more_specific} , \Cref{theo:rosenthal_WGE_polynom} and \Cref{lem:lsa_sigmas_eq}.
\end{proof}
The result of \Cref{prop:W_ergodic_LSA} can be viewed as a counterpart of \cite[Theorem~3]{mou2020linear} and \cite[Theorem~1]{durmus2022finite}. As compared to \cite[Theorem~3]{mou2020linear}, the result of \Cref{prop:W_ergodic_LSA} yields slightly worse decay rates of the non-leading (w.r.t. $n$) terms, but allows to obtain the same leading term, scaling as $\sqrt{\trace{\Sigma_{\infty}}}$. At the same time, our proof of \Cref{prop:W_ergodic_LSA} simplifies the argument of \cite[Theorem~3]{mou2020linear} and relies only on bounds on the mixing time of $\{\theta_n\}_{n \in \nset}$, obtained in \Cref{th:explicit_constants_under_A1_more_specific}.

\subsection{Diffusion-based MCMC methods}
We consider applications of our results to nonasymptotic analysis of MCMC algorithms. Let $\pi$ be a probability on $(\rset^{d},\B(\rset^d))$ defined for any $\msa \in \B(\rset^d)$ as
\[
\pi(\msa) = \int_{\msa}\rme^{-U(x)}\,\rmd x / Z
\eqsp, \quad Z:= \int_{\rset^{d}}\rme^{-U(x)}\,\rmd x\eqsp, 
\]
where $U: \rset^{d} \to \rset$ is a measurable function such that $\int_{\rset^{d}}\rme^{-U(x)}\,\rmd x < \infty$. We aim to approximate the integral
$\pi(f) = \int_{\rset^{d}}f(x)\pi(\rmd x)$
for a function $f \in \ltwo(\pi)$. When sampling i.i.d. observations from $\pi$ is not feasible, one  possible solution to this problem is to apply Markov Chain Monte Carlo (MCMC) methods. In this context, we attempt to construct a Markov chain $(\theta_k)_{k \in \nset}$ which admits $\pi$ as unique invariant distribution $\pi$, and use the estimate
$\pi_n(f) = n^{-1}\sum_{\ell = 0}^{n-1}f(\theta_{\ell})$.
Many popular MCMC algorithms are inspired by the overdamped Langevin SDE (stochastic differential equation) given by
\begin{equation}
\label{eq:langevin_sde}
\rmd Y_t= - \nabla U(Y_t)\,\rmd t + \sqrt{2} \rmd W_t\eqsp,
\end{equation}
where $(W_t)_{t \geq 0}$ is a $d$-dimensional Brownian motion. Under suitable conditions for $U$, \eqref{eq:langevin_sde} has a unique strong solution for every $Y_0 = x \in \rset^{d}$ and \eqref{eq:langevin_sde} defines a Markov semigroup  which is $\pi$-reversible; see e.g., \cite{roberts:tweedie-Langevin:1996}. 
In this section, we attempt to provide a finite-time performance guarantee for MCMC samplers based on discretizations of the SDE \eqref{eq:langevin_sde}.
\par 
An example of such a discretization is the Unadjusted Langevin Algorithm (ULA), see \cite{roberts:tweedie-Langevin:1996}. It arises as an Euler discretization scheme associated with the Langevin SDE \eqref{eq:langevin_sde}:
\begin{equation}
\label{eq:langevin_chain}
\textstyle{
\theta_{k+1} = \theta_{k} - \gamma \nabla U(\theta_k) +\sqrt{2\gamma} Z_{k+1}
}\eqsp,
\end{equation}
where $\gamma > 0$ is the discretization step and $(Z_k)_{k \in \nset}$ is an i.i.d. sequence of $d-$dimensional standard Gaussian vectors; see \cite{roberts:tweedie-Langevin:1996}. The nonasymptotic properties of the ULA algorithm were considered in many recent works; see among many others, e.g., \cite{durmus:moulines:2015,dalalyan2017theoretical}.
\par
Another example of diffusion-based MCMC is the Metropolis Adjusted Langevin Algorithm (MALA), which can be obtained from \eqref{eq:langevin_chain} by integrating an additional Metropolis-Hastings correction step. To analyze the ULA and MALA algorithms in the same theoretical framework, we consider a family of Markov kernels $\{\MKQ_{\gamma}, \gamma \in (0;\bar{\gamma}]\}$ parameterized by a scalar variable $\gamma$, which corresponds to the step size in the discretization \eqref{eq:langevin_chain}. More precisely, the corresponding Markov kernels for $\theta \in \rset^{d}$ and $\msa \in \B(\rset^{d})$ are respectively given by
\begin{align*}
\textstyle{
\MKQ_{\gamma}^{\operatorname{(ULA)}}(\theta,\msa)} &= \textstyle{\int_{\rset^{d}}\indi{\msa}\bigl(\varphi(\theta,z)\bigr) \qbf_d(z) \rmd z\,, \quad \varphi(\theta,z) = \textstyle{\theta - \gamma \nabla U(\theta) + \sqrt{2\gamma} z}}\eqsp,  \\
\textstyle{\MKQ_{\gamma}^{\operatorname{(MALA)}}(\theta,\msa)} &= \textstyle{\int_{\rset^{d}}\indi{\msa}(\varphi(\theta,z)) \bigl(1 \wedge \rme^{-\tau_{\gamma}}\bigr) \qbf_d(z) \rmd z + 
\indi{\msa}(\theta) \bigl(1 - \int_{\rset^{d}} \{1 \wedge \rme^{-\tau_{\gamma}}\} \qbf_d(z) \rmd z \bigr)}\eqsp,
\end{align*}
where we have set
\begin{align*}
\textstyle{\tau_{\gamma}(\theta,z)} = \textstyle{U(\varphi(\theta,z)) - U(\theta) + (1/2)\bigl(\norm{z - (\gamma/2)^{1/2}\{\nabla U(\varphi(\theta,z)) + \nabla U(\theta)\}}^2 - \norm{z}^2\bigr)}\eqsp,
\end{align*}
and $\qbf_d(z)$ is a density of a $d$-dimensional standard Gaussian random variable. Assume that the Markov kernels $\MKQ_{\gamma}^{\operatorname{(ULA)}}$ or $\MKQ_{\gamma}^{\operatorname{(MALA)}}$ satisfy the following condition:
\begin{assumptionMC}
\label{as:MC_mixing}
There exists $\bar{\gamma} > 0$ and a function $V: \rset^d \to [\rme;+\infty)$, such that that for any $\gamma \in (0;\bar{\gamma}]$, the Markov kernel $\MKQ_{\gamma} \in \{\MKQ_{\gamma}^{\operatorname{(ULA)}}, \MKQ_{\gamma}^{\operatorname{(MALA)}}\}$ satisfies \Cref{ass:VGE} with mixing time $\tmix_{\gamma} \leq C_{\MKQ}\gamma^{-1}$ and invariant distribution $\pi_{\gamma}$.
\end{assumptionMC}
Under appropriate conditions on $U$, \Cref{as:MC_mixing} is satisfied  for the considered Markov kernels. In particular, it is checked for the ULA and MALA in \cite{durmus2022geometric} under rather weak conditions for the potential $U(\theta)$ with the drift function $V(\theta) = \exp\{\bar{\eta}\norm{\theta}^2\}$ with a  kernel-specific constant $\bar{\eta} > 0$. Note that in case of ULA, the invariant distribution $\pi_{\gamma} \neq \pi$, while MALA keeps the correct invariant distribution. However under appropriate conditions, quantitative bounds in various metrics between $\pi$ and $\pi_{\gamma}$ can be established; see \cite{durmus2024asymptotic} and the references therein.

Under \Cref{as:MC_mixing} we can now state the following result, with $\bar{f}(\theta) := f(\theta) - \pi_{\gamma}(f)$:
\begin{theorem}
\label{th:concentration_mcmc}
Assume  \Cref{as:MC_mixing}. Then, for any $\gamma \in (0;\bar{\gamma}]$, $\beta \geq 0$, and any measurable function $f: \rset^{d} \to \rset$, $\Vnorm[\log^{\beta}{V}]{f} < \infty$, $p\geq 2$, and $n \geq 3$, it holds that 
\begin{multline}
\label{eq:main_rosenthal_MCMC}
\PE_{\pi_{\gamma}}^{1/p}\bigl[ \bigl|n^{-1}\sum_{\ell=0}^{n-1}\bar{f}(\theta_{\ell})\bigr|^{p}\bigr]  \lesssim  \frac{p^{1/2} \sigma_{\pi_{\gamma}}(f)}{n^{1/2}} + \frac{(2\beta)^{\beta} \{\varkappa \pi_{\gamma}(V)\}^{2/p} p^{2+\beta} \log_2(2p) \Vnorm[\log^{\beta}{V}]{f}}{\rme^{\beta}(\gamma n)^{3/4}}  \\
+ \frac{(2\beta/\rme)^{\beta} \{\varkappa \pi_{\gamma}(V)\}^{1/p} p^{2+\beta} \log_2(2p) \log^{\beta+1}(n) \Vnorm[\log^{\beta}{V}]{f}}{\gamma n}\eqsp,
\end{multline}
where $\sigma_{\pi_{\gamma}}(f)$ is defined in \eqref{eq:asympt_var_def} replacing $\pi$ by $\pi_\gamma$ and $\lesssim$ stands for inequality up to a constant, not depending upon $p$, $\gamma$, $n$, $\beta$ and $V$. 
\end{theorem}
The statement of \Cref{th:concentration_mcmc} directly follows from \Cref{theo:rosenthal_VGE} and the bound on mixing time from \Cref{as:MC_mixing}. Note that the moment bound above is natural in a sense that it highlights the dependence of $\PE_{\pi}^{1/p}\bigl[\bigl| \Sstat_n / n \bigr|^{p}\bigr]$ on the \emph{physical integration time} of the diffusion, that is,   the horizon  $\gamma n =: T$ reached after $n$ iterations of \eqref{eq:langevin_chain} with step size $\gamma$. Note that the remainder terms of the right-hand side of \eqref{eq:main_rosenthal_MCMC} scales homogeneously with respect to $\gamma n$, which resembles the corresponding continous-time results \cite{trottner2023concentration}.


\section{Proofs}
\label{sec:proofs}

\subsection{Outline of the proof technique} We start this section with an outline of the proof based on dyadic decomposition. We provide this proof outline under \Cref{ass:UGE}, and, notably, proofs of analogous results under \Cref{ass:VGE} or \Cref{assum:wasserstein-convergence} follow the same pattern, with small additional technicalities related to the control of solution to the Poisson equation associated with $\MKQ$ (see \eqref{eq:Poisson} below) in appropriate norms. In this subsection, for sequences $\{a_n\}_{n \in \nset}$ and $\{b_n\}_{n \in \nset}$ we write $a_n \lesssim b_n$, if there is an absolute constant $c > 0$, such that $a_n \leq c b_n$ for any $n \in \nset$.
\par 
Using the blocking technique, it is relatively easy to obtain a version of \eqref{eq:ros_independent_pinelis} with a leading term given by a variance proxy:
\begin{equation}
\label{eq:crude_rosenthal_uge}
\textstyle
\PE_{\pi}^{1/p}\bigl[\bigr| \Sstat_n \bigr|^p\bigr] \leq \ConstD_{\ref{lem:auxiliary_rosenthal},1} n^{1/2}\tmix^{1/2}p^{1/2}\infnorm{f} + \ConstD_{\ref{lem:auxiliary_rosenthal},2} \tmix p \infnorm{f}\eqsp.
\end{equation}
Proof is provided in \Cref{lem:auxiliary_rosenthal}. This blocking technique is not applicable to get the exact leading term $\sigma_{\pi}(f)$. Instead, we proceed with a proof strategy based on a martingale representation of $\Sstat_n$ via the Poisson equation. We repeat the same decomposition with appropriate technical modifications under \Cref{ass:VGE} or \Cref{assum:wasserstein-convergence}. Given \Cref{ass:UGE} and a bounded measurable function $f: \Zset \mapsto \rset$, $\infnorm{f} \leq 1$, we define a function $g: \Zset \mapsto \rset$ as
\begin{equation}
\label{eq:Poisson_equation_definition}
g(z) = \sum_{k=0}^{\infty}\{\MKQ^{k}f(z) - \pi(f)\}\eqsp.
\end{equation}
It is easy to check that $g(z)$ is a solution to the Poisson equation associated with $\MKQ$:
\begin{equation}
\label{eq:Poisson}
g(z) - \MKQ g(z) = f(z) - \pi(f)\eqsp.
\end{equation}
Moreover, as it is shown in \Cref{lem: poisson_sol}, it holds that $\infnorm{g} \lesssim  \tmix \infnorm{f}$. The dependence in $\tmix$ of this upper bound can not be improved. Using the definition of the Poisson solution $g$, we represent $\Sstat_n$ in the following way:
\begin{equation}
\label{eq:martingale_decomp_tv}
\textstyle
\Sstat_n = \sum_{i=0}^{n-1} \barf(Z_i) = \sum_{i=0}^{n-1}\left\{ g(Z_i) - \MKQ g(Z_i) \right\} = \sum_{i=0}^{n-1} \Delta M_i^{(0)} + g(Z_0) - g(Z_n)\eqsp,
\end{equation}
where we set $\Delta M_i^{(0)} = g(Z_{i+1}) - \MKQ g(Z_{i})$. The leading term in the sum \eqref{eq:martingale_decomp_tv} is given by $\sum_{i=0}^{n-1} \Delta M_i^{(0)}$. Indeed, using Minkowski's inequality and \Cref{lem: poisson_sol}, we get
\begin{equation}
\label{eq:M_0_decomp}
\textstyle \PE_{\pi}^{1/p}\bigl[\bigl| \Sstat_n \bigr|^{p}\bigr] \leq \PE_{\pi}^{1/p}\bigl[\bigl| \sum_{i=0}^{n-1}\Delta M_i^{(0)} \bigr|^{p} \bigr] + 2\infnorm{g} \lesssim
\PE_{\pi}^{1/p}\bigl[\bigl|\sum_{i=0}^{n-1} \Delta M_i^{(0)}\bigr|^{p} \bigr] + \tmix \infnorm{f} \eqsp.
\end{equation}
Moreover, the sequence $\{\Delta M_i^{(0)}\}_{i=0}^{n-1}$ is a martingale-increment sequence w.r.t. the filtration $\mcf_{i} = \sigma(Z_j, j \leq i)$. Thus one can proceed with the Pinelis version of Rosenthal inequality for martingales \cite[Theorem 4.1]{pinelis_1994}, which implies
\begin{multline}
\label{eq:rosenthal_pinelis}
\textstyle \PE_{\pi}^{1/p}\bigl[\bigl| \sum_{i=0}^{n-1} \Delta M_i^{(0)} \bigr|^{p} \bigr] \leq \bConst{\sf{Rm}, 1} p^{1/2} \PE_{\pi}^{1/p}\bigl[\bigl|\sum_{i=0}^{n-1}
\CPE[\pi]{ (\Delta M_i^{(0)})^2}{\mathcal{F}_i}\bigr|^{p/2}\bigr] \\
+ \bConst{\sf{Rm}, 2} p \PE_{\pi}^{1/p}\bigl[ \bigl\{\underset{i}{\max}|\Delta M_i^{(0)}| \bigr\}^{p}\bigr]\eqsp. 
\end{multline}
Here $\bConst{\sf{Rm}, 1}$ and $\bConst{\sf{Rm}, 2}$ are the  constants from the martingale Rosenthal inequality. Optimal constants are discussed in \cite[Theorem 4.1]{pinelis_1994}, and are provided in \Cref{sec:constants}. The second term above is controlled by $\infnorm{g}$ and its upper bound is also proportional to $\tmix$:
\[
\textstyle{ 
\PE_{\pi}^{1/p}\bigl[\bigl\{\underset{i}{\max}|\Delta M_i^{(0)}| \bigr\}^{p}\bigr] \lesssim \tmix \infnorm{f}}\eqsp.
\]
The control of this term is elementary under \Cref{ass:UGE}. Under \Cref{ass:VGE} or \Cref{assum:wasserstein-convergence}, the analysis of this term is slightly more involved. Still, it is negligible compared to the first one in \eqref{eq:rosenthal_pinelis}.
\par
Now we proceed with the first term in \eqref{eq:rosenthal_pinelis}. Using the Markov property, it is easy to see that $\CPE[\pi]{ (\Delta M_i^{(0)})^2}{\mcf_i} = g_1(Z_i)$, $\PP_\pi$-a.s.,
where $g_1(z)= Q g^2(z) - \{Qg(z)\}^2$. Furthermore, it is important to highlight that $\pi(g_1) \neq 0$ since
\[
\textstyle{
\pi(g_1) = \pi((g-\MKQ g)(g+ \MKQ g)) \overset{(a)}{=} \pi(\bar{f}(g + \MKQ g)) = \pi(\bar{f}^2)  + 2\sum_{\ell = 1}^{\infty} \pi(\barf \MKQ^{\ell} \barf) = \sigma^2_{\pi}(f)
}\eqsp.
\]
In (a) we used the definition of the Poisson equation \eqref{eq:Poisson}. Using Minkowski's inequality,
\[
\textstyle \PE_{\pi}^{2/p}\bigl[\bigl|\sum_{i=0}^{n-1} g_1(Z_i) \bigr|^{p/2}\bigr] \leq \PE_{\pi}^{2/p}\bigl[\bigl|\sum_{i=0}^{n-1}\{g_1(Z_i) - \pi(g_1) \}\bigr|^{p/2}\bigr] + n \sigma^2_{\pi}(f)\eqsp.
\]
Combining the above results in \eqref{eq:rosenthal_pinelis} and setting $\bar{g}_1(z) = g_1(z) - \pi(g_1)$, we get
\begin{equation}
\label{eq:rosenthal_pinelis_grouped}
\textstyle \PE_{\pi}^{1/p}\bigl[\bigl| \Sstat_n \bigr|^{p}\bigr] \lesssim  p^{1/2} n^{1/2} \sigma_{\pi}(f)
+  p^{1/2} \PE_{\pi}^{1/p}\bigl[\bigl|\sum\nolimits_{i=0}^{n-1} \bar{g}_1(Z_i) \}\bigr|^{p/2}\bigr] + p \tmix \infnorm{f}  \eqsp.
\end{equation}
It remains to bound the term 
\begin{equation}
\label{eq:T_1_term_def}
T_1 = \PE_{\pi}^{1/p}\bigl[\bigl|\sum_{i=0}^{n-1} \bar{g_1}(Z_i)\bigr|^{p/2}\bigr]\eqsp.
\end{equation}
It is important to note that while $0 \leq \pi(g_1) \lesssim \tmix$ (see \eqref{eq:bound_pi_g_k_h_k} below), its infinite norm scales \emph{quadratically} with $\tmix$: $\infnorm{g_1} \lesssim \tmix^2 \infnorm{f}^2$. Thus, despite $T_1$ is a remainder term in $n$, we can not simply bound it using the simple version of the Rosenthal inequality from \Cref{lem:auxiliary_rosenthal} if we aim to get a sharp dependence with respect to $\tmix$. Indeed, applying \Cref{lem:auxiliary_rosenthal} directly yields
\[
T_1^2 \lesssim n^{1/2} p^{1/2} \tmix^{5/2} \infnorm{f}^2 + \tmix^{2} p \infnorm{f}^2\eqsp,
\]
and \eqref{eq:rosenthal_pinelis_grouped} writes as
\begin{equation}
\label{eq:rosenthal_inhomogeneous}
\textstyle \PE_{\pi}^{1/p}\bigl[\bigl| \Sstat_n \bigr|^{p}\bigr] \lesssim n^{1/2} p^{1/2} \sigma_{\pi}(f) + n^{1/4} p^{3/4} \tmix^{5/4} \infnorm{f} + \tmix p \infnorm{f}\eqsp.
\end{equation}
The bound \eqref{eq:rosenthal_inhomogeneous} has unsatisfactory scaling with $\tmix$. Indeed, renormalizing \eqref{eq:rosenthal_inhomogeneous} by $n$, we observe that having $\PE_{\pi}^{1/p}\bigl[\bigl| \Sstat_n / n \bigr|^{p}\bigr] \leq \varepsilon$ requires a number of samples $n = \mathcal{O}(\tmix^{5/3})$, which is not linear w.r.t. $\tmix$. Solving the Poisson equation associated with $g_1$ and then applying a second time the martingale decomposition to $T_1$ leads to the same conclusion. This is again linked with the discrepancy in the control of $\sigma^2_\pi(g_1)$ and $\infnorm{g_1}$.

\paragraph{Dyadic descent} First note that the term $T_1^2$ has the exact form of the initial term, replacing $f$ with $\bar{g}_1$ and $p$ with $p/2$. This hints toward the possibility of employing a dyadic descent. For simplicity, we assume from now on that $p=2^s$ for some $s \in \nset^*$. For  $k \in \{1,\dots,s-1\}$, we define the quantities
\begin{equation}
\label{eq:h_k_g_k_def}
\textstyle{
g_k =  \MKQ g^{2^k} - (\MKQ g)^{2^k} \eqsp, \qquad \bar{g}_k = g_k - \pi({g}_k) \eqsp, \qquad R_{k,s}^2 = \PE_{\pi}^{2^{k-s}}\bigl[\bigl| \sum_{i=0}^{n-1} \bar{g}_k(Z_i) \bigr|^{2^{s-k}}\bigr]}\eqsp.
\end{equation}
Note that the term $R_{1,s}$ exactly coincides with $T_1$. At the same time, it can be easily noticed that the term $R_{s-1,s}^4$ can be straightforwardly bounded, since it is just the variance of a linear statistics. Thus our aim will be to carefully bound $R_{k,s}$ in terms of $R_{k+1,s}$. In the key \Cref{lem:recurrence-R-k}, we derive with exact constants the induction
\begin{equation}
\label{eq:recurrent_equation_main}
R_{k,s} \lesssim R^{1/2}_{k+1,s} + n^{1/4} \tmix^{2^{k-1} - 1/4}  + \tmix^{2^{k-1}}\eqsp.
\end{equation}
Towards this aim, with the function $h_k$ defined for $k \in \{1,\dots,s-1\}$ by
\begin{equation}
    \label{eq:def_h_k}
    h_k = g^{2^k} - (\MKQ g)^{2^k} \eqsp,
\end{equation}
we note that $\pi(h_k)= \pi(g_k)$, which implies
\begin{equation}
    \label{eq:defcomp_bar_g}
\bar{g}_k = \MKQ g^{2^k} - g^{2^k} + g^{2^k} - (\MKQ g)^{2^k} - \pi(g_k)  = \MKQ g^{2^k} - g^{2^k} + h_k - \pi(h_k)\eqsp.
\end{equation}
Using Minkowski's inequality, we get
\begin{equation}
\label{eq:decomposition of R_k}
\textstyle
R_{k,s}^2 \leq \underbrace{\PE_{\pi}^{2^{k-s}}\bigl[\bigl| \sum\nolimits_{i=0}^{n-1} g^{2^k}(Z_i) - \MKQ g^{2^k}(Z_i)\bigr|^{2^{s-k}}\bigr]}_{T_{\text{mart}}} + \underbrace{\PE_{\pi}^{2^{k-s}}\bigl[\bigl| \sum\nolimits_{i=0}^{n-1} h_k(Z_i) - \pi(h_k) \bigr|^{2^{s-k}}\bigr]}_{T_{\text{lin}}}\eqsp.
\end{equation}
 We aim to estimate both summands in the decomposition \eqref{eq:decomposition of R_k} starting with $T_{\text{lin}}$. To this end, we first  establish preliminaries properties on $h_k$. Note that for all $k \in \nset$, we have
\begin{equation}
\label{eq:smart-decomposition-hk}
h_k = (g - \MKQ g)(g + \MKQ g)(g^2 + (\MKQ g)^2)\cdots (g^{2^{k-1}} + (\MKQ g)^{2^{k-1}})\eqsp, \eqsp \infnorm{h_k}  \leq 2^{k-1}\infnorm{g}^{2^k-1} \infnorm{\bar{f}}\eqsp.
\end{equation}
For the last bound we have used that $g - \MKQ g = \bar{f}$. From \Cref{lem: poisson_sol} it follows that $\infnorm{h_k} \lesssim \tmix^{2^k-1}$. Note that considering $h_k$ and the decomposition \eqref{eq:defcomp_bar_g}, we gain a factor $\tmix$ as compared to using the definition \eqref{eq:h_k_g_k_def} since  $\infnorm{g_k} \lesssim \tmix^{2^k}$. The same observation applies to $\pi(g_k)$ since
\begin{equation}
\label{eq:bound_pi_g_k_h_k}
    \pi(g_k) = \pi(h_k) \lesssim \tmix^{2^k-1} \eqsp.
\end{equation}
Applying \eqref{eq:crude_rosenthal_uge} with $p=2^{s-k}$ we get, omitting the terms not depending upon $n$ and $\tmix$, that
\begin{equation}
    \label{eq:bound_t_lin_skectch}
    T_{\text{lin}} \lesssim n^{1/2}\tmix^{2^k - 1/2} + \tmix^{2^k}\eqsp.
\end{equation}
Now we upper bound the martingale-difference term in \eqref{eq:decomposition of R_k}. Using Minkowski's inequality, we obtain, using again \Cref{lem: poisson_sol}, that
\begin{equation}
\label{eq:get martingal of R_k_main}
\textstyle
\PE_{\pi}^{2^{k-s}}\bigl[\bigl| \sum_{i=0}^{n-1} g^{2^k}(Z_i) - \MKQ g^{2^k}(Z_i)\bigr|^{2^{s-k}}\bigr]
\lesssim \PE_{\pi}^{2^{k-s}}\bigl[\bigl| \sum_{i=0}^{n-1} \Delta M_i^{(k)} \bigr|^{2^{s-k}}\bigr] + \tmix^{2^k} \eqsp,
\end{equation}
where we have set $\Delta M_i^{(k)} = g^{2^k}(Z_{i+1}) - \MKQ g^{2^k}(Z_i)$. Since $\CPE[\pi]{\Delta M_i^{(k)}}{\mathcal{F}_i} = 0$,
\cite[Theorem 4.1]{pinelis_1994} implies that
\[
\textstyle
\PE_{\pi}^{2^{k-s}}\bigl[\bigl| \sum_{i=0}^{n-1} \Delta M_i^{(k)} \bigr|^{2^{s-k}}\bigr] \lesssim \PE_{\pi}^{2^{k-s}}\bigl[ \bigl| \sum_{j=0}^{n-1} \PE^{\mathcal{F}_j}_{\pi}[(\Delta M_j^{(k)})^2]\bigr|^{2^{s-k-1}}\bigr] + \PE_{\pi}^{2^{k-s}}\bigl[ \bigl\{ \underset{i}{\max} |\Delta M_i^{(k)}| \bigr\}^{2^{s-k}}\bigr]\eqsp.
\]
Using Markov's property and Jensen's inequality,
\begin{equation*}
0 \leq \PE^{\mathcal{F}_j}_{\pi}[(\Delta M_j^{(k)})^2] = \MKQ g^{2^{k+1}}(Z_j) - (\MKQ g^{2^k}(Z_j))^2 \leq \MKQ g^{2^{k+1}}(Z_j) - (\MKQ g(Z_j))^{2^{k+1}} = g_{k+1}(Z_j)\eqsp.
\end{equation*}
Using Minkowski's inequality, \eqref{eq:h_k_g_k_def} and \eqref{eq:bound_pi_g_k_h_k} we finally obtain
\begin{equation*}
\textstyle
\PE_{\pi}^{2^{k-s}}\bigl[ \bigl| \sum_{j=0}^{n-1} g_{k+1}(Z_j)  \bigr|^{2^{s-k-1}}\bigr] \lesssim n^{1/2}\tmix^{2^k -1/2} + R_{k+1,s}\eqsp.
\end{equation*}
Therefore, using that $\bigl|\Delta M_j^{(k)}\bigr| \leq 2\infnorm{g^{2^k}}$, and combining the bounds above in \eqref{eq:decomposition of R_k}, we obtain \eqref{eq:recurrent_equation_main}. Note that the r.h.s. of \eqref{eq:recurrent_equation_main} is homogeneous w.r.t. the ratio $n/\tmix$, as it was desired. To conclude it remains to bound $R_{s-1,s}$ using \Cref{lem:bound_of_second_moment}. This bound is homogeneous w.r.t. $n/\tmix$: omitting the terms not depending on $n$ and $\tmix$, we get
\begin{equation}
\label{eq:bound_R_s-1_main}
\textstyle{
R_{s-1,s}^2 \lesssim n^{1/2} \tmix^{2^{s-1}-1/2} + \tmix^{2^{s-1}}
}\eqsp.
\end{equation}
Now, combining the bounds above, and accounting for the multiplicative terms depending upon $s$, we get from \eqref{eq:rosenthal_pinelis_grouped} for $p = 2^{s}$, that
\begin{equation}
\label{eq:ros_main_simplified}
\textstyle{
\PE_{\pi}^{1/p}\bigl[\bigl| \Sstat_n \bigr|^p\bigr] \lesssim  p^{1/2} n^{1/2}\sigma_{\pi}(f) + n^{1/4}\tmix^{3/4}p(\log_2p) + \tmix p(\log_2p)
}\eqsp.
\end{equation}


\subsection{Proof of \Cref {theo:rosenthal}}
\label{sec:proof:theo:rosenthal}
We consider first the case $p = 2^{s}$, $s \in \nset$. Then the general result follows from the Lyapunov inequality.  

\begin{lemma}
\label{lem:UGE_dyadic}
Assume \Cref{ass:UGE}. Then, for any bounded measurable function $f: \Zset \rightarrow \rset^d$, $s \in \nset$, and $p = 2^{s}$, it holds that
\begin{equation*}
\PE_{\pi}^{1/p}[| \Sstat_n|^p] \leq \bConst{\sf{Rm}, 1} \sqrt{2}p^{1/2}n^{1/2}\sigma_{\pi}(f) +
\ConstD_{\ref{lem:UGE_dyadic},1} n^{1/4}\tmix^{3/4}p\log_2(p) \infnorm{f} + \ConstD_{\ref{lem:UGE_dyadic},2} \tmix p \log_2(p) \infnorm{f}\eqsp,
\end{equation*}
where the constants 
$\ConstD_{\ref{lem:UGE_dyadic},1}$ and $\ConstD_{\ref{lem:UGE_dyadic},2}$ are given by 
\begin{equation}
\label{eq:ros_uge_const_def}
\ConstD_{\ref{lem:UGE_dyadic},1} = (8/3) (19/3)^{1/2} \bConst{\sf{Rm}, 1} \eqsp, \quad \ConstD_{\ref{lem:UGE_dyadic},2} = (32/3) (\bConst{\sf{Rm}, 2}^{1/2} \bConst{\sf{Rm}, 1} ^2 + \bConst{\sf{Rm}, 2})\eqsp.
\end{equation}
\end{lemma}
\begin{proof}
To simplify the notations used in the proof, we will express all intermediate bounds below in terms of $s$. We proceed following the same arguments outlined in \Cref{sec:uge-chains}. First, we get from \eqref{eq:rosenthal_pinelis_grouped} with explicit constants that 
\begin{equation}
\label{eq:rosenthal_pinelis_grouped_proof}
\textstyle \PE_{\pi}^{2^{-s}}\bigl[\bigl| \Sstat_n \bigr|^{2^{s}}\bigr] \leq \bConst{\sf{Rm}, 1} 2^{s/2} n^{1/2} \sigma_{\pi}(f) 
+ \bConst{\sf{Rm}, 1} 2^{s/2} R_{1,s} + (16/3) \bConst{\sf{Rm}, 2} 2^{s} \tmix + (16/3) \tmix\eqsp.
\end{equation}
Now we aim to recursively bound $R_{j,s}$ in terms of $R_{j+1,s}$ for $j \in \{1,\ldots,s-2\}$. Using \Cref{lem:recurrence-R-k} and then \Cref{cor:technical_lemma} with $\alpha = \bConst{\sf{Rm}, 1}$, $\beta = \gamma = 8\tau/3$, $\kappa_0 = (16/3)^{-1/2} \tmix^{-1/4} n^{1/4} \ConstD_{\ref{lem:recurrence-R-k},1}$ and $\kappa_1 = \ConstD_{\ref{lem:recurrence-R-k},2}$, we finally get that for $n \geq \tmix$, it holds
\begin{multline}
\label{eq:bound of R_1}
R_{1,s} \leq \bConst{\sf{Rm}, 1} 2^{s/2} R_{s-1,s}^{1/2^{s-2}} + (8/3)(19/3)^{1/2} \bConst{\sf{Rm}, 1}  n^{1/4}\tmix^{3/4}2^{s/2}(s-2) \\
+ (8/3) \bConst{\sf{Rm}, 1} (9\bConst{\sf{Rm}, 2})^{1/2} \tmix 2^{s/2}(s-2)\eqsp.
\end{multline}
Now it remains to bound the term $R_{s-1,s}^2$. From \eqref{eq:decomposition of R_k} and \eqref{eq:get martingal of R_k_main} it follows
\begin{equation*}
\textstyle{
R_{s-1,s}^2 \leq  \PE_{\pi}^{1/2}\bigl[ \bigl| \sum_{i=0}^{n-1} g^{2^{s-1}}(Z_{i+1}) - \MKQ g^{2^{s-1}}(Z_i) \bigr|^{2} \bigr] + 2(8 \tmix /3)^{2^{s-1}} + \PE_{\pi}^{1/2}\bigl[\bigl| \sum_{i=0}^{n-1} h_{s-1}(Z_i) - \pi(h_{s-1}) \bigr|^{2}\bigr]
}\eqsp.
\end{equation*}
Note that if $i > j$, it holds that
\[
\PE_{\pi}\left[(g^{2^{s-1}}(Z_{i+1}) - \MKQ g^{2^{s-1}}(Z_i))(g^{2^{s-1}}(Z_{j+1}) - \MKQ g^{2^{s-1}}(Z_j))\right] = 0\eqsp.
\]
Therefore, since $\pi \MKQ = \pi$, we obtain
\begin{multline*}
\textstyle{
\PE_{\pi}\bigl[\bigl| \sum_{i=0}^{n-1} g^{2^{s-1}}(Z_{i+1}) - \MKQ g^{2^{s-1}}(Z_i)\bigr|^{2}\bigr] = n\PE_{\pi}\bigl[\bigl|g^{2^{s-1}}(Z_1) - \MKQ g^{2^{s-1}}(Z_0)\bigr|^2 \bigr]} \\
\textstyle{ \leq n\PE_{\pi}\left[ g^{2^s}(Z_0) - (\MKQ g(Z_0))^{2^s}\right] \leq n 2^{s}(8 \tmix /3)^{2^s-1}
}\eqsp.
\end{multline*}
Using \eqref{eq:bound infnorm h} and \Cref{lem:bound_of_second_moment}, we get
\begin{equation}
\label{eq:bound of R_s-1}
R_{s-1,s}^2 \leq (3/2) n^{1/2} 2^{s/2} (8 \tmix /3)^{2^{s-1}-1/2} + 2(8 \tmix /3)^{2^{s-1}}\eqsp.
\end{equation}
Combining \eqref{eq:bound of R_1} and \eqref{eq:bound of R_s-1}, we obtain for $n\geq \tmix$:
\begin{equation}
\label{eq: final bound R_1}
\textstyle{
R_{1,s} \leq (8/3) (19/3)^{1/2} \bConst{\sf{Rm}, 1} n^{1/4}\tmix^{3/4} 2^{s/2} (s-1) + 8 \bConst{\sf{Rm}, 1}  \bConst{\sf{Rm}, 2}^{1/2} \tmix 2^{s/2} (s-1)
}\eqsp.
\end{equation}
It remains to combine \eqref{eq: final bound R_1} and \eqref{eq:rosenthal_pinelis_grouped_proof}, and the statement follows.
\end{proof}

\paragraph{Proof of \Cref {theo:rosenthal} under arbitrary initial distribution} We now prove \eqref{eq:main_text_rosenthal_ksi}. By \cite[Lemma~19.3.6 and Theorem~19.3.9 ]{douc:moulines:priouret:soulier:2018}, for any two probabilities $\xi,\xi'$ on $(\Zset,\Zsigma)$ there is a \emph{maximal exact coupling} $(\Omega,\mathcal{F},\PPcoupling{\xi}{\xi'},Z,Z',T)$ of $\PP ^{\MKQ}_{\xi}$ and $\PP ^{\MKQ}_{\xi'}$, that is,
\begin{equation}
\label{eq:coupling_time_def_markov}
\textstyle{
\tvnorm{\xi \MKQ^n- \xi'\MKQ^n} = 2 \PPcoupling{\xi}{\xi'}(T > n)
}\eqsp.
\end{equation}
We write $\PEcoupling{\xi}{\xi'}$ for the expectation with respect to $\PPcoupling{\xi}{\xi'}$. With the triangle inequality and maximal exact coupling construction \eqref{eq:coupling_time_def_markov}, we obtain that
\begin{equation}
\textstyle{
\PE^{1/p}_{\xi} [|\sum_{i=0}^{n-1} \bar{f}(Z_i)|^{p}]
\leq \PE^{1/p}_{\pi}[|\sum_{i=0}^{n-1} \bar{f}(Z_i)|^{p}] +
\{\PEcoupling{\xi}{\pi}[|\sum_{i=0}^{n-1}\bigl(f(Z_i) - f(Z^{\prime}_i)\bigr)|^{p}]\}^{1/p}
}\eqsp.
\end{equation}
The first term is bounded with \Cref{theo:rosenthal}. Moreover, with \eqref{eq:coupling_time_def_markov} and $\infnorm{f}  \leq 1$, we get
\[
\textstyle{
\left|\sum_{i=0}^{n-1}\bigl(f(Z_i) - f(Z^{\prime}_i)\bigr)\right|^{p} \leq 2^{p}\bigl(\sum_{i=0}^{n-1}\indiacc{Z_i \neq Z^{\prime}_{i}}\bigr)^{p} = 2^{p} \bigl(\sum_{i=0}^{n-1}\indiacc{T > i}\bigr)^{p} \leq 2^{p} T^{p}
}\eqsp.
\]
We obtain combining the previous bounds that
\begin{equation*}
\textstyle{
\PE^{1/p}_{\xi}\left[\left|\sum_{i=0}^{n-1} \bar{f}(Z_i)\right|^{p}\right]
\leq \PE^{1/p}_{\pi}\left[\left|\sum_{i=0}^{n-1} \bar{f}(Z_i)\right|^{p}\right] + 2 \PEcoupling{\xi}{\pi}^{1/p}\bigl[T^{p}\bigr]
}\eqsp.
\end{equation*}
Assumption \Cref{ass:UGE} and triangle inequality imply that, for any $k \in \nset$,
\[ 
\textstyle{
\Delta(Q^{k}) = (1/2)\underset{z, z^{\prime} \in \Zset}{\sup}\tvnorm{Q^k(z, \cdot) - Q^k(z^{\prime}, \cdot)} \leq  \underset{z \in \Zset}{\sup}\tvnorm{Q^k(z, \cdot) - \pi} \leq 8 (1/4)^{k/\tmix}
}\eqsp.
\]
Hence, setting $\rho = (1/4)^{1/\tmix}$, we get
\begin{multline*}
\textstyle \PEcoupling{\xi}{\pi}\bigl[ T^{p}\bigr] = 1 + \sum_{k=2}^{\infty} \{ k^{p} - (k-1)^p \}\,\PPcoupling{\xi}{\pi}\bigl(T > k-1\bigr)  \\ 
\textstyle \leq 1 + \sum_{k=2}^{\infty} \{k^{p} - (k-1)^p\} \Delta(Q^{k-1}) \leq 1+ 8\rho^{-1} \bigl(1-\rho\bigr) \sum_{k=1}^\infty k^p \rho^{k}.
\end{multline*}
Now we use the upper bound, for $\rho \in (0,1)$,
\begin{equation*}
\textstyle{
\sum_{k=1}^{\infty}k^{p} \rho^{k}
\leq \rho^{-1}\int_{0}^{+\infty}x^{p}\rho^{x}\,\rmd x \leq \rho^{-1}\left(\ln{\rho^{-1}}\right)^{-p-1} \Gamma(p+1)
}\eqsp.
\end{equation*}
Combining the bounds above and the elementary inequality $1-\rho \leq \ln{\rho^{-1}}$, we obtain
\begin{align}
\PEcoupling{\xi}{\pi}\bigl[ T^{p}\bigr] \leq 1 + 8 \rho^{-2}\left(\ln{\rho^{-1}}\right)^{-p} \Gamma(p+1) \leq 1 + 128 (\tmix/\ln{4})^{p}\Gamma (p+1)\eqsp.
\end{align}
To complete the proof, we use an upper bound $\Gamma(p+1) \leq (p+1)^{p+1/2}e^{-p}$ due to \cite[Theorem 2]{guo:bounds_for_gamma_function} and apply an elementary inequality $(p+1)^{1/2} \leq 2^{p/2}$.

\subsection{Auxiliary lemmas for \Cref{theo:rosenthal}}
We establish a first Rosenthal inequality in which the leading term with respect to $n$ is not governed by the asymptotic variance $\sigma_{\pi}(f)$ defined in \eqref{eq:asympt_var_def}.
\begin{lemma}
\label{lem:auxiliary_rosenthal}
Under \Cref{ass:UGE}, for any bounded measurable function $f: \Zset \rightarrow \rset$, and $p \geq 2$,
\begin{equation*}
\PE_{\pi}^{1/p}\bigl[\bigr| \Sstat_n \bigr|^p\bigr] \leq \ConstD_{\ref{lem:auxiliary_rosenthal},1} n^{1/2}\tmix^{1/2}p^{1/2} \|f\|_{\infty} + \ConstD_{\ref{lem:auxiliary_rosenthal},2} \tmix p \|f\|_{\infty}\eqsp,
\end{equation*}
where $\ConstD_{\ref{lem:auxiliary_rosenthal},1} = (16/3) \bConst{\sf{Rm}, 1}$, $\ConstD_{\ref{lem:auxiliary_rosenthal},2} = 8 \bConst{\sf{Rm}, 2}$.
\end{lemma}
\begin{proof}
Without loss of generality, assume that  $\|f\|_{\infty} \leq 1$. Using Minkowski's inequality, 
\begin{equation*}
\textstyle
\PE_{\pi}^{1/p}\bigl[\bigr| \Sstat_n \bigr|^p\bigr] \leq  \PE_{\pi}^{1/p}\bigl[\bigl| \sum_{i=0}^{\lfloor n/\tmix \rfloor \tmix-1}\barf(Z_i) \bigr|^p\bigr] + 2\tmix \eqsp.
\end{equation*}
Note that $g_\tmix = \sum_{i = 0}^{+\infty}\MKQ^{i\tmix} \bar{f} $ is well-defined under \Cref{ass:UGE}. Moreover, $\infnorm{g_\tmix} \leq 8/3$, and $g_\tmix$ is a solution of the Poisson equation associated with the $\tmix$-th iterate of $\MKQ$, that is, 
\[
g_{\tmix}(z) - \MKQ^{\tmix}g_{\tmix}(z) = \bar{f}(z) \eqsp.
\]
Define $q := \lfloor n/\tmix \rfloor$, then we have
\begin{equation}
\label{eq:decomposition}
\textstyle{
\sum_{i=0}^{q\tmix - 1}\barf(Z_i) = \sum_{r = 0}^{\tmix-1} B_{\tmix,r} \quad \text {with} \quad B_{\tmix,r}= \sum_{k = 0}^{q-1} \bigl\{g_\tmix(Z_{k\tmix+r}) - Q^{\tmix}g_\tmix(Z_{k\tmix+r})\bigr\}
}\eqsp.
\end{equation}
Using Minkowski's inequality, we get $\PE_{\pi}^{1/p}\bigl[\bigl| \sum_{i=0}^{q \tmix-1}\barf(Z_i) \bigr|^p\bigr] \leq \sum_{r=0}^{\tmix-1} b_{\tmix,r,p}$, where
\begin{align}
\label{eq:bloks}
\textstyle{
b_{\tmix,r,p} = \PE_{\pi}^{1/p} [|B_{\tmix,r}|]^p \leq \PE_{\pi}^{1/p}\bigl[ \bigl| \sum_{k=1}^{q} \left\{ g_\tmix(Z_{k\tmix+r}) - Q^{\tmix}g_\tmix(Z_{(k-1)\tmix+r}) \right\} \bigr|^p \bigr] + 2 \infnorm{g_\tmix}
}\eqsp.
\end{align}
Applying the Pinelis version of Rosenthal inequality \cite[Theorem~4.1]{pinelis_1994}, 
\begin{equation}
\label{eq:bound-rosenthal}
\begin{split}
&\textstyle
\PE_{\pi}^{1/p}\bigl[\bigl| \sum_{k=1}^{q} \left\{ g_{\tmix}(Z_{k\tmix+r}) - Q^{\tmix}g_{\tmix}(Z_{(k-1)\tmix+r}) \right\} \bigr|^p \bigr] \\
& \qquad  \textstyle \leq \bConst{\sf{Rm}, 2} p \PE_{\pi}^{1/p}\bigl[ \underset{1 \leq i \leq q}{\max} \left| g_{\tmix}(Z_{i\tmix+r}) - Q^{\tmix}g_{\tmix}(Z_{(i-1)\tmix+r}) \right| ^ p \bigr]  \\
&\textstyle
 \qquad \qquad + \bConst{\sf{Rm}, 1} p^{1/2}\PE^{1/p}\bigl[ \bigl( \sum_{j = 1}^{q} \PE_{\pi}^{\mathcal{F}_{\tmix,r,j-1}} \bigl| g_{\tmix}(Z_{j\tmix+r}) - Q^{\tmix}g_{\tmix}(Z_{(j-1)\tmix+r}) \bigr|^2 \bigr)^{p/2} \bigr],
\end{split}
\end{equation}
where $\bConst{\sf{Rm}, 1}$ and $\bConst{\sf{Rm}, 2}$ are absolute constants given in \Cref{sec:constants}, and $\mathcal{F}_{\tmix,r,j}= \sigma( Z_{i \tmix+r}, 0 \leq i \leq j)$. Using that $\bigl|g_{\tmix}(Z_{k\tmix+r}) - Q^{\tmix}g_{\tmix}(Z_{(k-1)\tmix+r})\bigr| \leq 2 \infnorm{g_\tmix}$, we get
\begin{equation*}
b_{\tmix,r,p} \leq 2 \bConst{\sf{Rm}, 1} p^{1/2}q^{1/2}\infnorm{g_\tmix} + 2 \bConst{\sf{Rm}, 2} p\infnorm{g_\tmix} + 2\infnorm{g_\tmix}.
\end{equation*}
Combining the previous inequalities and using that $\bConst{\sf{Rm}, 2} p \geq 2$, we obtain that
\begin{equation*}
\PE_{\pi}^{1/p}\bigl[\bigl| \Sstat_n \bigr|^p\bigr] \leq (16/3) \bConst{\sf{Rm}, 1} p^{1/2}n^{1/2}\tmix^{1/2} + 8 \bConst{\sf{Rm}, 2} p \tmix\eqsp.
\end{equation*}
\end{proof}

\begin{lemma}
\label{lem:recurrence-R-k}
Assume \Cref{ass:UGE}. Then for any bounded measurable function $f: \Zset \rightarrow \rset^d$, $\infnorm{f} \leq 1$, any $s \in \nset$ and $k \in \{1,\ldots,s-1\}$,
\begin{multline}
\label{eq:recurrent equation}
R_{k,s} \leq \bConst{\sf{Rm}, 1}^{1/2} 2^{(s-k)/4} R^{1/2}_{k+1,s} \\
+ (16/3)^{-1/2} \tmix^{-1/4} \ConstD_{\ref{lem:recurrence-R-k},1} n^{1/4} (8\tmix/3)^{2^{k-1}} 2^{(s+k)/4} +
\ConstD_{\ref{lem:recurrence-R-k},2} 2^{s/2} (8\tmix/3)^{2^{k-1}}  \eqsp,
\end{multline}
where $\ConstD_{\ref{lem:recurrence-R-k},1} = (19/3)^{1/2}\bConst{\sf{Rm}, 1}^{1/2}$ and $\ConstD_{\ref{lem:recurrence-R-k},2} = 3\bConst{\sf{Rm}, 2}^{1/2}$.
\end{lemma}
\begin{proof}
We aim to estimate both summands in the decomposition \eqref{eq:decomposition of R_k}. Note that for $k \in \nset$, using \eqref{eq:smart-decomposition-hk}, $\infnorm{h_k}  \leq 2^{k-1}\infnorm{g}^{2^k-1}$
where we have used that $g - \MKQ g = f$. From \Cref{lem: poisson_sol} it follows
\begin{equation}
    \label{eq:bound infnorm h}
    \infnorm{h_k} \leq 2^{k-1}(8/3)^{2^k - 1}\tmix^{2^k-1}.
\end{equation}
\Cref{lem:auxiliary_rosenthal} with $p = 2^{s-k}$ implies that
\begin{equation*}
\textstyle
\PE_{\pi}^{2^{k-s}}\bigl[\bigl| \sum_{i=0}^{n-1} h_k(Z_i) - \pi(h_k) \bigr|^{2^{s-k}}\bigr] \leq \ConstD_{\ref{lem:auxiliary_rosenthal},1} n^{1/2}\tmix^{2^k - 1/2} 2^{(s+k)/2 - 1} (8/3)^{2^k-1} + \ConstD_{\ref{lem:auxiliary_rosenthal},2} \tmix^{2^k}2^{s-1}(8/3)^{2^k-1}.
\end{equation*}
Now we upper bound the martingale-difference term in \eqref{eq:decomposition of R_k}. Using Minkowski's inequality, we obtain, using again \Cref{lem: poisson_sol}, that
\begin{equation}
\label{eq:get martingal of R_k}
\textstyle
\PE_{\pi}^{2^{k-s}}\bigl[\bigl| \sum_{i=0}^{n-1} g^{2^k}(Z_i) - \MKQ g^{2^k}(Z_i)\bigr|^{2^{s-k}}\bigr]
\leq \PE_{\pi}^{2^{k-s}}\bigl[\bigl| \sum_{i=0}^{n-1} \Delta M_i^{(k)} \bigr|^{2^{s-k}}\bigr] + 2(8/3)^{2^k}\tmix^{2^k} \eqsp,
\end{equation}
where we have set $\Delta M_i^{(k)} = g^{2^k}(Z_{i+1}) - \MKQ g^{2^k}(Z_i)$. Since $\CPE[\pi]{\Delta M_i^{(k)}}{\mathcal{F}_i} = 0$,
\cite[Theorem 4.1]{pinelis_1994} implies that
\begin{multline*}
\textstyle
\PE_{\pi}^{2^{k-s}}\bigl[\bigl| \sum_{i=0}^{n-1} \Delta M_i^{(k)} \bigr|^{2^{s-k}}\bigr] \leq \bConst{\sf{Rm}, 1} 2^{(s-k)/2}\PE_{\pi}^{2^{k-s}}\bigl[ \bigl| \sum_{j=0}^{n-1} \PE^{\mathcal{F}_j}_{\pi}[(\Delta M_j^{(k)})^2]\bigr|^{2^{s-k-1}}\bigr]  \\ + \bConst{\sf{Rm}, 2} 2^{s-k} \PE_{\pi}^{2^{k-s}}\bigl[ \bigl\{ \underset{i}{\max} |\Delta M_i^{(k)}| \bigr\}^{2^{s-k}}\bigr].
\end{multline*}
Using Markov's property and Jensen's inequality,
\begin{equation*}
0 \leq \PE^{\mathcal{F}_j}_{\pi}[(\Delta M_j^{(k)})^2] = \MKQ g^{2^{k+1}}(Z_j) - (\MKQ g^{2^k}(Z_j))^2 \leq \MKQ g^{2^{k+1}}(Z_j) - (\MKQ g(Z_j))^{2^{k+1}} = g_{k+1}(Z_j)\eqsp.
\end{equation*}
Using Minkowski's inequality, $\pi(h_{k+1}) = \pi(g_{k+1})$, we finally obtain
\begin{equation}
\label{eq:bound_2nd_moment_uge}
\begin{split}
\textstyle
\PE_{\pi}^{2^{k-s}}\bigl[ \bigl| \sum_{j=0}^{n-1} g_{k+1}(Z_j)  \bigr|^{2^{s-k-1}}\bigr] 
& \textstyle \leq n^{1/2}\infnorm{h_{k+1}}^{1/2} + R_{k+1,s} \\
& \textstyle \leq n^{1/2}\tmix^{2^k - 1/2 }2^{k/2}(8/3)^{2^k-1/2} + R_{k+1,s}\eqsp.
\end{split}
\end{equation}
Therefore, using  $\bigl|\Delta M_j^{(k)}\bigr| \leq 2\infnorm{g^{2^k}}$, \Cref{lem: poisson_sol}, combining the bounds above in \eqref{eq:decomposition of R_k} and taking the square root, we obtain \eqref{eq:recurrent equation}.
\end{proof}

\subsection{Auxiliary lemmas for $V$-uniform geometrically ergodic case.}
We give in the next result the counterpart of  \Cref{lem:auxiliary_rosenthal} under \Cref{ass:VGE}, i.e., a simple version of Rosenthal inequality  which does not include the variance term.
\begin{lemma}
\label{lem: auxiliary_rosenthal_VGE}
Assume \Cref{ass:VGE}. Then for any $p\geq 2$, $n \geq \tmix$ and $f: \Zset \rightarrow \rset$ with $\Vnorm[V^{1/p}]{f} < \infty$,
\begin{equation*}
\PE_{\pi}^{1/p}\bigl[\bigl| \Sstat_n \bigr|^p\bigr] \leq  \ConstD_{\ref{lem: auxiliary_rosenthal_VGE},1} p^{3/2} n^{1/2} \tmix^{1/2} \{\varkappa \pi(V)\}^{1/p} \Vnorm[ V^{1/p}]{f} + \ConstD_{\ref{lem: auxiliary_rosenthal_VGE},2} p^2n^{1/p}\tmix^{1-1/p}\{\varkappa \pi(V)\}^{1/p}\Vnorm[ V^{1/p}]{f}\eqsp,
\end{equation*}
where $\ConstD_{\ref{lem: auxiliary_rosenthal_VGE},1} = (8/3) \bConst{\sf{Rm}, 1}$, $\ConstD_{\ref{lem: auxiliary_rosenthal_VGE},2} = (16/3) \bConst{\sf{Rm}, 2} + 11/3$.
\end{lemma}
\begin{proof}
The proof is along the same lines as the proof of \Cref{lem:auxiliary_rosenthal}. Without loss of generality we assume that $\Vnorm[V^{1/p}]{f} = 1$. Using Minkowski's inequality,
\begin{equation}
\label{eq: whole_part_VGE}
\textstyle{
\PE_{\pi}^{1/p}\bigl[\bigl| \sum_{i=0}^{n-1}\barf(Z_i) \bigr|^p\bigr] \leq \PE_{\pi}^{1/p}\bigl[\bigl| \sum_{i=0}^{\lfloor n/\tmix \rfloor \tmix-1}\barf(Z_i) \bigr|^p\bigr] + 2\tmix\{\pi(V)\}^{1/p}
}\eqsp.
\end{equation}
As before, we consider the function $g_\tmix = \sum_{i = 0}^{+\infty}\MKQ^{i\tmix} \bar{f}$, which is a solution of the Poisson equation associated with  $\MKQ^{\tmix}$ and $f$. It is well-defined under \Cref{ass:VGE} and \Cref{lem: poisson_sol} shows that
\begin{align}
\label{eq: bound_of_pi(g)_t_mix}
|g_\tmix(z)| &\leq (1/3)\varkappa^{1/p}2^{3-1/p} \{\pi(V) + V(z)\}^{1/p} p \eqsp, \nonumber\\
|\MKQ^{\tmix}g_\tmix(z)| & \leq (1/3)\varkappa^{1/p}2^{3-1/p} \{\pi(V) + V(z)\}^{1/p} p \eqsp, \\
\PE_{\pi}^{1/p}[|g_\tmix(Z_0)|^{p}] &\leq  (8/3) \{\varkappa \pi(V)\}^{1/p} p\eqsp \nonumber.
\end{align}
We use the decomposition \eqref{eq:decomposition}.
Proceeding as in \eqref{eq:bloks}, we get $\PE_{\pi}^{1/p}\bigl[\bigl| \sum_{i=0}^{q \tmix-1}\barf(Z_i) \bigr|^p\bigr] \leq \sum_{r=0}^{\tmix-1} b_{\tmix,r,p}$, where $b_{\tmix,r,p} = \PE_{\pi}^{1/p} [|B_{\tmix,r}|]^p$, 
and it remains to upper bound $b_{\tmix,r,p}$. Applying Minkowski's inequality,
\begin{equation}
\label{eq:b_tmix_bound_v_erg}
\textstyle{
b_{\tmix,r,p} \leq \PE_{\pi}^{1/p}\bigl[\bigl| \sum_{k=1}^{q} \bigl\{ g_\tmix(Z_{k\tmix+r}) - Q^{\tmix}g_\tmix(Z_{(k-1)\tmix+r}) \bigr\} \bigr|^p \bigr] + (16/3) \varkappa^{1/p} \{\pi(V)\}^{1/p} p
}\eqsp.
\end{equation}
Now we use again the Pinelis version of Rosenthal inequality \eqref{eq:bound-rosenthal}. Applying Minkowski's inequality and using \eqref{eq: bound_of_pi(g)_t_mix}, we get
\begin{equation}
\label{eq: bound_max_rosenthal}
    \PE_{\pi}^{1/p}\bigl[ \underset{1 \leq i \leq q}{\max} \bigl| g_\tmix(Z_{i\tmix+r}) - Q^{\tmix}g_\tmix(Z_{(i-1)\tmix+r}) \bigr| ^ p \bigr] \leq (16/3) \varkappa^{1/p}  q^{1/p} \{\pi(V)\}^{1/p} p\eqsp.
\end{equation}
To bound the corresponding martingale-difference term, we note that $\PE_{\pi}^{\mathcal{F}_{\tmix,r,j-1}}\bigl| g_\tmix(Z_{j\tmix+r}) - Q^{\tmix}g_\tmix(Z_{(j-1)\tmix+r}) \bigr|^2 = Q^{\tmix}g_\tmix^2(Z_{(j-1)\tmix+r}) - (Q^{\tmix}g_\tmix(Z_{(j-1)\tmix+r}))^2$, and
\begin{multline*}
\textstyle{
\PE_{\pi}^{2/p}\bigl[ \bigl( \sum_{j = 1}^{q} \PE_{\pi}^{\mathcal{F}_{\tmix,r,j-1}} \bigl| g_\tmix(Z_{j\tmix+r}) - Q^{\tmix}g_\tmix(Z_{(j-1)\tmix+r}) \bigr|^2 \bigr)^{p/2} \bigr]
}\\ 
\textstyle{
\leq \sum_{j=1}^{q}  \PE_{\pi}^{2/p} \bigl[ Q^{\tmix}g_\tmix^2(Z_{(j-1)\tmix+r}) - (Q^{\tmix}g_\tmix(Z_{(j-1)\tmix+r}))^2 \bigr]^{p/2}
}\eqsp.
\end{multline*}
Using the conditional version of Jensen's inequality, we get 
\[
\PE_{\pi}^{2/p} \bigl[ Q^{\tmix}g_\tmix^2(Z_{(j-1)\tmix+r}) - (Q^{\tmix}g_\tmix(Z_{(j-1)\tmix+r}))^2 \bigr]^{p/2}
\leq \PE_{\pi}^{2/p}\bigl[Q^{\tmix}g_\tmix^{p}(Z_{(j-1)\tmix+r})\bigr]
\leq ((8/3)\varkappa^{1/p} \{\pi(V)\}^{1/p} p)^2\eqsp.
\]
Combining the bounds above and using Minkowski's inequality, we obtain
\begin{equation}
\label{eq: bound_sum_rosenthal}
\textstyle{
\PE_{\pi}^{1/p}\bigl[ \bigl( \sum_{j = 1}^{q} \PE_{\pi}^{\mathcal{F}_{\tmix,r,j-1}} \bigl| g_\tmix(Z_{j\tmix+r}) - Q^{\tmix}g_\tmix(Z_{(j-1)\tmix+r}) \bigr|^2 \bigr)^{p/2} \bigr] \leq \tfrac{8}{3} \varkappa^{1/p} q^{1/2} \{\pi(V)\}^{1/p} p
}\eqsp.
\end{equation}
From \eqref{eq:b_tmix_bound_v_erg} and \eqref{eq: bound_sum_rosenthal} it follows, that
\[
b_{\tmix,r,p}
\leq (16/3) \bConst{\sf{Rm}, 2} p^2q^{1/p} \{\varkappa \pi(V)\}^{1/p} + (8/3) \bConst{\sf{Rm}, 1} p^{3/2}q^{1/2}\{\varkappa \pi(V)\}^{1/p}
+ (16/3) p \{\varkappa \pi(V)\}^{1/p}\eqsp.
\]
Finally, substituting $q = \lfloor n/\tmix \rfloor$ and combining the bound above with \eqref{eq: whole_part_VGE}, the result follows.
\end{proof}

\begin{lemma}
\label{lem:recurrence-R-k_VGE}
Assume \Cref{ass:VGE}. Then for any $\beta \geq 0$, $n \geq \tmix$, $f: \Zset \rightarrow \rset$ with $\Vnorm[\log^{\beta}V]{f} < \infty$, and $k \in \{1,\ldots,s-1\}$,
\begin{multline}
\label{eq: recurrent_equation_VGE}
R_{k,s}^2 \leq \bConst{\sf{Rm}, 1}2^{(s-k)/2}R_{k+1, s} + \bConst{\sf{Rm}, 2} 2^{s-k+1}\Bound_{max}(p,k,n,f,\tmix) \\ + 2\Bound_{\operatorname{\operatorname{lin}}, 1}(p,k,n,f,\tmix) +\Bound_{\operatorname{lin}, 2}(p,k,n,f,\tmix) \eqsp,
\end{multline}
where 
\begin{align}
\Bound_{\operatorname{lin}, 1}
&= \ConstD_{\ref{lem: auxiliary_rosenthal_VGE},1} (2^s)^{2^{k} + 1/2}2^{-1-k/2} (n/\tmix)^{1/2}\tmix^{2^k}\{\varkappa \pi(V)\}^{(2^{k+1} -1)2^{-s}}(8/3)^{2^k - 1} \Vnorm[V^{2^{-s}}]{f}^{2^k} \eqsp, \nonumber  \\
\Bound_{\operatorname{\operatorname{lin}}, 2} 
&= \ConstD_{\ref{lem: auxiliary_rosenthal_VGE},2} (2^s)^{2^{k} + 1}2^{-k -1}(n/\tmix)^{2^{k-s}}\tmix^{2^k}\{\varkappa \pi(V)\}^{(2^{k+1} - 1)2^{-s}}(8/3)^{2^k - 1} \Vnorm[V^{2^{-s}}]{f}^{2^k}\eqsp, \\
\label{eq:t_max_bound_def}
\textstyle{
\Bound_{\max}} &= \textstyle{\bigl\{(8 \rme/3) 2^{s} \tmix \{\log{n}\} \{\varkappa \pi(V)\}^{2^{-s}}\Vnorm[V^{1/(2^s \log{n})}]{f}\bigr\}^{2^k}
}\eqsp.
\end{align}
Arguments $p,k,n,f,\tmix$ are implicit in the definition of the above terms.
\end{lemma}
\begin{proof}
Proceeding as in \Cref{lem:recurrence-R-k}, we aim to bound the both terms in the decomposition \eqref{eq:decomposition of R_k}. In order to do this, we need to control the norm of $h_k$ defined in \eqref{eq:smart-decomposition-hk}. Since $g = \sum_{k=0}^{\infty}\MKQ^{k} \bar{f}$, we can bound $\MKQ g(z)$ using \Cref{lem: poisson_sol}-\ref{pois:VGE} applied with $\alpha = 2^{-s}$. Therefore, we get for $i \in \{0,\ldots,k-1\}$ that
\begin{align}
\label{eq: bound_part_of_h_k}
\textstyle{
g^{2^i}(z) + (Qg(z))^{2^i} \leq 2 \left( (2^{3-2^{-s}}/3) \varkappa^{2^{-s}} 2^{s} \{\pi(V) + V(z)\}^{2^{-s}} \tmix \Vnorm[V^{2^{-s}}]{f} \right)^{2^i}
}\eqsp.
\end{align}
Combining 
\eqref{eq:smart-decomposition-hk} and \eqref{eq: bound_part_of_h_k}, we obtain
\begin{equation}
\label{eq: bound_V_norm_of_h_k}
\textstyle{
\Vnorm[V^{2^{k-s}}]{h_k} \leq 2^{k-1}\left( (8/3) 2^{s} \{\varkappa \pi(V)\}^{2^{-s}} \tmix \right)^{2^k - 1} \Vnorm[V^{2^{-s}}]{f}^{2^k}
}\eqsp.
\end{equation}
The application of the crude Rosenthal inequality of \Cref{lem: auxiliary_rosenthal_VGE} implies that 
\[
\textstyle{
\PE_{\pi}^{2^{k-s}}\bigl[\bigl| \sum_{i=0}^{n-1} h_k(Z_i) - \pi(h_k) \bigr|^{2^{s-k}}\bigr] \leq \Bound_{\operatorname{\operatorname{lin}}, 1}(s,k,n,f,\tmix) + \Bound_{\operatorname{\operatorname{lin}}, 2}(s,k,n,f,\tmix)
}\eqsp,
\]
Now we upper bound the martingale-difference term $T_{\text{mart}}$ in \eqref{eq:decomposition of R_k}.
Using Minkowski's inequality and \Cref{lem: poisson_sol}, we obtain that
\begin{multline}
\label{eq: get_martingal_of_R_k_VGE}
\textstyle{
\PE_{\pi}^{2^{k-s}}\bigl[\bigl| \sum_{i=0}^{n-1} g^{2^k}(Z_i) - \MKQ g^{2^k}(Z_i)\bigr|^{2^{s-k}}\bigr]  
} \\ 
\textstyle{
\leq \PE_{\pi}^{2^{k-s}}\bigl[\bigl| \sum_{i=0}^{n-1} g^{2^k}(Z_{i+1}) - \MKQ g^{2^k}(Z_i)\bigr|^{2^{s-k}}\bigr] + 2\bigl\{ (8/3) \varkappa^{2^{-s}} \tmix 2^s\{\pi(V)\}^{2^{-s}} \Vnorm[V^{2^{-s}}]{f}\bigr\}^{2^k}
}\eqsp.
\end{multline}
Setting $\Delta M_i^{(k)} = g^{2^k}(Z_{i+1}) - \MKQ g^{2^k}(Z_i)$ and applying \cite[Theorem 4.1]{pinelis_1994}, we get
\begin{multline*}
\textstyle{
\PE_{\pi}^{2^{k-s}}\bigl[\bigl| \sum_{i=0}^{n-1} \Delta M_i^{(k)} \bigr|^{2^{s-k}}\bigr] \leq \bConst{\sf{Rm}, 1} 2^{(s-k)/2}\PE_{\pi}^{2^{k-s}}\bigl[ \bigl| \sum_{j=0}^{n-1} \PE_{\pi}^{\mathcal{F}_j}\bigl\{ (\Delta M_j^{(k)})^2 \bigr\}\bigr|^{2^{s-k-1}}\bigr]
}\\ 
\textstyle{
+ \bConst{\sf{Rm}, 2}2^{s-k}\PE_{\pi}^{2^{k-s}}\bigl[ \bigl\{ \underset{i}{\max} |\Delta M_i^{(k)}| \bigr\}^{2^{s-k}}\bigr]
}\eqsp.
\end{multline*}
Proceeding as in \eqref{eq:bound_2nd_moment_uge}, we get
\begin{equation*}
\textstyle{
\PE_{\pi}^{2^{k-s}}\bigl[ \bigl| \sum_{j=0}^{n-1} \PE_{\pi}^{\mathcal{F}_j}\bigl\{ (\Delta M_j^{(k)})^2 \bigr\} \bigr|^{2^{s-k-1}}\bigr] \leq n^{1/2}\Vnorm[V^{2^{k+1-s}}]{h_{k+1}}^{1/2}\{\pi(V)\}^{2^{k-s}} + R_{k+1, s}
}\eqsp.
\end{equation*}
Applying the Lyapunov,  Minkowski and Jensen inequality  with $n^{1/(2^{s-k} \log n)} \leq n^{1/\log n} = \rme$, we get
\begin{align}
\textstyle
\PE_{\pi}^{2^{k-s}}\bigl[ \bigl\{ \underset{i}{\max} |\Delta M_i| \bigr\}^{2^{s-k}}\bigr] 
& \textstyle \leq \PE_{\pi}^{1/(2^{s-k}\log{n})}\bigl[ \bigl\{ \underset{i}{\max} |\Delta M_i| \bigr\}^{2^{s-k}\log{n}}\bigr] \label{eq:clever_max_bound} \\
&\textstyle \leq  \PE_{\pi}^{1/(2^{s-k}\log{n})}\bigl[ \bigl\{ \underset{i}{\max} |g^{2^s \log{n}}(Z_i)| \bigr\} \bigr] \leq \Bound_{\max}(p,k,n,f,\tmix)\eqsp. \nonumber
\end{align}
To get the previous inequality we additionally used that $\Vnorm[\log^{\beta} V]{f} < \infty$ implies that $\Vnorm[V^{\alpha}]{f} < \infty$ for any $\alpha > 0$ and then applied the bound of \Cref{lem: poisson_sol} with $\alpha = (2^s \log{n})^{-1}$.
Combining the previous inequalities, we obtain \eqref{eq: recurrent_equation_VGE}.
\end{proof}

\begin{lemma}
\label{lem: rosenthal_VGE_for_2^s}
Assume \Cref{ass:VGE}. Then for any $\beta \geq 0$, measurable function $f: \Zset \rightarrow \rset^d$ with $\Vnorm[\log^{\beta}{V}]{f} < \infty$, $s \in \nsets$, $s \geq 2$, $p = 2^s$, and $n \geq 3$, it holds that
\begin{multline*}
\PE_{\pi}^{1/p}\bigl[\bigl| \Sstat_n \bigr|^{p}\bigr]  \leq  \bConst{\sf{Rm}, 1} p^{1/2}n^{1/2}\sigma_{\pi}(f) + \ConstD_{\ref{lem: rosenthal_VGE_for_2^s}, 1}(\beta/\rme)^{\beta}p^{\beta + 2}n^{1/4}\tmix^{3/4}\{\varkappa \pi(V)\}^{2/p}\Vnorm[\log^{\beta} V]{f}\log_2p \\ + \ConstD_{\ref{lem: rosenthal_VGE_for_2^s}, 2}(\beta / \rme)^{\beta} p^{\beta + 2} \log^{\beta+1}(n)\tmix\{\varkappa \pi(V)\}^{1/p} \Vnorm[\log^{\beta} V]{f}\log_2p\eqsp,
\end{multline*}
where $\sigma_{\pi}(f)$ is defined in \eqref{eq:asympt_var_def} and
\[
\textstyle{\ConstD_{\ref{lem: rosenthal_VGE_for_2^s}, 1} = (8/3) (16/3)^{1/2} \bConst{\sf{Rm}, 1}^2\bigl(\bConst{\sf{Rm}, 2}^{1/2} + 1 \bigr)\eqsp, \quad 
\ConstD_{\ref{lem: rosenthal_VGE_for_2^s}, 2} = (8\rme / 3) \bigl(\bConst{\sf{Rm}, 1}^2 \bConst{\sf{Rm}, 2}^{1/2} \sqrt{2} + \bConst{\sf{Rm}, 2}\bigr)}\eqsp.
\]
\end{lemma}
\begin{proof}
Suppose first that $n < \tmix$. Then Minkowski's inequality yield
\[
\PE_{\pi}^{1/p}\bigl[\bigl| \Sstat_n \bigr|^{p}\bigr] \leq 2\tmix \{ \pi(V)\}^{1/p} \Vnorm[V^{1/p}]{f}\eqsp,
\]
and the bound follows from \Cref{lem:small_technical_bound} applied with $\alpha = 1/p$. Now we assume that $n \geq \tmix$. With $g$ defined in \eqref{eq:Poisson_equation_definition}, we proceed as in \eqref{eq:martingale_decomp_tv} and apply the upper bound \eqref{eq:v_alpha_log_s_bound} with $\alpha = 2^{-s}$:
\begin{equation*}
\begin{split}
\textstyle 
\PE_{\pi}^{2^{-s}}\bigl[\bigl| \sum_{i=0}^{n-1}\barf(Z_i)\bigr|^{2^s}\bigr]
& \textstyle{\leq \PE_{\pi}^{2^{-s}}\bigl[ \bigl| \sum_{i=0}^{n-1}\Delta M_i^{(0)} \bigr|^{2^s}\bigr] + 2\PE_{\pi}^{2^{-s}}\bigl[g(Z_0)^{2^{s}}\bigr]} \\
& \textstyle{\leq \PE_{\pi}^{2^{-s}}\bigl[ \bigl| \sum_{i=0}^{n-1}\Delta M_i^{(0)} \bigr|^{2^s} \bigr] + (16/3) 2^{s} \tmix \{\varkappa \pi(V)\}^{2^{-s}}\Vnorm[V^{2^{-s}}]{f}
}\eqsp.
\end{split}
\end{equation*}
Applying the Pinelis version of Rosenthal inequality \cite[Theorem 4.1]{pinelis_1994}, we get
\begin{equation}
\label{eq:martingale_vge}
\PE_{\pi}^{2^{-s}}\bigl[\bigl| \sum_{i=0}^{n-1} \Delta M_i^{(0)} \bigr|^{2^s} \bigr] \leq \underbrace{\bConst{\sf{Rm}, 1} 2^{s/2}\PE_{\pi}^{2^{-s}}\bigl[\bigl|\sum_{i=0}^{n-1} \CPE[\pi]{ (\Delta M_i^{(0)})^2}{\mathcal{F}_i} \bigr|^{2^{s-1}}\bigr]}_{T_1} + \underbrace{\bConst{\sf{Rm}, 2} 2^{s} \PE_{\pi}^{2^{-s}}\bigl[ \{\underset{i}{\max}|\Delta M_i^{(0)}| \bigr\}^{2^s}\bigr]}_{T_2}
\eqsp.
\end{equation}
We bound $T_1$ and $T_2$ separately. We begin with the remainder $T_2$. Proceeding as in \eqref{eq:clever_max_bound}, 
\begin{equation}
\label{eq:bound_max_p_more_log_n}
\textstyle{
\PE_{\pi}^{2^{-s}}\bigl[ \bigl\{\underset{i}{\max}|\Delta M_i^{(0)}| \bigr\}^{2^s}\bigr]
\leq (8\rme /3) 2^{s} \tmix \log{n}  \{\varkappa \pi(V)\}^{2^{-s}}\Vnorm[V^{1/(2^{s} \log{n})}]{f}
}\eqsp.
\end{equation}
Then, with $R_{s,k}$ and $g_k$ defined in \eqref{eq:h_k_g_k_def}, we obtain
\begin{equation*}
\textstyle{
\PE_{\pi}^{2^{-s}}\bigl[\bigl|\sum_{i=0}^{n-1} \CPE[\pi]{ (\Delta M_i^{(0)})^2}{\mathcal{F}_i} \bigr|^{2^{s-1}}\bigr] \leq n^{1/2}\sigma_{\pi}(f) + R_{1, s}
}\eqsp.
\end{equation*}
Now, we sequentially bound the quantity $R_{k,s}^2$ in terms of $R_{k+1,s}^2$ for $1 \leq k \leq s-1$. In order to do this, we apply \Cref{lem:recurrence-R-k_VGE} for $k \in \{1,\ldots,s-2\}$. Then, using \Cref{lem:technical_lemma}, we obtain
\begin{multline}
\label{eq:general_bound_R_1_VGE}
R_{1,s} \leq \bConst{\sf{Rm}, 1}2^{s/2}R_{s-1,s}^{1/2^{s-2}}  + (8\rme \sqrt{2}/3)\bConst{\sf{Rm}, 1}\bConst{\sf{Rm}, 2}^{1/2} 2^{(3/2)s} \tmix \log{n}  \{\varkappa \pi(V)\}^{2^{-s}}\Vnorm[V^{1/(2^s \log{n})}]{f} (s-2)  \\
+ \bConst{\sf{Rm}, 1} (8/3) (16/3)^{1/2} \bigl(\bConst{\sf{Rm}, 2}^{1/2} + 2 \bigr)2^{(3/2)s}n^{1/4}\tmix^{3/4}\{\varkappa \pi(V)\}^{2^{-s+1}} \Vnorm[V^{2^{-s}}]{f}(s-2) \eqsp.
\end{multline}
Consider now $R_{s-1,s}^2$. Proceeding as in \eqref{eq:decomposition of R_k} and \eqref{eq:get martingal of R_k}, we get
\begin{multline*}
\textstyle{
R_{s-1, s}^2 \leq  \PE_{\pi}^{1/2}\bigl[\bigl| \sum_{i=0}^{n-1} g^{2^{s-1}}(Z_{i+1}) - \MKQ g^{2^{s-1}}(Z_i)\bigr|^{2}\bigr] + \PE_{\pi}^{1/2}\bigl[\bigl| \sum_{i=0}^{n-1} h_{s-1}(Z_i) - \pi(h_{s-1}) \bigr|^{2}\bigr]
}\\ 
+ 2\bigl\{ (8/3) \tmix 2^s\Vnorm[V^{2^{-s}}]{f}\{ \varkappa \pi(V)\}^{2^{-s}}\bigr\}^{2^{s-1}}\eqsp.
\end{multline*}
Therefore, using that $\pi \MKQ = \pi$, we obtain
\begin{multline*}
\textstyle \PE_{\pi}\bigl[\bigl| \sum_{i=0}^{n-1} g^{2^{s-1}}(Z_{i+1}) - \MKQ g^{2^{s-1}}(Z_i)\bigr|^{2}\bigr] = n\PE_{\pi}\bigl[\bigl|g^{2^{s-1}}(Z_1) - \MKQ g^{2^{s-1}}(Z_0)\bigr|^2 \bigr] \\
\textstyle \leq n\PE_{\pi}\bigl[ g^{2^s}(Z_0)\bigr] \leq n \left( (8/3) 2^{s} \tmix \Vnorm[V^{2^{-s}}]{f} \right)^{2^s} \{\varkappa \pi(V)\}\eqsp.
\end{multline*}
Using the definition of $h_k$ in \eqref{eq:h_k_g_k_def}, and applying \Cref{lem:bound_of_second_moment}-\ref{var:VGE} with $\alpha = 1/2$, we get
\begin{multline*}
    R_{s-1,s}^2 \leq n^{1/2} \left( (8/3) 2^{s} \tmix \Vnorm[V^{2^{-s}}]{f} \right)^{2^{s-1}} \{\varkappa \pi(V)\}^{1/2} + 2\bigl\{ (8/3) \tmix 2^s\Vnorm[V^{2^{-s}}]{f}\bigr\}^{2^{s-1}} \{\varkappa \pi(V)\}^{1/2}   \\
    + \left(1 + 2^{9/4} / \sqrt{3}\right) n^{1/2}
    \tmix^{1/2} \Vnorm[V^{1/2}]{h_{s-1}}\{\varkappa \pi(V)\}^{1/2}\eqsp.
\end{multline*}
Again using  \Cref{lem:technical_lemma}, $n^{1/i}\tmix^{1-1/i} \leq n^{1/4}\tmix^{3/4}$ for $i \geq 4$, and  \Cref{lem:small_technical_bound}, we get from \eqref{eq:general_bound_R_1_VGE} that 
\begin{multline}
    \label{eq: bound_R_1_VGE}
    R_{1,s} \leq  (8\rme \sqrt{2}/3)\bConst{\sf{Rm}, 1}\bConst{\sf{Rm}, 2}^{1/2} 2^{(3/2)s} \tmix \log{n}  \{\varkappa \pi(V)\}^{2^{-s}}\Vnorm[V^{1/(2^s \log{n})}]{f} (s-1)  \\ +
 \bConst{\sf{Rm}, 1} (8/3) (16/3)^{1/2} \bigl(\bConst{\sf{Rm}, 2}^{1/2} + 2 \bigr)2^{(3/2)s}n^{1/4}\tmix^{3/4}\{\varkappa \pi(V)\}^{2^{-s+1}} \Vnorm[V^{2^{-s}}]{f} s\eqsp.
\end{multline}
Combining the bounds above, we get
\begin{multline*}
\PE_{\pi}^{1/p}\bigl[\bigl| \Sstat_n \bigr|^{p}\bigr]  \leq  \bConst{\sf{Rm}, 1} p^{1/2}n^{1/2}\sigma_{\pi}(f) + \ConstD_{\ref{lem: rosenthal_VGE_for_2^s}, 1} p^{2} n^{1/4}\tmix^{3/4}\{\varkappa \pi(V)\}^{2/p} \Vnorm[V^{1/p}]{f} \log_2{p} \\ + \ConstD_{\ref{lem: rosenthal_VGE_for_2^s}, 2} p^{2} \tmix \log{n} \{\varkappa \pi(V)\}^{1/p} \Vnorm[V^{1/(p \log{n})}]{f} \log_2{p}\eqsp.
\end{multline*}
Now it remains to apply \Cref{lem:small_technical_bound} with $\alpha = 1/p$ and $1/(p \log n)$, respectively. 
\end{proof}

\subsection{Proof of \Cref{theo:rosenthal_VGE}.}
\label{sec:proof_rosenthal_VGE}
The result for general $p$ follows from \Cref{lem: rosenthal_VGE_for_2^s} and the Lyapunov inequality. The proof of \eqref{eq:main_rosenthal_VGE_xi}
follows from the maximal exact coupling argument along the lines of the result provided in \cite[Theorem 2]{durmus2024probability}.

\subsection{Auxiliary lemmas for Wasserstein ergodicity}
\begin{lemma}
\label{lem: auxiliary_rosenthal_WGE}
Assume \Cref{ass:cost_fun} and \Cref{assum:wasserstein-convergence}. Then for any $p\geq 2$, $n \geq \tmix$ and $f: \Zset \rightarrow \rset$ such that $f\in \mathcal{L}_{1/p,V}$, it holds
\begin{equation*}
\PE_{\pi}^{1/p}\bigl[\bigl|\Sstat_n|^p\bigr] \leq \ConstD_{\ref{lem: auxiliary_rosenthal_WGE}, 1} p^{3/2}n^{1/2}\tmix^{1/2} \varsigma \{\varkappa_c \pi(V)\}^{1/p}\Nnorm[1/p, V]{f} + \ConstD_{\ref{lem: auxiliary_rosenthal_WGE}, 2} p^2 n^{1/p} \tmix^{1-1/p} \varsigma \{\varkappa_c \pi(V)\}^{1/p} \Nnorm[1/p, V]{f} \eqsp,
\end{equation*}
where $\ConstD_{\ref{lem: auxiliary_rosenthal_WGE}, 1} = (8/3) \bConst{\sf{Rm}, 1}$, and $\ConstD_{\ref{lem: auxiliary_rosenthal_WGE}, 2} = (16/3) \bConst{\sf{Rm}, 2} + 11/3$.
\end{lemma}
\begin{proof}
The proof follows the lines of \Cref{lem: auxiliary_rosenthal_VGE}. We again introduce $g_{\tmix} = \sum_{i = 0}^{+\infty} Q^{i\tmix} \bar{f}$, which is well defined under \Cref{ass:cost_fun} and \Cref{assum:wasserstein-convergence}. The only difference is the control of the norm of Poisson solution. Using \Cref{lem: poisson_sol}, we get
\begin{align}
\label{eq: bound_of_pi(g)_t_mix_WGE}
    |g_{\tmix}(z)| & \leq (1/3)2^{3-1/p} p \varsigma \varkappa_{c}^{1/p} \{\pi(V) + V(z)\}^{1/p} \Nnorm[1/p, V]{f}\eqsp, \nonumber \\
    |\MKQ^{\tmix}g_{\tmix}(z)| & \leq (1/3) 2^{3-1/p} \varsigma \varkappa_{c}^{1/p} \{\pi(V) + V(z)\}^{1/p} \Nnorm[1/p, V]{f}\eqsp, \\
    \PE_{\pi}^{1/p}[|g_\tmix(Z_0)|^{p}] & \leq (8/3) p \varsigma \{\varkappa_{c} \pi(V)\}^{1/p} \Nnorm[1/p, V]{f} \eqsp. \nonumber
\end{align}
All the remaining proof follows the lines of \Cref{lem: auxiliary_rosenthal_VGE} and is omitted.
\end{proof}

\begin{lemma}
\label{lem:recurrence-R-k_WGE}
Assume \Cref{ass:cost_fun} and \Cref{assum:wasserstein-convergence}. Then for any $\beta \geq 0$, measurable function $f: \Zset \rightarrow \rset^d$, $f\in \mathcal{L}_{\beta,\log V}$, any $s \in \nset$ and $k \in \{1,\ldots,s-1\}$, it holds
\begin{equation*}
R_{k,s}^2 \leq \bConst{\sf{Rm}, 1} 2^{(s-k)/2}R_{k+1, s} + \bConst{\sf{Rm}, 2} 2^{s-k + 1}\WBound_{max}(s,k) + 2\WBound_{\operatorname{\operatorname{lin}}, 1}(s,k) +\WBound_{\operatorname{\operatorname{lin}}, 2}(s,k)  \eqsp,
\end{equation*}
where $\WBound_{max}(s,k)$, $\WBound_{\operatorname{\operatorname{lin}}, 1}(s,k)$ and $\WBound_{\operatorname{\operatorname{lin}}, 2}(s,k)$ are defined in \eqref{eq:t_max_bound_def_WGE} and \eqref{eq:bound_lin_term_WGE}, respectively.
\end{lemma}
\begin{proof}
Proceeding as in \Cref{lem:recurrence-R-k} and \Cref{lem:recurrence-R-k_VGE}, we upper bound both terms in the decomposition \eqref{eq:decomposition of R_k}. To do so, we control the norm of $h_k$ defined in \eqref{eq:smart-decomposition-hk}. Using \Cref{lem: product_of_two_funct},
\begin{equation*}
\Nnorm[2^{k-s}, V]{h_k}
\leq 2^{k + k\cdot2^{k-s}}\Nnorm[2^{-s}, V]{f}\Nnorm[2^{-s}, V]{g + Qg}\Nnorm[2^{-s+1}, V]{g^2 + (Qg)^2}\dots \Nnorm[2^{-s+k-1}, V]{g^{2^{k-1}} + (Qg)^{2^{k-1}}}.
\end{equation*}
Bounding $\Nnorm[2^{-s}, V]{Qg(z)}$ by the same expression as in \Cref{lem: poisson_sol}, we get
\begin{equation*}
\Nnorm[2^{-s}, V]{Qg(z)} \leq (8/3) 2^s \tmix \varsigma \{\varkappa_{c} \pi(V)\}^{2^{-s}}  \Nnorm[2^{-s},V]{f}\eqsp.
\end{equation*}
Using \Cref{lem: pow_of_funct}, we get 
\begin{equation*}
\Nnorm[2^{i-s}, V]{g^{2^i} + (Qg)^{2^i}} \leq \Nnorm[2^{i-s}, V]{g^{2^i}} + \Nnorm[2^{i-s}, V]{(Qg)^{2^i}} \leq 2^{i} \bigl((8/3) 2^s\tmix \varsigma \{\varkappa_{c} \pi(V)\}^{2^{-s}}\Nnorm[2^{-s},V]{f}\bigr)^{2^i}\eqsp.
\end{equation*}
Combining the above bounds, 
\begin{equation}
\label{eq: bound_V_norm_of_h_k_wge}
\Nnorm[2^{k-s}, V]{h_k} \leq 2^{k(k-1)/2+ k(1 + 2^{k-s})}\bigl((8/3) 2^s\tmix \varsigma \{\varkappa_{c} \pi(V)\}^{2^{-s}}\bigr)^{2^k - 1} \Nnorm[2^{-s}, V]{f}^{2^k}\eqsp.
\end{equation}
\Cref{lem: auxiliary_rosenthal_WGE} implies that $\PE_{\pi}^{2^{k-s}}\bigl[\bigl| \sum_{i=0}^{n-1}\{h_k(Z_i) - \pi(h_k)\}\bigr|^{2^{s-k}}\bigr] \leq \WBound_{\operatorname{\operatorname{lin}}, 1} (s,k) + \WBound_{\operatorname{\operatorname{lin}}, 2} (s,k)$, where
\begin{equation}
\label{eq:bound_lin_term_WGE}
\begin{split}
\WBound_{\operatorname{\operatorname{lin}}, 1}(s,k) &= \ConstD_{\ref{lem: auxiliary_rosenthal_WGE}, 1} 2^{k^2/2 + s(2^k + 1/2)} (n/\tmix)^{1/2} (\varsigma \tmix)^{2^k}\{\varkappa_c \pi(V)\}^{2^{k-s+1}-2^{-s}}(\tfrac{8}{3})^{2^k-1}\Nnorm[2^{-s}, V]{f}^{2^k} \eqsp,\\
\WBound_{\operatorname{\operatorname{lin}}, 2}(s,k) &= \ConstD_{\ref{lem: auxiliary_rosenthal_WGE}, 2} 2^{k^2/2 + s(2^k + 1)} (n/\tmix)^{2^{k-s}} (\varsigma \tmix)^{2^k}\{\varkappa_{c} \pi(V)\}^{2^{k-s+1}-2^{-s}}(\tfrac{8}{3})^{2^k-1}\Nnorm[2^{-s}, V]{f}^{2^k}\eqsp,
\end{split}
\end{equation}
and $\ConstD_{\ref{lem: auxiliary_rosenthal_WGE}, 1}, \ConstD_{\ref{lem: auxiliary_rosenthal_WGE}, 1}$ are absolute constants given in \Cref{sec:constants}. Now we upper bound the martingale-difference term in \eqref{eq:decomposition of R_k}.
Using the Minkowski inequality and \Cref{lem: poisson_sol}, we obtain that
\begin{multline*}
\PE_{\pi}^{2^{k-s}}\bigl[\bigl| \sum_{i=0}^{n-1} g^{2^k}(Z_i) - Qg^{2^k}(Z_i)\bigr|^{2^{s-k}}\bigr] \\ 
\leq \PE_{\pi}^{2^{k-s}}\bigl[\bigl| \sum_{i=0}^{n-1} g^{2^k}(Z_{i+1}) - Qg^{2^k}(Z_i)\bigr|^{2^{s-k}}\bigr] + 2\bigl\{ (8/3) 2^s \varsigma \tmix \{\varkappa_c \pi(V)\}^{2^{-s}} \Nnorm[2^{-s}, V]{f}\bigr\}^{2^k}\eqsp.
\end{multline*}
Now, repeating the steps \eqref{eq: get_martingal_of_R_k_VGE}--\eqref{eq:clever_max_bound} of \Cref{lem:recurrence-R-k_VGE}, we get the statement of lemma with 
\begin{equation}
\label{eq:t_max_bound_def_WGE}
\WBound_{max}(s,k) = \bigl\{ (8\rme/3) 2^{s} \tmix \varsigma \log{n}  \{\varkappa_{c} \pi(V)\}^{1/(2^s)}\Nnorm[1/(2^s \log{n}), V]{f}\bigr\}^{2^k}\eqsp.
\end{equation}
\end{proof}

\begin{lemma}
\label{lem: rosenthal_WGE_for_2^s}
Assume  \Cref{ass:cost_fun} and \Cref{assum:wasserstein-convergence}. Then for any measurable function $f: \Zset \rightarrow \rset^d$, $f\in \mathcal{L}_{\beta,\log V}$, $s \in \nsets, s \geq 2$, $p = 2^s$, and $n \geq \tmix$, it holds that
\begin{multline*}
\PE_{\pi}^{1/p}\bigl[\bigl| \Sstat_n \bigr|^{p} \bigr] \leq \bConst{\sf{Rm}, 1}p^{1/2}n^{1/2}\sigma_{\pi}(f) + \ConstD_{\ref{lem: rosenthal_WGE_for_2^s}, 1}(\beta/\rme)^{\beta} p^{\beta+2}n^{1/4}\tmix^{3/4} \varsigma \{\varkappa_{c} \pi(V)\}^{2/p}\Nnorm[\beta, \log V]{f}\log_2p  \\ 
+\ConstD_{\ref{lem: rosenthal_WGE_for_2^s}, 2}(\beta/\rme)^{\beta} p^{\beta +2} \varsigma \tmix \log^{\beta+1}{n} \{\varkappa_{c} \pi(V)\}^{1/p}\Nnorm[\beta, \log V]{f}\log_2{p} \eqsp.
\end{multline*}
where $\sigma_{\pi}(f)$ is defined in \eqref{eq:asympt_var_def}, and
\[
\textstyle{
\ConstD_{\ref{lem: rosenthal_WGE_for_2^s}, 1} = (16/3)^{3/2} \bConst{\sf{Rm}, 1}^2\bigl( \sqrt{2} \bConst{\sf{Rm}, 2}^{1/2} + 1\bigr)\eqsp, \quad \ConstD_{\ref{lem: rosenthal_WGE_for_2^s}, 2} = (8\rme/3) \bConst{\sf{Rm}, 2}^{1/2}(\bConst{\sf{Rm}, 1}^2 + \bConst{\sf{Rm}, 2}^{1/2})
}\eqsp.
\]
\end{lemma}
\begin{proof}
Proceeding as in \Cref{lem: rosenthal_VGE_for_2^s}, we can assume that $n \geq \tmix$. With $g$ defined in \eqref{eq:Poisson_equation_definition}, we proceed as in \eqref{eq:martingale_decomp_tv} and apply the upper bound \eqref{eq:v_alpha_log_s_bound} with $\alpha = 2^{-s}$. Thus,
\begin{equation*}
\textstyle{
\PE_{\pi}^{2^{-s}}\bigl[\bigl| \Sstat_n \bigr|^{2^s} \bigr]\leq  \PE_{\pi}^{2^{-s}}\bigl[ \bigl| \sum_{i=0}^{n-1}\Delta M_i^{(0)} \bigr|^{2^s} \bigr] + (16/3) 2^s \varsigma \tmix \{\varkappa_{c} \pi(V)\}^{2^{-s}} \Nnorm[2^{-s}, V]{f}
}\eqsp.
\end{equation*}
Proceeding exactly as in \eqref{eq:martingale_vge}, we need to bound $T_1$ and $T_2$. We start with the remainder term $T_2$ and, proceeding as in \eqref{eq:clever_max_bound} and \eqref{eq:bound_max_p_more_log_n}, we get
\begin{equation*}
\textstyle{
T_2 = \bConst{\sf{Rm}, 2} 2^{s} \PE_{\pi}^{2^{-s}}\bigl[ \{\underset{i}{\max}|\Delta M_i^{(0)}| \bigr\}^{2^s}\bigr] \leq (8\rme /3) \bConst{\sf{Rm}, 2} 2^{2s} \tmix \varsigma \log{n} \{\varkappa_c \pi(V)\}^{2^{-s}} \Nnorm[1/(2^s \log{n}), V]{f}
}\eqsp.
\end{equation*}
Then, with $R_{1,s}$ defined in \eqref{eq:h_k_g_k_def}, we obtain
\begin{equation*}
\textstyle{
\PE_{\pi}^{2^{-s}}\bigl[\bigl|\sum_{i=0}^{n-1}\PE_{\pi}^{\mathcal{F}_i}\bigl\{ (\Delta M_i^{(0)})^2\bigr\}\bigr|^{2^{s-1}}\bigr] \leq n^{1/2}\sigma_{\pi}(f) + R_{1, s}
}\eqsp.
\end{equation*}
Applying \Cref{lem:recurrence-R-k_WGE} for $k \in \{1,\ldots,s-2\}$ and using \Cref{lem:technical_lemma}, we get
\begin{multline*}
R_{1,s} \leq \bConst{\sf{Rm}, 1}2^{s/2}R_{s-1,s}^{1/2^{s-2}}  + (8 \rme \sqrt{2} / 3) \bConst{\sf{Rm}, 1}\bConst{\sf{Rm}, 2}^{1/2}  2^{(3/2)s} \varsigma \tmix \log{n} \{\varkappa_c \pi(V)\}^{2^{-s}}\Nnorm[2^{-s}, V]{f}(s-2)  \\ +
\bConst{\sf{Rm}, 1} (16/3)^{3/2} \varsigma \bigl(\bConst{\sf{Rm}, 2}^{1/2} + 2 \bigr) 2^{(3/2)s} n^{1/4} \tmix^{3/4}\{\varkappa_{c} \pi(V)\}^{2^{-s+1}} \Nnorm[2^{-s}, V]{f}(s-2) \eqsp.
\end{multline*}
Consider now $R_{s-1,s}^2$. Proceeding as in \eqref{eq:decomposition of R_k}, and following the lines of \Cref{lem: rosenthal_VGE_for_2^s}, we get 
\begin{multline*}
R_{s-1,s}^2 \leq  n^{1/2} \bigl((8/3) 2^s \varsigma \tmix \Nnorm[2^{-s}, V]{f} \bigr)^{2^{s-1}} \{\varkappa_{c} \pi(V)\}^{1/2} + 2\bigl\{ (8/3) 2^s \varsigma \tmix \{\varkappa_c \pi(V)\}^{2^{-s}} \Nnorm[2^{-s}, V]{f}\bigr\}^{2^{s-1}} \\ +
\bigl(1 + (4/\sqrt{3}) \varsigma^{1/2}\bigr)n^{1/2}\tmix^{1/2} \{\varkappa_c \pi(V)\}^{1/2}\Nnorm[1/2, V]{h_{s-1}}\eqsp.
\end{multline*}
Therefore, using that $n \geq \tmix$, and taking into account the bound for $\Nnorm[1/2, V]{h_{s-1}}$ established in \eqref{eq: bound_V_norm_of_h_k_wge}, we arrive at the bound 
\begin{multline*}
R_{1,s} \leq (8 \rme \sqrt{2} / 3) \bConst{\sf{Rm}, 1}\bConst{\sf{Rm}, 2}^{1/2}  2^{(3/2)s} \varsigma \tmix \log{n} \{\varkappa_c \pi(V)\}^{2^{-s}}\Nnorm[2^{-s}, V]{f} (s-1)  \\ +
\bConst{\sf{Rm}, 1} (16/3)^{3/2} \varsigma \bigl(\bConst{\sf{Rm}, 2}^{1/2} + 2 \bigr) 2^{(3/2)s} n^{1/4} \tmix^{3/4}\{\varkappa_{c} \pi(V)\}^{2^{-s+1}} \Nnorm[2^{-s}, V]{f}\,s\eqsp.
\end{multline*}
Thus, combining the bounds above, we get
\begin{gather*}
\PE_{\pi}^{1/p}\bigl[\bigl| \Sstat_n \bigr|^{p} \bigr] \leq \bConst{\sf{Rm}, 1}p^{1/2}n^{1/2}\sigma_{\pi}(f) + \ConstD_{\ref{lem: rosenthal_WGE_for_2^s}, 1} p^{2} n^{1/4}\tmix^{3/4} \varsigma \{\varkappa_{c} \pi(V)\}^{2/p}\Nnorm[1/p, V]{f}\log_2p \\ 
+ \ConstD_{\ref{lem: rosenthal_WGE_for_2^s}, 2} p^{2} \tmix \varsigma \log{n} \{\varkappa_c \pi(V)\}^{1/p}\Nnorm[1/(p \log{n}), V]{f}\log_2{p}\eqsp.
\end{gather*}
To conclude the proof we apply \Cref{lem:small_technical_bound}.
\end{proof}

\subsection{Proof of \Cref{theo:rosenthal_WGE}}
\label{sec:proof_rosenthal_WGE}
The result for general $p$ follows from \Cref{lem: rosenthal_WGE_for_2^s} and the Lyapunov inequality. The result under arbitrary initial distribution $\xi$ follows directly from the distributional coupling result provided in \cite[Theorem~8]{durmus2024probability} (see also Section~4.7 in \cite{durmus2024probability}).

\begin{appendix}

\section{Auxiliary results}
\label{appB}
To extend our proof technique from \Cref{sec:uge-chains}, we need a similar result for the convergence rate of $\Vnorm[V^{\alpha}]{\xi \MKQ^n - \pi}$, $0 < \alpha \leq 1$.
\begin{lemma}
\label{lem:v_norm_alpha}
Assume \Cref{ass:VGE}. Then, for any probability measure $\xi$ such that $\xi(V) < \infty$, $\alpha \in (0;1]$ and all $n \in \nset$, it holds
\begin{equation}
\label{eq: bound_for_V^alpha_norm}
\Vnorm[V^\alpha]{\xi \MKQ^n - \pi} \leq 2^{1-\alpha} \varkappa^{\alpha} \{\pi(V) + \xi(V)\}^\alpha \left(1/4\right)^{\lfloor n/\tmix \rfloor \alpha}\eqsp.
\end{equation}
\end{lemma}
\begin{proof}
Consider two probability measures $\xi$ and $\xi^{\prime}$ and $\alpha \in (0,1)$. Then, using Jensen's inequality,
\begin{align*}
\textstyle \Vnorm[V^\alpha]{\xi - \xi^{\prime}} 
&= \textstyle \int_{\Zset}|\xi - \xi^{^\prime}|(dx)V^{\alpha}(x) = 2\tvnorm{\xi - \xi^{^\prime}}\int_{\Zset}\frac{|\xi - \xi^{^\prime}|(dx)}{2\tvnorm{\xi - \xi^{^\prime}}}V^{\alpha}(x)  \\
& \textstyle \leq (2 \tvnorm{\xi - \xi^{\prime}})^{1-\alpha}\Vnorm[V]{\xi - \xi^{\prime}}^\alpha\eqsp.
\end{align*}
Applying the bound above with $\xi = \delta_x \MKQ^n$, $\xi^{\prime} = \pi$, using $\tvnorm{\xi - \xi^{\prime}} \leq 1$, we obtain that \Cref{ass:VGE} implies
\begin{equation}
\label{eq: bound_for_V^alpha_norm_new}
\Vnorm[V^\alpha]{\delta_x Q^n - \pi} \leq \varkappa^{\alpha} 2^{1-\alpha}\{\pi(V) + V(x) \}^\alpha \left(1/4\right)^{\lfloor n/\tmix \rfloor \alpha}\eqsp.
\end{equation}
\end{proof}

Since the bounds obtained in \Cref{theo:rosenthal}, \Cref{theo:rosenthal_VGE}, \Cref{theo:rosenthal_WGE} are valid for any $p \geq 2$, we can rewrite these results as a Bernstein-type bounds using Markov inequality and optimization with respect to $p$; see the lemma below for details.

\begin{lemma}
\label{lem:deviation_bound}
Let $Y$ be a random variable such that, for any $p \geq 2$, $\PE^{1/p}\bigl[\bigl|Y|^p\bigr] \leq \varphi(p)$, 
for some function $\varphi(\cdot): \rset_{+} \mapsto \rset_{+}$. Then for any $\delta \in (0,1/\rme^2)$,
\begin{equation}
\label{eq: general_Bernstein_inequality}
\mathbb{P}\left(|Y| \geq \rme \varphi(\ln(1/\delta))\right) \leq \delta\eqsp.
\end{equation}
\end{lemma}
\begin{proof}
Using Markov inequality, we get for any $p \geq 2$ that
\begin{equation}
\label{eq: Markov_inequality}
\mathbb{P}\bigl(|Y| \geq t \bigr) \leq \frac{\PE\bigl[|Y|^p\bigr]}{t^p} \leq
\left(\frac{\varphi(p)}{t}\right)^{p}\eqsp.
\end{equation}
Fix now $p=\ln(1/\delta)$ and $t = \rme \eqsp\varphi(\ln(1/\delta))$. Then the r.h.s. of \eqref{eq: Markov_inequality} does not exceed $\rme^{-\ln\frac{1}{\delta}} = \delta$, and the statement follows.
\end{proof}

Now we provide a result which provides an upper bound for the second moment of the additive functional under different ergodicity assumptions. 

\begin{lemma}
\label{lem:bound_of_second_moment}
Assume that
\begin{enumerate}[label=(\alph*)]
\item \label{var:UGE} \Cref{ass:UGE} holds and the function $f: \Zset \rightarrow \rset$ satisfies $\|f\|_{\infty} < \infty$;
\item \label{var:VGE} \Cref{ass:VGE} holds and the function $f: \Zset \rightarrow \rset$ satisfies $\Vnorm[V^{\alpha}]{f} < \infty$ with some $\alpha \in (0;1/2]$;
\item \label{var:WGE} \Cref{ass:cost_fun} and \Cref{assum:wasserstein-convergence}  holds and the function $f: \Zset \rightarrow \rset$ satisfies $f\in \mathcal{L}_{\alpha,V}$ with some $\alpha \in (0;1/2]$.
\end{enumerate}
Then it holds that 
\begin{enumerate}[label=(\roman*)]
\item under \ref{var:UGE}, $\PE_{\pi}^{1/2}\bigl[\bigl| \Sstat_n \bigr|^2\bigr] \leq (1 + 4/\sqrt{3})n^{1/2}\tmix^{1/2}$;
\item under \ref{var:VGE}, $\textstyle{
\PE_{\pi}^{1/2}\bigl[\bigl| \Sstat_n \bigr|^2\bigr] \leq \left(1 + \frac{\varkappa^{\alpha/2} 2^{2-\alpha/2}}{\sqrt{3} \alpha^{1/2}}\right) n^{1/2} \tmix^{1/2} \Vnorm[V^{\alpha}]{f}\{\pi(V)\}^{\alpha}
}$;
\item under \ref{var:WGE}, $\PE_{\pi}^{1/2}\bigl[\bigl| \Sstat_n \bigr|^2\bigr] \leq \bigl(1 + \frac{2^{3/2} \varsigma^{1/2} \varkappa_c^{\alpha/2}}{\sqrt{3}\alpha^{1/2}}\bigr)n^{1/2}\tmix^{1/2} \{\pi(V)\}^{\alpha}\Nnorm[\alpha, V]{f}$.
\end{enumerate}
\end{lemma}
\begin{proof}
Note that under $\pi$,
\begin{align}
\label{eq:second_moment_decomposition}
\PE_{\pi}^{1/2}\bigl[\bigl| \Sstat_n \bigr|^2\bigr] =\bigl( n\PE_{\pi}\{\barf(Z_0)\}^2 + 2\sum_{i=0}^{n-1}\sum_{k = 1}^{n-1-i}\PE_{\pi}\bigl[\barf(Z_i)\barf(Z_{i+k})\bigr] \bigr)^{1/2}\eqsp.
\end{align}
Using the Markov property, we obtain
\begin{equation*}
\PE_{\pi}[\barf(Z_i)\barf(Z_{i+k})] = \PE_{\pi}\bigl[ \barf(Z_i) \PE_{\pi}\bigl[\barf(Z_{i+k})|\mathcal{F}_{i} \bigr] \bigr] = \PE_{\pi}\bigl[ \barf(Z_i)(Q^kf(Z_i)-\pi(f)) \bigr]\eqsp.
\end{equation*}
Now we just need to bound the covariance term above under different ergodicity assumptions. In particular, we get from \Cref{ass:UGE} (case \ref{var:UGE}), that
\[
|\PE_{\pi}[\barf(Z_i) \barf(Z_{i+k})]| \leq 2 (1/4)^{\lfloor k / \tmix\rfloor}\eqsp.
\]
From \Cref{ass:VGE} and Lemma A.1 from the main text combined with Jensen's inequality (case \ref{var:VGE}),
\begin{equation*}
\PE_{\pi}\bigl[\barf(Z_i)\barf(Z_{i+k})\bigr] \leq 2^{2-\alpha} \varkappa^{\alpha} \{\pi(V)\}^{2\alpha}\bigl(1/4\bigr)^{\lfloor k/\tmix\rfloor \alpha} \Vnorm[V^{\alpha}]{f}^2\eqsp.
\end{equation*}
Recall that under \Cref{ass:cost_fun} and \Cref{assum:wasserstein-convergence} from \cite[Corollary 20.4.7]{douc:moulines:priouret:soulier:2018} for $\alpha  \in (0,1/2]$ and $f \in \mathcal{L}_{\alpha, V} $ it holds
\begin{equation}
\label{eq:bound_norm_WGE}
|\xi Q^n(f) - \pi(f)| \leq \varsigma \varkappa_{c}^{\alpha} (1/4)^{\lfloor n/\tmix\rfloor \alpha}\{\xi(V^{\alpha})+\pi(V^{\alpha})\}\Nnorm[\alpha, V]{f} \eqsp.
\end{equation}
Hence, similarly, from \Cref{assum:wasserstein-convergence}, \eqref{eq:bound_norm_WGE}, and Jensen's inequality (case \ref{var:WGE}) it follows that
\begin{equation*}
\textstyle{
\PE_{\pi}\bigl[\bar{f}(Z_i)\bar{f}(Z_{i+k})\bigr] \leq 2 \varsigma \varkappa_c^{\alpha}(1/4)^{\lfloor k/\tau\rfloor\alpha}\{\pi(V)\}^{2\alpha} \Nnorm[\alpha, V]{f}^{2}
}\eqsp.
\end{equation*}
Now it remains to take sum w.r.t. $k$ in \eqref{eq:second_moment_decomposition} and apply the lower bound \eqref{eq:lower_bound_exponent} if needed.
\end{proof}

\begin{lemma}
\label{lem: poisson_sol}
Assume that 
\begin{enumerate}[label=(\alph*)]
\item \label{pois:UGE} \Cref{ass:UGE} holds and the function $f: \Zset \rightarrow \rset$ satisfies $\|f\|_{\infty} < \infty$;
\item \label{pois:VGE} \Cref{ass:VGE} holds and the function $f: \Zset \rightarrow \rset$ satisfies $\Vnorm[V^{\alpha}]{f} < \infty$ with some $\alpha \in (0;1]$;
\item \label{pois:WGE} \Cref{ass:cost_fun} and \Cref{assum:wasserstein-convergence}  holds and the function $f: \Zset \rightarrow \rset$ satisfies $f\in \mathcal{L}_{\alpha,V}$ with some $\alpha \in (0;1/2]$;
\end{enumerate}
Then the function
\begin{equation}
\label{eq:Poisson_equation}
g(z) = \sum_{k=0}^{\infty}\{\MKQ^{k}f(z) - \pi(f)\}
\end{equation}
is well-defined and is a solution of the Poisson equation, i.e. ${g(z) - \MKQ g(z) = f(z) - \pi(f)}$. Moreover, it holds that
\begin{enumerate}[label=(\roman*)]
\item under \ref{pois:UGE}, $\infnorm{g} \leq (8/3)\tmix\infnorm{f}$;
\item under \ref{pois:VGE}, $|g(z)| \leq \frac{\varkappa^{\alpha} 2^{3-\alpha}}{3 \alpha}\{\pi(V) + V(z)\}^{\alpha}\tmix\Vnorm[V^{\alpha}]{f}$, for any $z \in \Zset$;
\item under \ref{pois:WGE}, $|g(z)| \leq \frac{2^{3-\alpha}}{3\alpha} \varsigma \varkappa_{c}^{\alpha} \tmix \{\pi(V) + V(z)\}^{\alpha} \Nnorm[\alpha, V]{f}$, for any $z \in \Zset$ and any $0 < \alpha \leq 1/2$. Moreover, $\Nnorm[\alpha, V]{g} \leq \frac{8}{3\alpha}\tmix \varsigma \varkappa_{c}^{\alpha}\{\pi(V)\}^{\alpha}  \Nnorm[\alpha,V]{f}$;
\end{enumerate}
\end{lemma}
\begin{proof}
We begin with the proof under \ref{pois:UGE}. Note that $\infnorm{g} \leq \sum_{k=0}^{+\infty} \infnorm{\MKQ^k f - \pi(f)}$. Moreover, using \Cref{ass:UGE},
\begin{equation}
\label{eq:bound_g_norm}
\textstyle{
\infnorm{g} \leq 2 \infnorm{f} \sum_{k=0}^{+\infty}\bigl(1/4\bigr)^{\lfloor k/\tmix\rfloor} \leq 2 \infnorm{f} \sum_{n=0}^{+\infty}\sum_{r=0}^{\tmix-1}\bigl(1/4\bigr)^n = (8/3)\tmix\infnorm{f}
}\eqsp.
\end{equation}
From \cite[Proposition 21.2.3]{douc:moulines:priouret:soulier:2018} and \eqref{eq:bound_g_norm} it follows that the function $g$ is a solution of the Poisson equation. Assume now \ref{pois:VGE}. Note that 
\begin{equation}
\label{eq:general_bound_g}
\textstyle{
|g(z)| \leq \sum_{k=0}^{+\infty}\left| \MKQ^k f(z) - \pi(f)\right|
}\eqsp .
\end{equation}
 Applying \Cref{lem:v_norm_alpha}, we get
\begin{align*}
\textstyle{
|g(z)| \leq  \varkappa^{\alpha} \sum_{k=0}^{+\infty}2^{1-\alpha}\{\pi(V) + V(z) \}^{\alpha} \left(1/4\right)^{\lfloor k/\tmix \rfloor \alpha}\Vnorm[V^{\alpha}]{f} = \varkappa^{\alpha} 2^{1-\alpha}\{\pi(V) + V(z)\}^{\alpha}\tmix\frac{\Vnorm[V^{\alpha}]{f}}{1 - (1/4)^{\alpha}}
}\eqsp.
\end{align*}
Since
\begin{equation}
\label{eq:lower_bound_exponent}
1 - \exp (-x) \geq (3/4) x / \ln{4} \text{ for } x \in [0; \ln{4}]\eqsp,
\end{equation}
it follows from the bound above that
\begin{equation}
\label{eq: bound_of_g}
|g(z)| \leq (3 \alpha)^{-1} \varkappa^{\alpha} 2^{3-\alpha}\{\pi(V) + V(z)\}^{\alpha}\tmix\Vnorm[V^{\alpha}]{f}\eqsp,
\end{equation}
and \cite[Proposition 21.2.3]{douc:moulines:priouret:soulier:2018} implies that $g(z)$ is a solution of the Poisson equation. Assume now \ref{pois:WGE}. Using \eqref{eq:general_bound_g} and \eqref{eq:bound_norm_WGE}, we get
\begin{align*}
\textstyle{
|g(z)| \leq  \sum_{k=0}^{+\infty} \varsigma \varkappa_{c}^{\alpha} (1/4)^{\lfloor n/\tmix\rfloor \alpha}[V^{\alpha}(z) + \pi(V^{\alpha})] \Nnorm[\alpha, V]{f} = \varsigma \varkappa_{c}^{\alpha} \tmix \{\pi(V^{\alpha})+V^{\alpha}(z)\} \Nnorm[\alpha, V]{f}(1-(1/4)^{2\alpha})^{-1}
}\eqsp.
\end{align*}
Using Jensen's inequality and \eqref{eq:lower_bound_exponent}, we obtain 
\begin{equation}
\label{eq:bound_of_g_wge}
\textstyle{
|g(z)| \leq \varsigma 2^{3-\alpha} /(3\alpha) \varkappa_{c}^{\alpha} \tmix \{\pi(V) + V(z)\}^{\alpha} \Nnorm[\alpha, V]{f}
}\eqsp,
\end{equation}
and \cite[Proposition 21.2.3]{douc:moulines:priouret:soulier:2018} implies that $g(z)$ is a solution of the Poisson equation. Using Lemma A.8 from the main text,
\begin{equation}
\label{eq: bound_lip_norm_of_g_wge}
\textstyle{
\Nnorm[\alpha, V]{g} \leq 2 \varsigma \varkappa_{c}^{\alpha}\{\pi(V)\}^{\alpha}  \Nnorm[\alpha,V]{f} \sum_{k=0}^{+\infty}(1/4)^{\lfloor n/\tmix\rfloor\alpha}\leq \varsigma (8/(3\alpha)) \tmix \varkappa_{c}^{\alpha}\{\pi(V)\}^{\alpha}  \Nnorm[\alpha,V]{f}
}\eqsp.
\end{equation}
\end{proof}

Now we provide a technical lemma used to recalculate the function norm for the case of logarithmic growth.
\begin{lemma}
\label{lem:small_technical_bound}
Let $f:\Zset \to \rset^{d}$ be a function such that $\Vnorm[\log^{s}V]{f} < \infty$ for some $s \geq 0$. Then for any $\alpha > 0$ it holds that $\Vnorm[V^{\alpha}]{f} < \infty$ and 
\[
\Vnorm[V^{\alpha}]{f} \leq \left(s / (\alpha \rme)\right)^{s}\Vnorm[\log^{s}V]{f}
\]
\end{lemma}
\begin{proof}
The proof directly follows from the bound
\begin{equation}
\label{eq:v_alpha_log_s_bound}
\textstyle{
\Vnorm[V^{\alpha}]{f} \leq \sup_{x \geq 1}\{x^{s} \rme^{-\alpha x}\} \Vnorm[\log^{s}V]{f} \leq \left(s / (\alpha \rme)\right)^{s}\Vnorm[\log^{s}V]{f}
}\eqsp.
\end{equation}
\end{proof}

\begin{lemma}
\label{lem:technical_lemma}
Let $m \in \nset$, and, for $k \in \{1,\dots,m-1\}$, $0 \leq r_k \leq a_k r_{k+1}^{1/2}+ b_k + c_k$, where $a_k, b_k, c_k \geq 0$. Then it holds
\[
r_1 \leq \prod_{i=1}^{m-1} a_i^{1/2^{i-1}} r_m^{1/2^{m-1}} + \sum_{\ell=1}^{m-1} \prod_{i=1}^{\ell-1} a_i^{1/2^{i-1}}  b_\ell^{1/2^{\ell-1}} + \sum_{\ell=1}^{m-1} \prod_{i=1}^{\ell-1} a_i^{1/2^{i-1}}  c_\ell^{1/2^{\ell-1}}\eqsp.
\]
\end{lemma}
\begin{proof}
The proof is through an elementary recurrence.
\end{proof}
\begin{corollary}
\label{cor:technical_lemma}
Assume that for $s \geq 2$,  $k = \{1, \ldots, s-2\}$ and $\alpha, \beta, \gamma, \kappa_1, \kappa_2 \geq 1$, it holds that
\begin{equation}
\label{eq: technical_inequality}
    r_k \leq \alpha^{1/2} 2^{(s-k)/4} r_{k+1}^{1/2} + \kappa_0 \beta^{2^{k-1}} 2^{k/4} 2^{s/4} + \kappa_1 \gamma^{2^{k-1}} 2^{s/2}\, .
\end{equation}
Then,
\begin{equation*}
r_1 \leq \alpha 2^{s/2} r_{s-1}^{1/2^{s-2}} + \alpha \{\beta \kappa_0 + \gamma \kappa_1\} 2^{s/2} (s-2)\eqsp.
\end{equation*}
\end{corollary}
\begin{proof}
We apply \Cref{lem:technical_lemma} with 
\[
a_k = \alpha^{1/2}2^{(s-k)/4}\eqsp, \quad b_k= \kappa_0 \beta^{2^{k-1}} 2^{k/4} 2^{s/4}\eqsp, \quad c_k = \kappa_1 \gamma^{2^{k-1}} 2^{s/2}\eqsp.
\]
Now for any $m \in \{1,\ldots,s-2\}$, 
\[
\prod_{i=1}^{m-1} a_i^{1/2^{i-1}} \leq \alpha^{1- 2^{-(m-1)}} 2^{\{s/2\}(1-2^{-(m-1)})}\eqsp,
\]
and with simple algebraic manipulations we get 
\begin{align*}
\sum_{\ell=1}^{s-2} \prod_{i=1}^{\ell-1} a_i^{1/2^{i-1}}  b_\ell^{1/2^{\ell-1}} + \sum_{\ell=1}^{m-1} \prod_{i=1}^{\ell-1} a_i^{1/2^{i-1}}  c_\ell^{1/2^{\ell-1}} 
&\leq \alpha (\beta \kappa_0 + \gamma \kappa_1) \sum_{\ell=1}^{s-2} 2^{\{s/2\} (1-2^{-(\ell-1)})} 2^{\ell/2^{\ell+1}} 2^{s/2^{\ell+1}} \\
&\leq \alpha (\beta \kappa_0 + \gamma \kappa_1) 2^{s/2} (s-2)\eqsp.
\end{align*}
\end{proof}

\begin{lemma}
\label{lem: pow_of_funct}
Let $h \in \Lclass_{\alpha, V}$ with $\alpha \in (0, 1/2]$ Then for any $i \in \mathbb{N}$ such that $2^i \alpha \leq 1$, it holds
\begin{equation*}
\Nnorm[2^i\alpha, V]{h^{2^i}} \leq (\Nnorm[\alpha, V]{h})^{2^i}2^{i} \eqsp.
\end{equation*}
\end{lemma}
\begin{proof}
Using the definition of $\Lclass_{\alpha, V}$, we get for any $r \leq i$ that
\begin{equation*}
\textstyle{
|h(z)|^{2^r} \leq (\Nnorm[\alpha, V]{h})^{2^r}  V^{2^r\alpha}(z)
}\eqsp.
\end{equation*}
Using Jensen's inequality, for $a, b > 0$ and $\beta \in (0, 1]$, we get $(a^\beta + b^\beta)/(a + b)^\beta \leq 2$. Thus,
\begin{equation}
\label{eq:technical_equat}
\frac{|h(z)^{2^r} + h(z')^{2^r}|}{\bar{V}^{2^r\alpha}(z,z')}\leq 2 \Nnorm[\alpha, V]{h}^{2^r}\eqsp.
\end{equation}
Inequality above implies that 
\begin{equation}
\label{eq:decomposition_of_lip_prod_norm}
\begin{split}
\underset{z,z'}{\sup}\frac{|h^{2^i}(z) - h^{2^i}(z')|}{c^{1/2}(z,z')\bar{V}^{2^i\alpha}(z,z')} 
&\leq \underset{z,z'}{\sup} \bigl\{\frac{|h(z) - h(z')|}{c^{1/2}(z,z')\bar{V}^{\alpha}(z,z')}\frac{|h(z) + h(z')|}{\bar{V}^{\alpha}(z,z')}\ldots \frac{|h^{2^{i-1}}(z) + h^{2^{i-1}}(z')|}{\bar{V}^{2^{i-1}\alpha}(z,z')}  \bigr\}
\\
&\leq 2^{i-1} \Nnorm[\alpha, V]{h}^{2^i}\eqsp.
\end{split}
\end{equation}
\end{proof}

\begin{lemma}
\label{lem: P^n h class}
Assume \Cref{ass:cost_fun} and \Cref{assum:wasserstein-convergence}. Then for any $\alpha \in (0, 1/2]$, $h \in \Lclass_{\alpha, V}$ and $n \in \nset$, it holds that $\MK^n h - \pi(h) \in \Lclass_{\alpha,V}$ with $\Nnorm[\alpha, V]{\MK^n h - \pi(h)} \leq 2 \varsigma \varkappa_{c}^{\alpha}\{\pi(V)\}^{\alpha}  \Nnorm[\alpha,V]{h} (1/4)^{\lfloor n/\tau\rfloor \alpha}$.
\end{lemma}
\begin{proof}
Using \Cref{assum:wasserstein-convergence}, we get
\begin{equation}
\label{eq:P^n_h_bound_lip}
|\MK^nh(z) - \MK^nh(z')| \leq  \Nnorm[\alpha,V]{h}\Wass{c^{1/2}\bar{V}^{\alpha}}(\delta_z\MK^n, \delta_{z'}\MK^n)\leq 2 \Nnorm[\alpha,V]{h} c^{1/2}\bar{V}^{\alpha} \varsigma \varkappa_{c}^{\alpha}(1/4)^{\lfloor n/\tau\rfloor \alpha}\eqsp.
\end{equation}
Moreover,
\begin{align*}
\textstyle   |\MK^nh(z) - \pi(h)| &\textstyle \leq  \Nnorm[\alpha,V]{h}\Wass{c^{1/2}\bar{V}^{\alpha}}(\delta_z\MK^n, \pi\MK^n) \leq \Nnorm[\alpha,V]{h} \varsigma \varkappa_c^{\alpha}(1/4)^{\lfloor n/\tau\rfloor \alpha}\int \{V^{\alpha}(z) + V^{\alpha}(z')\}\pi(\rmd z') \\ 
&\textstyle \leq \Nnorm[\alpha,V]{h} \varsigma \varkappa_c^{\alpha}(1/4)^{\lfloor n/\tau\rfloor \alpha}V^{\alpha}(z) \{\pi(V)\}^{\alpha}\eqsp,
\end{align*}
and the statement follows. 
\end{proof}

\begin{lemma}
\label{lem: product_of_two_funct}
Let $h_1 \in \Lclass_{\beta_1, \lyapW}$ and $h_2 \in \Lclass_{\beta_2, \lyapW}$ with $\beta_1, \beta_2 \in \rset_{+}$. Then,
\begin{equation*}
\Nnorm[\beta_1 + \beta_2, \lyapW]{h_1 h_2} \leq 2^{1+\beta_1 + \beta_2} \Nnorm[\beta_1, \lyapW]{h_1}  \Nnorm[\beta_2, \lyapW]{h_2}\eqsp.
\end{equation*}
\end{lemma}
\begin{proof}
The proof can be found in \cite[Lemma 28]{durmus2024probability}.
\end{proof}
\newpage
\section{Constants}
\label{sec:constants}
\begin{table}[!ht]
\begin{tabular}{l l l}
\hline
\bfseries Constant name & \bfseries Description & \bfseries Reference \\
\hline
$\bConst{\sf{Rm}, 1} = 60 \rme$, $\bConst{\sf{Rm}, 2} = 60$ & Constants from the Pinelis version \\
& of the Rosenthal inequality for martingales & \cite[Th.~4.1]{pinelis_1994} \\
\hline
$\ConstD_{\ref{lem:UGE_dyadic},1} = (16/3) (19/3)^{1/2} \bConst{\sf{Rm}, 1}$ & Constants in Rosenthal inequality & \Cref{lem:UGE_dyadic} \\
$\ConstD_{\ref{lem:UGE_dyadic},2} = (32/3) (\bConst{\sf{Rm}, 2}^{1/2} \bConst{\sf{Rm}, 1} ^2 + \bConst{\sf{Rm}, 2})$ & under \Cref{ass:UGE} & \\
\hline
$\ConstD_{\ref{lem:auxiliary_rosenthal},1} = (16/3) \bConst{\sf{Rm}, 1}, \,\, \ConstD_{\ref{lem:auxiliary_rosenthal},2} = 8 \bConst{\sf{Rm}, 2}$ & Simple Rosenthal under \Cref{ass:UGE} & \Cref{lem:auxiliary_rosenthal} \\
\hline
$\ConstD_{\ref{lem:recurrence-R-k},1} = (19/3)^{1/2}\bConst{\sf{Rm}, 1}^{1/2}$, & Recurrent bound $R_{k,s}$ & \Cref{lem:recurrence-R-k} \\
$\ConstD_{\ref{lem:recurrence-R-k},2} = 3\bConst{\sf{Rm}, 2}^{1/2}$ &  in terms of $R_{k+1,s}$ &  \\
\hline
 $\ConstD_{\ref{lem: auxiliary_rosenthal_VGE},1} =(8/3)\bConst{\sf{Rm}, 1}$,  $\ConstD_{\ref{lem: auxiliary_rosenthal_VGE},2} = (16/3)\bConst{\sf{Rm}, 2} + 11/3 $ & Simple Rosenthal under \Cref{ass:VGE} & \Cref{lem: auxiliary_rosenthal_VGE}\\
\hline
$\ConstD_{\ref{lem: rosenthal_VGE_for_2^s}, 1} = (8/3) (16/3)^{1/2} \bConst{\sf{Rm}, 1}^2\bigl(\bConst{\sf{Rm}, 2}^{1/2} + 1 \bigr) $ &  Constants in Rosenthal inequality & \Cref{lem: rosenthal_VGE_for_2^s}\\
${\ConstD_{\ref{lem: rosenthal_VGE_for_2^s}, 2} =(8\rme / 3) \bigl(\bConst{\sf{Rm}, 1}^2 \bConst{\sf{Rm}, 2}^{1/2} \sqrt{2} + \bConst{\sf{Rm}, 2}\bigr)}$ & under \Cref{ass:VGE} & \\
\hline
 $\ConstD_{\ref{lem: auxiliary_rosenthal_WGE}, 1} =  (8/3)\bConst{\sf{Rm}, 1}, \ConstD_{\ref{lem: auxiliary_rosenthal_WGE}, 2} = (16/3)\bConst{\sf{Rm}, 2} + 11/3$ & Simple Rosenthal under \Cref{ass:cost_fun} and \Cref{assum:wasserstein-convergence} & \Cref{lem: auxiliary_rosenthal_WGE}\\
\hline
$\ConstD_{\ref{lem: rosenthal_WGE_for_2^s}, 1} = (16/3)^{3/2} \bConst{\sf{Rm}, 1}^2\bigl( \sqrt{2} \bConst{\sf{Rm}, 2}^{1/2} + 1\bigr)$ &  Constants in Rosenthal inequality & \Cref{lem: rosenthal_WGE_for_2^s}\\
$\ConstD_{\ref{lem: rosenthal_WGE_for_2^s}, 2} = (8\rme/3) \bConst{\sf{Rm}, 2}^{1/2}(\bConst{\sf{Rm}, 1}^2 + \bConst{\sf{Rm}, 2}^{1/2})
$ & under \Cref{ass:cost_fun} and \Cref{assum:wasserstein-convergence}  & \\
\hline
\end{tabular}
\caption{Universal constants appearing in the main results}
\label{tab:univ_constants}
\end{table}
\end{appendix}
\bibliography{refs-1}       

\newpage
\section{Supplementary Material}
\label{sec:supplement}
\subsection{Checking assumptions on weighted Wasserstein ergodicity}
\label{sec:chech_wge}
In what follows, we provide a set of verifiable conditions that allows to verify \Cref{assum:wasserstein-convergence}. With the notations of Section 2.3 of the main paper, we make the following assumptions for the Markov kernel $\MK$:
\begin{assumptionC}
\label{assG:kernelP_q}
There exist a measurable function $V: \Zset \to [\rme;+\infty)$, $m \in \nsets$, $\lambda \in (0,1)$, $b \geq 0$, such that
\[
\textstyle{
\MK^m V(z) \leq \lambda V(z) + b
}\eqsp,
\]
Moreover, there exist $\lambda_{m,V} > 0$ and $b_{m,V} > 0$, such that, for $\ell \in \{1,\ldots,m-1\}$,
\[
\textstyle{
\MK^{\ell} V(z) \leq \lambda_{m,V} V(z) + b_{m,V}
}\eqsp.
\]
\end{assumptionC}
The assumption \Cref{assG:kernelP_q} is classical in the literature on Markov chains and is often referred to as the Foster-Lyapunov drift condition, see \cite[Chapter 19]{douc:moulines:priouret:soulier:2018}, \cite[Chapter 15-16]{meyn:tweedie}. Before proceeding with the next assumption, we need to state some additional definitions. We say that $\MKK$ is a kernel coupling of $\MK$ if for all $(z,z') \in \Zset^2$ and $\msa \in \Zsigma$, $\MKK((z, z'), \msa \times \Zset) = \MK (z, \msa)$ and $\MKK((z,z'), \Zset \times \msa) = \MK(z',\msa)$. If $\MKK$ is a kernel coupling of $\MK$, then for every $n \in \nset$, $\MKK^n$ is a kernel coupling of $\MK^n$ and for every $\nu \in \couplingmeasure(\xi,\xi')$, $\nu \MKK^n$ is a coupling of $(\xi \MK ^n,\xi'\MK^n)$, which implies that
\[
\Wass{\cost}(\xi \MK ^n,\xi' \MK ^n) \leq \int_{\Zset\times\Zset} \MKK^n\cost(z,z') \nu(\rmd z\rmd z')\eqsp.
\]
see \cite[Corollary 20.1.4]{douc:moulines:priouret:soulier:2018}. For any  probability measure $\zeta$ on $(\Zset^2,\Zsigma^{\otimes 2})$, we denote by $\PP_{\zeta}^\MKK$ and  $\PE_{\zeta}^\MKK$ the probability and  the
expectation on the canonical space $((\Zset^2)^\nset,(\Zsigma^{\otimes 2})^{\otimes \nset})$  such that the canonical process
$\{(Z_n, Z_n'), n \in \nset\}$ is a Markov chain with initial
probability $\mathcal{C}$ and Markov kernel $\MKK$. We write
$\PE_{z,z'}^{\MKK}$ instead of $\PE_{\delta_{z,z'}}^{\MKK}$.
Consider the following assumption.
\begin{assumptionC}
\label{assG:kernelP_q_contractingset_m}
There exist a kernel coupling $\MKK$ of $\MK$, $m \in \nsets, \minorwas \in (0,1)$, $\kappa_{\MKK} \geq 1$, so that
\begin{equation}
\label{eq:assG:kernelP_q_contractingset_m}
\textstyle{
\sup_{\ell \in \{1,\ldots,m-1\}} \MKK^{\ell} \metricc(z,z') \leq \kappa_{\MKK} \metricc(z,z')\eqsp, \qquad \MKK^m \metricc(z,z') \leq (1 - \minorwas \indi{\CKset}(z,z'))\metricc(z,z')
}\eqsp,
\end{equation}
with $\CKset=\{V \leq r\} \times \{V \leq r\}$, $\lambda$, $b$ in \Cref{assG:kernelP_q} and $r$ are given such that the condition $\bar{\lambda} =\lambda +2b/(r+1) < 1$ holds.
\end{assumptionC}
Without loss of generality, we consider the case where the constant $m$ in \Cref{assG:kernelP_q} and \Cref{assG:kernelP_q_contractingset_m} is the same. Define for $z,z'\in\Zset$, $\bar V(z,z') = \{V(z) + V(z')\}/2$, and $\bar r = (r+1)/2$. Consider the fixed point equation with unknown $\delta \geq 0$,
\begin{equation}
\label{eq: delta def}
(1-\minorwas) \bigl(\frac{\bar{\lambda} + b + \delta}{1+\delta}\bigr) = \frac{\bar{\lambda} \bar{r}+\delta}{\bar{r}+\delta}\eqsp.
\end{equation}
Since necessarily, $\bar{\lambda} + b \geq 1$, the left-hand side of this equation is  a decreasing function of $\delta$, while the right-hand side is an increasing function. Hence, \eqref{eq: delta def} has a unique positive root (denoted by $\rootwas$) if $(1-\minorwas)(\bar{\lambda} + b) > \bar{\lambda}$, and we define
\begin{equation}
\label{eq:delta_star_def}
\deltawas =
\begin{cases}
\rootwas & \text{ if } (1-\minorwas)(\bar{\lambda} + b) > \bar{\lambda}\eqsp, \\
0 & \text{otherwise}\eqsp.
\end{cases}
\end{equation}
The proof of the statement below is a slightly modified proof of the weak Harris theorem found, for example, in \cite[Section 4.7]{durmus2024probability}. The main feature we need to include in the proof is the $m$-step drift condition in \Cref{assG:kernelP_q} that does not necessarily imply that the corresponding Markov kernel contracts in one step. This is the case, for example, with the LSA example considered in the Section 4.

\begin{proposition}
\label{prop:wasser:geo}
Assume \Cref{assG:kernelP_q}, \Cref{assG:kernelP_q_contractingset_m} and \Cref{ass:cost_fun}. Then for any $(z,z')\in \Zset^2$, $\alpha \in (0,1/2]$, and $n\in \nset$, it holds
\begin{equation}
\label{eq: constraction}
\PE_{z,z'}^{\MKK}[\cost^{1/2}(Z_n, Z_n') \bar V^{\alpha}(Z_n, Z'_n)] \leq \kappa_{\MKK}^{1/2} c_{\MKK}^{\alpha}  \cost^{1/2}(z,z') \bar{V}^{\alpha}(z,z') \varrho^{\alpha n/m}  \eqsp,
\end{equation}
where $\bar{\lambda}$ is defined in \Cref{assG:kernelP_q_contractingset_m}, and
\begin{equation}
\label{eq:def:rho}
\textstyle{
\varrho = \frac{\bar{\lambda} \bar{r}+\delta_*}{\bar{r}+\delta_*} < 1 \eqsp,  \quad c_{\MKK} = (\lambda_{m,V} + b_{m,V} + \delta_*) / \varrho^{2m}
}\eqsp.
\end{equation}
\end{proposition}
\begin{proof}
Note that $2\alpha \leq 1$. Due to the definition of $\bar{V}$ the Markov kernel $\MKK^{m}$ satisfies the drift condition of the form
\begin{equation}
\label{eq:geometric-drift-condition_m}
\MKK^m \bar{V} \leq \bar{\lambda} \bar{V} + b \indi{\CKset} \eqsp,
\end{equation}
where $\CKset$ is defined in \Cref{assG:kernelP_q_contractingset_m}. For $\delta  \geq 0$, set $\bar{V}_{\delta}=\bar{V}+\delta$ and
\begin{equation}
\label{eq:rho_tilde_def}
\textstyle{
\tilde{\rho}_{\alpha, \delta} = \sup_{(z,z') \in \Zset^2} \bigl[(1-\minorwas \indi{\CKset}(z,z'))^{1/2} \lr{ \frac{\MKK^m \bar V^{2\alpha}_\delta(z,z')}{\bar{V}^{2\alpha}_\delta(z,z')}}^{1/2}\bigr]
}\eqsp,
\end{equation}
which is finite since $\MKK^m$ satisfies~\eqref{eq:geometric-drift-condition_m}. We will adjust parameter $\delta$ later during the proof. Furthermore, H\"older's inequality and \Cref{assG:kernelP_q_contractingset_m} yield
  \begin{align*}
    \MKK^m(\metricc^{1/2} \bar{V}_\delta^{\alpha})
    & \leq (\MKK^m\metricc)^{1/2} (\MKK^m\bar V^{2\alpha}_\delta)^{1/2} \\
    & \leq \lrb{(1-\epsilon \indi{\CKset})^{1/2} \lr{{\MKK^m\bar V^{2\alpha}_\delta}/{\bar{V}^{2\alpha}_\delta}}^{1/2} }
      \metricc^{1/2} \bar{V}^{\alpha}_\delta \leq \tilde{\rho}_{\alpha, \delta}  \metricc^{1/2} \bar{V}_\delta^{\alpha} \eqsp.
  \end{align*}
  Since $\bar{V} \leq \bar{V}_\delta$, we have for $k \geq 1$, that
\begin{equation}
\label{eq:wasser:geo:K_n}
\textstyle{
\MKK^{mk} (\metricc^{1/2} \bar V^{\alpha})\leq \MKK^{mk} (\metricc^{1/2} \bar V_\delta^{\alpha}) \leq \tilde{\rho}_{\alpha, \delta}^k \metricc^{1/2} \bar {V}_\delta^{\alpha}
}\eqsp.
\end{equation}
For $\ell \in \{0,\dots, m-1\}$, the drift condition \Cref{assG:kernelP_q} and Jensen's inequality imply
\begin{align*}
\MKK^\ell\bar V^{2\alpha}_\delta \leq (\MKK^\ell\bar V_\delta)^{2\alpha} \leq (\lambda_{m,V} \bar{V} + b_{m,V} + \delta)^{2\alpha} \leq (\lambda_{m,V} + b_{m,V} + \delta)^{2\alpha} \bar{V}^{2\alpha}\eqsp.
\end{align*}
Thus, with H\"older's and \Cref{assG:kernelP_q}, we get
\begin{equation*}
\MKK^\ell(\metricc^{1/2} \bar{V}_\delta^{\alpha})
    \leq (\MKK^\ell\metricc)^{1/2} (\MKK^\ell\bar V^{2\alpha}_\delta)^{1/2} \\
    \leq \boundmetric^{1/2} (\lambda_{m,V} + b_{m,V} + \delta)^{\alpha} \metricc^{1/2} \bar V^{\alpha} \eqsp.
\end{equation*}
This inequality and \eqref{eq:wasser:geo:K_n}   imply
\begin{equation}
  \label{eq:prop:wasser:geo_1}
      \MKK^{km + \ell}(\metricc^{1/2} \bar{V}_\delta^{\alpha}) \leq \tilde{\rho}_{\alpha, \delta}^k \MKK^\ell(\metricc^{1/2}\bar {V}_\delta^{\alpha}) \leq \tilde{\rho}_{\alpha, \delta}^k \boundmetric^{1/2} (\lambda_{m,V} + b_{m,V} + \delta)^{\alpha}  \metricc^{1/2} \bar V_{\delta}^{\alpha} \eqsp.
    \end{equation}
We now provide an upper bound on $ \tilde{\rho}_{\alpha, \delta}$.
Applying~\eqref{eq:geometric-drift-condition_m}, we obtain
\begin{equation}
\label{eq:wasser:geo:third}
(\MKK^m\bar V^{2\alpha}_\delta) /\bar V^{2\alpha}_\delta \leq \{\varphi(\bar{V})\}^{2\alpha} \indi{\CKset} + \{\psi(\bar{V})\}^{2\alpha} \indi{\CKset^c} \eqsp,
\end{equation}
with $\varphi(v) = (\bar{\lambda} v + b + \delta)/(v+\delta)$ and $ \psi(v)= (\bar{\lambda}\, v   + \delta)/(v+\delta)$. Since $\bar{V} \geq \indi{\CKset} +\bar{r}\indi{\CKset^c}$. The functions $\varphi$ and $\psi$ are decreasing on $[1;+\infty)$, we get
\begin{align*}
\textstyle{
(\MKK^m\bar V^{2\alpha}_\delta) / \bar V^{2\alpha}_\delta \leq \left\{\frac{\bar{\lambda} + b + \delta}{1+\delta} \right\}^{2\alpha} \indi{\CKset} + \left\{\frac{\bar{\lambda} \bar{r} + \delta}{\bar{r}+\delta} \right\}^{2\alpha}  \indi{\CKset^c}
}\eqsp.
\end{align*}
The previous inequality yields
\begin{equation*}
\textstyle{
(1-\minorwas \indi{\CKset})^{1/2} \lr{{\MKK^m\bar{V}^{2\alpha}_\delta}/{\bar{V}^{2\alpha}_\delta}}^{1/2} \leq \lrb{(1-\minorwas)^{1/2} \bigl(\frac{\bar{\lambda} + b + \delta}{1+\delta}\bigr)^{\alpha}}\indi{\CKset} + \lr{\frac{\bar{\lambda} \bar{r} + \delta}{\bar{r}+\delta}}^{\alpha} \indi{\CKset^c}
}\eqsp.
\end{equation*}
showing that
$\tilde{\rho}_{\alpha,\delta} \leq \rho_{\delta}^{2\alpha}$, where
\begin{equation}
\label{eq:def:rho_0}
\textstyle{
\rho_{\delta}=\lrb{(1-\minorwas)^{1/2} \bigl(\frac{\bar{\lambda} + b + \delta}{1+\delta}\bigr)^{1/2}} \vee \lr{\frac{\bar{\lambda} \bar{r}+\delta}{\bar{r}+\delta}}^{1/2}
}\eqsp.
\end{equation}
Now we choose $\delta = \deltawas$ defined in~\eqref{eq:delta_star_def}, and complete the proof setting $\ratewas = \rho_{\deltawas}$ in \eqref{eq:def:rho_0}.
\end{proof}

\begin{corollary}
\label{cor:wasserstein}
Under the assumptions of \Cref{prop:wasser:geo}, assumption \Cref{assum:wasserstein-convergence} holds with parameters
\[
\textstyle{
\tmix = m \log{4} / \log{(1/\varrho)}\eqsp, \quad \varkappa_{c} = c_{\MKK}\eqsp, \quad \varsigma = \kappa_{\MKK}^{1/2}
}\eqsp.
\]
\end{corollary}
\begin{proof}
The proof follows from \Cref{prop:wasser:geo} and \cite[Corollary~20.1.4]{douc:moulines:priouret:soulier:2018}.
\end{proof}

\subsection{Postponed proof of Section 4.1 (Linear Stochastic Approximation case)}
\label{sec:lsa_proofs}
For clarity, we restate that, for a fixed step size $\gamma > 0$, the LSA algorithm generates the sequence ${\theta_k}_{k \in \nset}$ defined by:
\begin{equation}
\label{eq:lsa_supp}
\textstyle \theta_{k} = \theta_{k-1} - \gamma \{ \funcA{Z_k} \theta_{k-1} - \funcb{Z_k} \} \eqsp,~~ k \geq 1 \eqsp.
\end{equation}

The fact that $-\bA$ is Hurwitz implies the following proposition:
\begin{proposition}[\protect{\cite[Proposition 1]{durmus2021tight}}]
\label{fact:Hurwitzstability_supp}
Let $-\bA$ be a Hurwitz matrix. Then there exists a unique symmetric positive definite matrix $\mathbf{P}$ satisfying the equation
$\bA^\top \mathbf{P} + \mathbf{P} \bA =  \Id$. In addition, setting
\begin{equation}
\label{eq: kappa_def_supp}
a = \normop{\mathbf{P}}^{-1}/2\eqsp, \quad
\text{and} \quad \gamma_\infty = (1/2) \normop{\bA}[\mathbf{P}]^{-2} \normop{\mathbf{P}}^{-1} \wedge \normop{\mathbf{P}} \eqsp,
\end{equation}
it holds for any $\gamma \in [0, \gamma_{\infty}]$ that $\normop{\Id - \gamma \bA}[\mathbf{P}]^2 \leq 1 - a \gamma$, and $\gamma a \leq 1/2$.
\end{proposition}

Recall the definitions of the constants from the main paper:
\begin{align}
\label{eq:def_qcond_b_Q_supp}
&\textstyle{\qcond = \lambda_{\sf max}( \mathbf{P} )/\lambda_{\sf min}( \mathbf{P} )  \eqsp, \quad  b_{\mathbf{P}} = 2 \sqrt{\qcond} \bConst{A} \eqsp,} \quad \\
\label{eq:def_alpha_p_infty_supp}
&\textstyle{\gamma_{q,\infty} = \gamma_{\infty} \wedge \smallAconst/q \eqsp, \quad \smallAconst = a/\{2b_\mathbf{P}^2\}}\eqsp.
\end{align}

\subsubsection{Proof of Proposition 4.2}
Note that \eqref{eq:lsa_supp} may be rewritten as 
\begin{equation}
\label{eq:clever_lsa_step_supp}
\textstyle{
\theta_{n} - \thetas = (\Id - \gamma \funcA{Z_n})(\theta_{n-1} - \thetas) - \gamma \funcnoise{\State_{n}}
}\eqsp.
\end{equation}
Expanding the recurrence \eqref{eq:clever_lsa_step_supp}, we get 
\begin{align}
\label{eq:LSA_recursion_expanded}
\theta_{n} - \thetas = \textstyle \ProdBa_{1:n} \{ \theta_0 - \thetas \} - \gamma \sum_{j=1}^n \ProdBa_{j+1:n} \funcnoise{Z_j}\eqsp,
\end{align}
where $\ProdBa_{m:n}$, $m \leq n$, is the product of random matrices given by the formula 
\begin{equation} 
\label{eq:definition-Phi} \textstyle
\ProdBa_{m:n}  = \prod_{i=m}^n (\Id - \gamma \funcA{Z_i} ) \eqsp, \quad m,n \in\nsets, \quad m \leq n \eqsp.
\end{equation}
A key property from which we derive our bounds is an exponential stability result for the $p$-th moment of the product of $\ProdBa_{1:n}$,
which was established in \citep[Corollary 1]{durmus2021tight}.
\begin{theorem}\label{fact:exponential-stability-product}
Assume \Cref{assum:noise-level} and \Cref{assum:A-b}. Then, for any  $p,q \in \nset$, $2 \leq p \leq q$, $\gamma \in (0,\gamma_{q,\infty}]$ and $n \in \nset$, it holds
\begin{equation}
  \label{eq:concentration iid}
 \PE^{1/p}\left[ \normop{\ProdBa_{1:n}}^{p} \right]
\leq \sqrt{\qcond} d^{1/q} (1 - a \gamma + (q-1)  b_{P}^2 \gamma^2)^{n/2}\eqsp.
\end{equation}
\end{theorem}
The condition connecting the choice of the step size $\gamma$ in the iterations \eqref{eq:lsa_supp} with the orders $p$ and $q$ is unavoidable; see \cite[Example~1]{durmus2021tight}. This example also shows that for the fixed step size $\gamma$ one cannot expect subgaussian or subexponential concentration bounds for $\norm{\theta_n - \thetas}$, as well as for $\norm{\bar{\theta}_n - \thetas}$. Using the above bounds, one can specify moment bounds on the error of the last LSA iterate $\theta_n - \thetas$:
\begin{proposition}
[\protect{\cite[Proposition 1]{durmus2022finite}}]
\label{cor:LSA_err_bound_iid}
Assume \Cref{assum:noise-level} and \Cref{assum:A-b}. Then, for any  $p,q \in\nset$, $2 \leq p \leq q$, $\gamma \in (0, \gamma_{q,\infty}]$, $n \in \nset$, and $\theta_0 \in \rset^d$ it holds
\begin{align}
\label{eq:n_step_norm_propery}
\PE^{1/p}\left[\norm{\theta_n - \thetas}^{p}\right]
\leq d^{1/q} \qcond^{1/2} \left(1 - \gamma a/4\right)^{n} \norm{\theta_0 - \thetas} + d^{1/q} \ConstD \sqrt{\gamma a  p} \supconsteps \eqsp,
\end{align}
where
\begin{equation}
\label{eq:const_D_def}
\ConstD = (2 \qcond)^{1/2} a^{-1}(1 + 4 \qcond^{1/2} \bConst{A}  a^{-1})
\eqsp.
\end{equation}
\end{proposition}

For $p \geq 2$, define the auxiliary constants:
\begin{equation}
\label{eq:drift_constants_iid}
\bconst_{\gamma,p} = 1 + 2 \rme (\gamma a p)^{p/2} \ConstD^{p} \supconsteps^{p}\eqsp,
\end{equation}
\begin{equation}
\label{eq:drift_constants_iid_main}
m(\gamma,p) = \frac{2(1+\log(8 \qcond))}{\gamma a} \vee \frac{4p(1 + \log{4} + \log(\gamma a p)/2 + \log{\ConstD} + \log{\supconsteps})}{\gamma a}\eqsp,
\end{equation}
where $\ConstD$ is given in \eqref{eq:const_D_def}.
\begin{lemma}
\label{cor:drift_V_2_lemma}
Assume \Cref{assum:noise-level} and \Cref{assum:A-b}. Let $p \geq 2$. Then for any $\gamma \in (0, \gamma_{p(1+\log{d}),\infty}]$, and $m \geq m_0(\gamma,p)$ it holds that
\begin{equation}
\label{eq:m_step_drift_W_2}
\MK_{\gamma}^{m} V_p(\theta) \leq (1/2)^{p/2}V_p(\theta) + \bconst_{\gamma,p}\eqsp.
\end{equation}
Moreover, we have for any $\ell \in \{1,\ldots,m-1\}$, that 
\begin{equation}
\label{eq:m_step_drift_W_p}
\MK_{\gamma}^{\ell} V_p(\theta) \leq \rme 2^{p} \qcond^{p/2} (1-\gamma a/4)^{p} V_p(\theta) + \bconst_{\gamma,p}\eqsp.
\end{equation}
\end{lemma}
\begin{proof}
Using \Cref{cor:LSA_err_bound_iid} with $q = p(1+\log{d})$, we get for any $m \in \nset$, $p \geq 2$, and $\gamma \in (0,\gamma_{p(1+\log{d}),\infty}]$, that
\begin{equation}
\label{eq:p_moment_bound_drift}
\PE\left[\norm{\theta_m - \thetas}^{p}\right] \leq   2^p d^{1/(1+\log{d})}  \qcond^{p/2} \left(1 - \gamma a/4\right)^{mp} \norm{\theta_0 - \thetas}^{p} + 2^{p} d^{1/(1+\log{d})}\,\ConstD^{p}\,(\gamma a p)^{p/2} \supconsteps^{p}\eqsp.
\end{equation}
Note that for $d \geq 1$, $d^{1/(1+\log{d})} \leq \rme$, and by definition of  $m_0(\gamma,q)$, it holds that
\[
4 \rme \qcond (1- \gamma a/4)^{2m_0(\gamma,q)}\leq 4 \rme \qcond \rme^{-\gamma a m_0(\gamma,q)/2} \leq 1/2\eqsp.
\]
showing \eqref{eq:m_step_drift_W_2}. The proof of \eqref{eq:m_step_drift_W_p} is along the same lines.
\end{proof}

Now, following \cite{durmus2021tight}, we consider the synchronous coupling construction defined by the recursions
\begin{equation}
\label{eq:sync_coupling_construction}
\begin{split}
\theta_{m} &= \left(\Id - \gamma \Am(Z_m)\right)\theta_{m-1} + \gamma \bm(Z_m) \eqsp, \quad \theta_{0} = \theta\eqsp, \\
\theta_{m}^{\prime} &= \left(\Id - \gamma \Am(Z_m)\right)\theta^{\prime}_{m-1} + \gamma \bm(Z_m) \eqsp, \quad \theta^{\prime}_{0} = \theta^{\prime}\eqsp.
\end{split}
\end{equation}
The pair $\{(\theta_m,\theta^\prime_m)\}_{m=0}^\infty$ defines a Markov chain. We denote by $\MKK_{\gamma}$ the associated kernel, which is a kernel coupling of $\MK_{\gamma}$.
\begin{lemma}
\label{lem:contraction_wasserstein_LSA}
Assume \Cref{assum:A-b} and let $p \geq 2$. Then for any $\gamma \in (0,\gamma_{p(1+\log{d}),\infty}]$, $\theta,\theta^\prime \in \rset^{d}$, and $m \geq m_{0}(\gamma,p)$, it holds that
\begin{equation}
\label{eq:assG:kernelP_q_contractingset_m_0}
\PE_{\theta_0,\theta_0^{\prime}}^{\MKK_{\gamma}}[\norm{\theta_m - \theta_m^{\prime}}^2] \leq (1/2) \norm{\theta_0 - \theta_0^\prime}^{2}\eqsp.
\end{equation}
\end{lemma}
\begin{proof}
By induction, we obtain from \eqref{eq:lsa} and \eqref{eq:LSA_recursion_expanded} the following decomposition:
\begin{align}
\theta_{n} - \thetas = \textstyle \ProdBa_{1:n} \{ \theta_0 - \thetas \} - \gamma \sum_{j=1}^n \ProdBa_{j+1:n} \funcnoise{Z_j}\eqsp.
\end{align}
Using that $\gamma \in (0, \gamma_{q,\infty}]$ and applying \Cref{cor:LSA_err_bound_iid}, we obtain
\begin{align}
\PE_{\theta_0,\theta_0^{\prime}}^{\MKK_{\gamma}}[\norm{\theta_m - \theta_m^{\prime}}^2] &= \PE\bigl[\norm{\ProdBa_{1:m}(\theta_0 - \theta_0^{\prime})}^{2}\bigr] \leq \PE\bigl[\norm{\ProdBa_{1:m}}^{2}\bigr] \norm{\theta_0 - \theta^{\prime}_0}^{2} \nonumber \\
&\leq \qcond \rme (1 - \gamma a/4)^{2m} \norm{\theta_0 - \theta^{\prime}_0}^{2} \label{eq:K_m_bound_c}\eqsp.
\end{align}
Now it remains to use that $m \geq m_0(\gamma,q)$ to ensure that $\qcond \rme (1 - \gamma a/4)^{2m} \leq 1/2$.
\end{proof}

We now have all the tools to prove Proposition 4.2 from the main paper. We restate it below with explicit constants.

\begin{proposition}[Proposition 4.2 restated]
\label{th:explicit_constants_under_A1_more_specific_supp}
Assume \Cref{assum:noise-level} and \Cref{assum:A-b}. Let $p \geq 2$. Then, for $\gamma \in (0,\gamma_{p(1+\log d),\infty}]$, the Markov kernel $\MK_{\gamma}$ defined in (20) in the main document, has a unique invariant distribution $\pi_{\gamma}$ and satisfies \Cref{ass:cost_fun} and \Cref{assum:wasserstein-convergence} with the drift function $V_p$ defined in (21), and
\begin{equation*}
\begin{split}
\cost(\theta,\theta') &= 1 \wedge \norm{\theta - \theta'}^2\eqsp, \quad  \varsigma = 1 \wedge \rme \qcond (1-\gamma a/4)\eqsp, \quad \varkappa_{c} = (\rme 2^{p} \qcond^{p/2} (1-\gamma a/4)^{p} + 3 \bconst_{\gamma,p}) / \varrho^{2m} \eqsp, \\
\varrho &\leq \frac{(3/8)(4\bconst_{\gamma,p} + 1) + 2\bconst_{\gamma,p}}{(4\bconst_{\gamma,p} + 1)/2 + 2\bconst_{\gamma,p}} \leq 7/8 \eqsp, \quad \tmix = m \log 4 / \log(8/7)\eqsp,
\end{split}
\end{equation*}
and block size $m = m(\gamma,p)$ defined in \eqref{eq:drift_constants_iid_main}.
\end{proposition}

\begin{proof}
To complete the proof we aim to apply the result of \Cref{cor:wasserstein}. In order to do this, we need to check that the assumptions \Cref{assG:kernelP_q}, \Cref{assG:kernelP_q_contractingset_m} and \Cref{ass:cost_fun} hold. Our underlying space
$\Zset = \rset^{d}$, which is equipped with the Euclidean distance $\metricz(\theta,\theta') = \norm{\theta-\theta'}, \theta,\theta' \in \rset^{d}$, and is a complete separable space. In addition, we set $\cost(\theta,\theta') = 1 \wedge \norm{\theta-\theta'}^2$, therefore, \Cref{ass:cost_fun} holds. We now check \Cref{assG:kernelP_q} for the kernel $\MK_{\gamma}$ and $m = m_0(\gamma,p)$ (see definition in Section 5.8 of the main paper). It is directly established by \Cref{cor:drift_V_2_lemma} that \Cref{assG:kernelP_q} holds with 
\[
V(z) = V_p(z)\eqsp, \quad \lambda = 1/2\eqsp, \quad b = \bconst_{\gamma,p}\eqsp, \quad \lambda_{m,V_p} =  \rme 2^{p} \qcond^{p/2} (1-\gamma a/4)^{p}\eqsp, \quad b_{m,V_p} = \bconst_{\gamma,p}\eqsp.
\]
Similarly, \Cref{lem:contraction_wasserstein_LSA} (see \eqref{eq:K_m_bound_c}) directly yields that for any $\ell \in \{1,\ldots,m-1\}$, it holds
\[
\MKK_{\gamma}^\ell \cost(\theta,\theta') = \MKK_{\gamma}^\ell (1 \wedge \norm{\theta-\theta'}^2) \leq 1 \wedge \MKK_{\gamma}^\ell (\norm{\theta-\theta'}^2) \leq (1 \wedge \rme \qcond (1-\gamma a/4)^2) c(\theta,\theta')\eqsp.
\]
Now, in order to ensure \Cref{assG:kernelP_q_contractingset_m}, we note that \Cref{lem:contraction_wasserstein_LSA} implies 
\[
\MKK_{\gamma}^\ell \cost(\theta,\theta') \leq 1 \wedge \rme \qcond (1-\gamma a/4)^{2m}\norm{\theta-\theta'}^2\eqsp.
\]
Thus, in order to ensure the second part of \Cref{assG:kernelP_q_contractingset_m}, we need to guarantee, that 
$$
\qcond (1-\gamma a/4)^{2m} \leq 1/2\eqsp, \text{ and }
\rme \qcond (1-\gamma a/4)^{2m}\norm{\theta-\theta'}^2 \leq 1/2\eqsp.
$$
To check the second part, note that for $p \geq 2$, 
\begin{align*}
\rme \qcond (1-\gamma a/4)^{2m}\norm{\theta-\theta'}^2 
&\leq 2\rme \qcond \rme^{-\gamma a m/2}(\norm{\theta-\thetas}^2 + \norm{\theta'-\thetas}^2) \\
&\leq 2\rme \qcond \rme^{-\gamma a m/2}(V_p(\theta) + V_p(\theta'))\eqsp.
\end{align*}
Now we assume that $V_p(\theta), V_p(\theta') \leq r$, where $r = \rme^{\gamma a m/4}$. Then it remains to ensure that 
\[
4 \rme \qcond \rme^{-\gamma a m / 4} \leq 1/2\eqsp, \text{ and } \rme^{\gamma a m/4} \geq 4 \bconst_{\gamma,p}\eqsp.
\]
Both inequalities follows directly from the definition of block size $m$. Thus, \Cref{assG:kernelP_q_contractingset_m} holds with $\varkappa_{\MKK_{\gamma}} = 1 \wedge \rme \qcond (1-\gamma a/4)^{2}$, $r = 4\bconst_{\gamma,p}$, $\varepsilon = 1/2$, and $m$ given in Section 5.8 of the main paper. Now we directly apply \Cref{cor:wasserstein}.
\end{proof} 

\subsubsection{Proof of Lemma 4.3}
\begin{proof}
    Note that 
    \begin{equation}
    \bA\Sigma_{\pi_{\gamma}}\bA^T = \lim_{n \to \infty} n \PE_{\pi_{\gamma}}\bigl[\bA (\prtheta_n - \thetas)(\prtheta_n - \thetas)^T\bA^T\bigr]\eqsp.
    \end{equation}
    Following the \cite[Proposition 3]{durmus2022finite}, we write 
    \[
    \sqrt{n}\bA(\prtheta_n - \thetas) = W_1 + W_2 + W_3\eqsp,
    \]
    where 
    \begin{align}
    W_1 &= -\frac{1}{\sqrt{n}}\sum_{k=0}^{n-1}\funcnoise{\State_{k+1}}\eqsp, \\
    W_2 &= \frac{\theta_0-\theta_n}{\gamma\sqrt{n}}\eqsp, \\
    W_3 &= -\frac{1}{\sqrt{n}}\sum_{k = 0}^{n-1} \zmfuncA{\State_{k+1}} (\theta_k - \thetas)\eqsp,
    \end{align}
    and 
    \[
    \zmfuncA{\State_{k+1}} = \funcA{\State_{k+1}}-\bA\eqsp.
    \]
    We estimate each term separately.
    From the definition of $\Sigma_{\varepsilon}$ and $\funcnoise{z}$, 
    \begin{equation*}
        \PE_{\pi_\gamma}[W_1W_1^T] = \Sigma_{\varepsilon}.
    \end{equation*}   
    From H\"{o}lder's inequality and \Cref{cor:LSA_err_bound_iid}, we get
    \begin{equation*}
        \norm{\PE_{\pi_\gamma}[W_2W_2^T]} \leq \frac{4}{\gamma^2n}\PE_{\pi_\gamma}[\norm{\theta_0-\thetas}^2] \leq \frac{8d \ConstD^2 a \supconsteps^2}{\gamma n} \to 0\eqsp, \quad n \to \infty\eqsp,
    \end{equation*}
    where $\ConstD$ is defined in \eqref{eq:const_D_def}. For any $i \neq j$, since $Z_{i+1}$ and $Z_{j+1}$ are independent, we get   
    \[
    \PE_{\pi_\gamma}[\zmfuncA{\State_{i+1}}(\theta_i - \thetas)(\theta_j - \thetas)^T \{\zmfuncA{\State_{j+1}}\}^T] = 0\eqsp.
    \]
    For $i = j$, taking the conditional expectation with respect to $\State_{i}$, we obtain that 
    \begin{equation}
        \PE_{\pi_\gamma}[W_3W_3^T] = \PE[\zmfuncA{\State_{1}} \covgammatheta \{\zmfuncA{\State_{1}}\}^T] = \mathcal{O}(\gamma),
    \end{equation}
    where $\covgammatheta$ the marginal covariance matrix of $\theta_0$ under $\pi_{\gamma}$ defined below in \eqref{eq:last_iter_cov_matrix}. We now consider the term $W_1W_2^T$, the same bounds will hold for $W_2W_1^T$. Using  H\"{o}lder's inequality and \Cref{cor:LSA_err_bound_iid}, we obtain
    \begin{equation*}
        \norm{\PE_{\pi_\gamma}[W_1W_2^T]} \leq \PE^{1/2}_{\pi_\gamma}[\norm{W_1}^2]\PE^{1/2}_{\pi_\gamma}[\norm{W_2}^2] \leq \frac{2}{\gamma n}\sqrt{n \trace(\Sigma_{\varepsilon})}d^{1/2} \ConstD \sqrt{2\gamma a}\supconsteps \underset{n \rightarrow +\infty}{\rightarrow} 0.
    \end{equation*}
    Now consider the term $W_3W_2^T$. Using  H\"{o}lder's inequality and \Cref{cor:LSA_err_bound_iid} again, we get
     \begin{equation*}
         \norm{\PE_{\pi_\gamma}[W_3W_2^T]} \leq \PE^{1/2}_{\pi_\gamma}[\norm{W_3}^2]\PE^{1/2}_{\pi_\gamma}[\norm{W_2}^2] \leq \frac{2}{\gamma n}\sqrt{n \trace(\PE_{\pi_\gamma[W_3W_3^T]})}d^{1/2} \ConstD \sqrt{2\gamma a}\supconsteps \underset{n \rightarrow +\infty}{\rightarrow} 0\eqsp.
     \end{equation*}
     
     Now it remains to estimate the term $W_1W_3^T$. Using the fact that $\PE_{\pi_\gamma}[\theta_0] = \thetas$ and the independence of $Z_i$ for different $i$, we obtain that
     \[
     \PE_{\pi_\gamma}[W_1W_3^T] = \PE_{\pi_\gamma}[W_3W_1^T] = 0\eqsp.
     \]
     Combining all the inequalities above and taking the limit $n \rightarrow +\infty$ we complete the proof.
\end{proof}

For notation simplicity, denote by $\covgammatheta$ the marginal covariance matrix of $\theta_k$ under $\pi_{\gamma}$, that is, 
\begin{equation}
\label{eq:last_iter_cov_matrix}
\covgammatheta = \PE_{\pi_{\gamma}}[(\theta_0 - \thetas)(\theta_0 - \thetas)^T]\eqsp.
\end{equation}

\begin{lemma}
\label{lem:sigma_last}
Assume \Cref{assum:noise-level} and \Cref{assum:A-b}. Then for the step size 
\[
\gamma \in (0, \sigma_{\min}(\Id \otimes \bA + \bA \otimes \Id)/\bConst{A}^2)\eqsp,
\]
it holds that
\begin{equation}
\label{eq:vectorization_identity_sigma_bar}
\vec{\covgammatheta} = \gamma (\Id \otimes\bA + \bA\otimes \Id - \gamma \mathbf{U})^{-1} \vec{\Sigma_{\epsilon}}\eqsp.
\end{equation}
where the matrix $\mathbf{U} \in \rset^{d^2 \times d^2}$ is given by 
\begin{equation}
\label{eq:U_matr_def}
\mathbf{U} = \PE[\funcA{\State_1}\otimes\funcA{\State_1}]\eqsp.
\end{equation}
\end{lemma}
\begin{proof}
From the recurrence \eqref{eq:lsa}, we obtain that 
\[
\theta_1 - \thetas = (\Id - \gamma \funcA{\State_1})(\theta_0 - \thetas) - \gamma \funcnoise{\State_1}\eqsp. 
\]
taking second moment w.r.t. $\pi_{\gamma}$ from both sides, and using the fact that $\State_1$ is independent from $\theta_0$, we get 
\[
\covgammatheta = \PE[(I-\gamma \funcA{\State_1}) \covgammatheta (I-\gamma \funcA{\State_1})^T] + \gamma^2 \Sigma_{\varepsilon}\eqsp. 
\]
Expanding the brackets, we get 
\[
\bA \covgammatheta + \covgammatheta \bA^{\top} - \gamma \PE[\funcA{\State_1} \covgammatheta  \{\funcA{\State_1}\}^{\top}] = \gamma \Sigma_{\varepsilon}\eqsp.
\]
Applying vectorization operator to both sides of the equality, and setting 
\[
\mathbf{U} = \PE[\funcA{\State_1}\otimes\funcA{\State_1}]\eqsp,
\]
it remains to check that the matrix 
\[
\Id \otimes \bA + \bA \otimes \Id - \gamma \mathbf{U}
\]
is non-degenerate. Set 
\begin{equation}
\label{eq:S_R_def}
S = \Id \otimes \bA + \bA \otimes \Id \in \rset^{d^2 \times d^2}\eqsp.
\end{equation}
Then it is easy to observe that 
\begin{align*}
(S - \gamma \mathbf{U})^{-1} = S^{-1} + S^{-1} \sum_{k=1}^{\infty} \gamma^k (\mathbf{U} S^{-1})^{k}\eqsp,
\end{align*}
provided that $\gamma \norm{\mathbf{U} S^{-1}} < 1$. Since $\norm{\mathbf{C} \otimes \mathbf{D}} = \norm{\mathbf{C}} \norm{\mathbf{D}}$ for any matrices $\mathbf{C}$ and $\mathbf{D}$, 
\[
\norm{\mathbf{U}} \leq \bConst{A}^2\eqsp,
\]
and we get the identity \eqref{eq:vectorization_identity_sigma_bar}. 
\end{proof}

\subsection{Bernstein-type inequalities}
\label{sec:Bernstein_type_inequality}
Combining Lemma A.2 from the main paper with Theorems 2.4 and 2.7, we obtain the following Bernstein-type bounds:
\begin{corollary}
\label{cor:VGE_bound}
Assume \Cref{ass:VGE}. For any $\delta \in (0,1/\rme^2)$ it holds with $\PP_{\pi}$-probability at least $1-\delta$ that
\begin{equation*}
\begin{split}
|\Sstat_n| &\leq \bConst{\sf{Rm}, 1} \sqrt{2}\sqrt{n \ln(1/\delta))}\sigma_{\pi}(f) \\
&\qquad+ \bigl\{4\ConstD_{5.7,1} n^{1/4}\tmix^{3/4} + 4\ConstD_{5.7,2} \tmix \log ^{\beta+1}n\bigr\}\hat{\ln}(1/\delta)(2\beta/\rme)^{\beta}\varkappa\{\pi(V)\} \Vnorm[\log^{\beta}{V}]{f}\eqsp,
\end{split}
\end{equation*}
where we set $\hat{\ln}(s) = 2\ln(s)^{2+ \beta} \log_2(2 \ln(s))$.
\end{corollary}

\begin{corollary}
\label{cor:WGE_bound}
Assume  \Cref{ass:cost_fun} and \Cref{assum:wasserstein-convergence}. For any $\delta \in (0,1/\rme^2)$ it holds with $\PP_{\pi}$-probability at least $1-\delta$ that
\begin{multline*}
| \Sstat_n| \leq \bConst{\sf{Rm}, 1} \sqrt{2}\sqrt{n \ln(1/\delta))}\sigma_{\pi}(f) \\
+ \bigl\{4\ConstD_{5.10,1} n^{1/4}\tmix^{3/4}   + 4\ConstD_{5.10,2}\tmix \log ^{\beta+1}n\bigr\}\hat{\ln}(1/\delta)(2\beta/\rme)^{\beta}\varkappa_c\{\pi(V)\}\eqsp,
\end{multline*}
where $\hat{\ln}(s)$ is defined in \Cref{cor:VGE_bound}.
\end{corollary}

\end{document}